\documentclass[11pt,reqno]{amsart}

\usepackage{amsmath,amsfonts,amsthm,amssymb}
\usepackage{graphicx}
\usepackage{verbatim}
\usepackage{color}
\usepackage{amscd}
\usepackage{verbatim}
\usepackage{mathrsfs}
\usepackage[colorlinks,citecolor=red,pagebackref,hypertexnames=false, breaklinks]{hyperref}

\begin{document}
\newcommand{\M}{{\mathcal M}}
\newcommand{\loc}{{\mathrm{loc}}}
\newcommand{\core}{C_0^{\infty}(\Omega)}
\newcommand{\sob}{W^{1,p}(\Omega)}
\newcommand{\sobloc}{W^{1,p}_{\mathrm{loc}}(\Omega)}
\newcommand{\merhav}{{\mathcal D}^{1,p}}
\newcommand{\be}{\begin{equation}}
\newcommand{\ee}{\end{equation}}
\newcommand{\mysection}[1]{\section{#1}\setcounter{equation}{0}}
\newcommand{\laplace}{\Delta}
\newcommand{\pl}{\laplace_p}
\newcommand{\grad}{\nabla}
\newcommand{\pd}{\partial}
\newcommand{\bo}{\pd}
\newcommand{\csub}{\subset \subset}
\newcommand{\sm}{\setminus}
\newcommand{\ssm}{:}
\newcommand{\diver}{\mathrm{div}\,}
\newcommand{\bea}{\begin{eqnarray}}
\newcommand{\eea}{\end{eqnarray}}
\newcommand{\bean}{\begin{eqnarray*}}
\newcommand{\eean}{\end{eqnarray*}}
\newcommand{\thkl}{\rule[-.5mm]{.3mm}{3mm}}
\newcommand{\cw}{\stackrel{\rightharpoonup}{\rightharpoonup}}
\newcommand{\id}{\operatorname{id}}
\newcommand{\supp}{\operatorname{supp}}
\newcommand{\wlim}{\mbox{ w-lim }}
\newcommand{\mymu}{{x_N^{-p_*}}}
\newcommand{\R}{{\mathbb R}}
\newcommand{\N}{{\mathbb N}}
\newcommand{\Z}{{\mathbb Z}}
\newcommand{\Q}{{\mathbb Q}}
\newcommand{\abs}[1]{\lvert#1\rvert}
\newtheorem{theorem}{Theorem}[section]
\newtheorem{corollary}[theorem]{Corollary}
\newtheorem{lemma}[theorem]{Lemma}
\newtheorem{notation}[theorem]{Notation}
\newtheorem{definition}[theorem]{Definition}
\newtheorem{remark}[theorem]{Remark}
\newtheorem{proposition}[theorem]{Proposition}
\newtheorem{assertion}[theorem]{Assertion}
\newtheorem{problem}[theorem]{Problem}
\newtheorem{conjecture}[theorem]{Conjecture}
\newtheorem{question}[theorem]{Question}
\newtheorem{example}[theorem]{Example}
\newtheorem{Thm}[theorem]{Theorem}
\newtheorem{Lem}[theorem]{Lemma}
\newtheorem{Pro}[theorem]{Proposition}
\newtheorem{Def}[theorem]{Definition}
\newtheorem{Exa}[theorem]{Example}
\newtheorem{Exs}[theorem]{Examples}
\newtheorem{Rems}[theorem]{Remarks}
\newtheorem{Rem}[theorem]{Remark}

\newtheorem{Cor}[theorem]{Corollary}
\newtheorem{Conj}[theorem]{Conjecture}
\newtheorem{Prob}[theorem]{Problem}
\newtheorem{Ques}[theorem]{Question}
\newtheorem*{corollary*}{Corollary}
\newtheorem*{theorem*}{Theorem}
\newcommand{\pf}{\noindent \mbox{{\bf Proof}: }}


\renewcommand{\theequation}{\thesection.\arabic{equation}}
\catcode`@=11 \@addtoreset{equation}{section} \catcode`@=12
\newcommand{\Real}{\mathbb{R}}
\newcommand{\real}{\mathbb{R}}
\newcommand{\Nat}{\mathbb{N}}
\newcommand{\ZZ}{\mathbb{Z}}
\newcommand{\CC}{\mathbb{C}}
\newcommand{\Pess}{\opname{Pess}}
\newcommand{\Proof}{\mbox{\noindent {\bf Proof} \hspace{2mm}}}
\newcommand{\mbinom}[2]{\left (\!\!{\renewcommand{\arraystretch}{0.5}
\mbox{$\begin{array}[c]{c}  #1\\ #2  \end{array}$}}\!\! \right )}
\newcommand{\brang}[1]{\langle #1 \rangle}
\newcommand{\vstrut}[1]{\rule{0mm}{#1mm}}
\newcommand{\rec}[1]{\frac{1}{#1}}
\newcommand{\set}[1]{\{#1\}}
\newcommand{\dist}[2]{$\mbox{\rm dist}\,(#1,#2)$}
\newcommand{\opname}[1]{\mbox{\rm #1}\,}
\newcommand{\mb}[1]{\;\mbox{ #1 }\;}
\newcommand{\undersym}[2]
 {{\renewcommand{\arraystretch}{0.5}  \mbox{$\begin{array}[t]{c}
 #1\\ #2  \end{array}$}}}
\newlength{\wex}  \newlength{\hex}
\newcommand{\understack}[3]{%
 \settowidth{\wex}{\mbox{$#3$}} \settoheight{\hex}{\mbox{$#1$}}
 \hspace{\wex}  \raisebox{-1.2\hex}{\makebox[-\wex][c]{$#2$}}
 \makebox[\wex][c]{$#1$}   }%
\newcommand{\smit}[1]{\mbox{\small \it #1}}
\newcommand{\lgit}[1]{\mbox{\large \it #1}}
\newcommand{\scts}[1]{\scriptstyle #1}
\newcommand{\scss}[1]{\scriptscriptstyle #1}
\newcommand{\txts}[1]{\textstyle #1}
\newcommand{\dsps}[1]{\displaystyle #1}
\newcommand{\dx}{\,\mathrm{d}x}
\newcommand{\dy}{\,\mathrm{d}y}
\newcommand{\dz}{\,\mathrm{d}z}
\newcommand{\dt}{\,\mathrm{d}t}
\newcommand{\dr}{\,\mathrm{d}r}
\newcommand{\du}{\,\mathrm{d}u}
\newcommand{\dv}{\,\mathrm{d}v}
\newcommand{\dV}{\,\mathrm{d}V}
\newcommand{\ds}{\,\mathrm{d}s}
\newcommand{\dS}{\,\mathrm{d}S}
\newcommand{\dk}{\,\mathrm{d}k}

\newcommand{\dphi}{\,\mathrm{d}\phi}
\newcommand{\dtau}{\,\mathrm{d}\tau}
\newcommand{\dxi}{\,\mathrm{d}\xi}
\newcommand{\deta}{\,\mathrm{d}\eta}
\newcommand{\dsigma}{\,\mathrm{d}\sigma}
\newcommand{\dtheta}{\,\mathrm{d}\theta}
\newcommand{\dnu}{\,\mathrm{d}\nu}


\renewcommand{\div}{\mathrm{div}}
\newcommand{\red}[1]{{\color{red} #1}}

\newcommand{\cqfd}{\begin{flushright}                  
			 $\Box$
                 \end{flushright}}


\title{On gradient estimates for the heat kernel}

\author{Baptiste Devyver}
\address{Baptiste Devyver, Department of Mathematics,  Technion - Israel Institute of Technology, Haifa 32000, Israel}
\email{devyver@technion.ac.il}

\date{March 27, 2018}

\maketitle
\tableofcontents

\begin{abstract}  We study pointwise and $L^p$ gradient estimates of the heat kernel, on manifolds that may have some amount of negative Ricci curvature, provided it is not too negative (in an integral sense) at infinity. We also study boundedness on $L^p$ spaces for the heat operator $e^{-t\vec{\Delta}_k}$ of $\vec{\Delta}_k$, the Hodge Laplacian on differential $k$-forms.

\end{abstract}

\section{Introduction and statement of the results}

\subsection{Introduction}

In their celebrated work \cite{LY}, P. Li and S. T. Yau proved that any positive solution $u(x,t)$ of the heat equation on a complete manifold of dimension $n$ with non-negative Ricci curvature satisfies the following gradient estimate:

\begin{equation}\label{Grad-LY}
\frac{|\nabla u|^2}{u^2}-\frac{u_t}{u}\leq \frac{n}{2t}.
\end{equation}
Later on, many works have been devoted to proving inequalities such as \eqref{Grad-LY} in various settings (see the recent \cite{BBG} for such a result, as well as for references). In order to obtain an inequality such as \eqref{Grad-LY}, it is customary to make a {\em global} curvature (or more generally``curvature-dimension'') assumption. The inequality \eqref{Grad-LY} has many useful consequences. Among them, are the Gaussian upper estimates for the heat kernel $p_t(x,y)$ and its gradient:

\begin{equation}\tag{$LY$}
p_{t}(x,y)\simeq
\frac{1}{V(x,\sqrt{t})}\exp
\left(-\frac{d^{2}(x,y)}{Ct}\right), \quad \forall~t>0,\, x,y\in
 M,\label{UE}
\end{equation}
and

\begin{equation}\label{G}\tag{$G$}
|\nabla_x p_t(x,y)|\lesssim \frac{1}{\sqrt{t}V(x,\sqrt{t})}\exp
\left(-\frac{d^{2}(x,y)}{Ct}\right),,\quad\forall t>0,\,\forall x,\,y\in M.
\end{equation}
While \eqref{UE} has been later on characterized in terms of functional inequalities on $M$ (the so-called relative Faber-Krahn inequalities) by A. Grigor'yan (see \cite{G2}), \eqref{G} remains more mysterious. For example, using the characterization of \eqref{UE} in terms of functional inequalities as in \cite{GSC}, one can show that \eqref{UE} holds on a non-parabolic manifold, that is isometric outside a compact set to a manifold with non-negative Ricci curvature. Such a perturbation result for \eqref{G} is still currently out of reach. The difficulty lies in the fact that most probably, there is no simple functional analytic characterization of \eqref{G}, contrary to \eqref{UE}.  Let us also mention that \eqref{G} and \eqref{UE} are known to hold on nilpotent Lie groups endowed with a sub-Laplacian, thanks to the work of Varopoulos \cite{Var}.

Beyond the customary assumption of non-negativity of the Ricci curvature, and apart from the case of Lie groups, \eqref{G} is known to hold only in very few cases. Some $L^p$ estimates for the gradient of the heat kernel have  recently been obtained in a quite general setting in \cite{MO}, but they are much weaker that \eqref{G}. Also recently, the question of extending the Li-Yau gradient inequality \eqref{Grad-LY} beyond the non-negative Ricci setting has been considered in a few papers, e.g. \cite{ZZ}, \cite{R2}, \cite[Section 3]{C8}. We shall present the known results in details in Section 1.4 --and compare them with our own--, but for the purpose of this introduction, let us limit ourselves to indicate some key facts. Typically, in the above-mentionned papers, an integral bound on the negative part of the Ricci curvature is assumed, and a Li-Yau gradient inequality is deduced; consequences for the heat kernel and its gradient then follow. As an example, in \cite{ZZ} the following local uniform smallness assumption on the negative part of the Ricci curvature is used: there exists $p>\frac{n}{2}$ such that

$$\sup_{x\in M}\left(\frac{1}{V(x,1)}\int_{B(x,1)}||\mathrm{Ric}_-(y)||^p\,dy\right)^{1/p}<\kappa,$$
where $\kappa=\kappa(p,n)$ is a small constant and $V(x,1)$ denotes the volume of the geodesic ball $B(x,1)$. Under this assumption on the Ricci curvature, it is proved in \cite[Theorem 1.1]{ZZ}, that the following Li-Yau gradient estimate holds: for every $T>0$, there are constants $\alpha=\alpha(T,n,\kappa)$ and $\beta=\beta(T,n,\kappa)$ such that, for all positive solution $u(x,t)$ of the heat equation,

\begin{equation}\label{Grad-LY2}
\alpha \frac{|\nabla u|^2}{u^2}-\frac{u_t}{u}\leq \frac{\beta}{t},\quad t\in (0,T).
\end{equation}
This easily yields \eqref{UE} and \eqref{G} {\em for small times}, that is for $t\in (0,T)$ (see \cite{R3} for the upper bound of $p_t(x,y)$). The method was subsequently refined by G. Carron in \cite{C8}, but it is doomed to fail for proving \eqref{G} for {\em large times}, unless one makes a {\em global} size restriction on the negative part of Ricci, such as

\begin{equation}\label{SR}
\sup_{x\in M}\int_{M}G(x,y)||\mathrm{Ric}_-(y)||\,d\mu(y)<\frac{1}{16n},
\end{equation}
where $G(x,y)$ is the positive, minimal Green function on $M$ (see Section 1.4 for a more detailed explanation of this last statement). Gradient estimates for the heat kernel under this kind of assumptions actually goes back to \cite{CZ}, where a different method, that we will present in a moment, has been employed. However, \eqref{SR} is a very strong assumption. In fact, there is an inherent limitation in obtaining gradient estimates for the heat kernel through Li-Yau gradient estimates: namely, that integral bounds on the negative part of the Ricci curvature are not strong enough to give a control on the topology of $M$, more specifically they do not control the number of ends (unless the bound is small enough as in \eqref{SR}, in which case $M$ has only one end). However it is known that if $M$ has several Euclidean ends, then \eqref{G} cannot hold, since there is not extra polynomial decay in time when taking the gradient of the heat kernel (see \cite[Proposition 6.1]{CCH}); while if $M$ has only one end which is Euclidean, some additional decay in time is expected. Let us stress that it is far from being clear how to incorporate this extra topological information in the Li-Yau method.

It turns out that there is another point of view on \eqref{G} that has proved very useful in the last few years: it is the fact that \eqref{G} is related to estimates for the heat kernel of the Hodge Laplacian $\vec{\Delta}=dd^*+d^*d$ on differential $1$-forms. According to \cite{CD3}, if Gaussian type upper estimates are available for both the heat kernel of the scalar Laplacian $\Delta$, and the heat kernel of the Hodge Laplacian $\vec{\Delta}$, then \eqref{G} holds. The Hodge Laplacian on $1$-forms has a well-known Bochner formula, which writes:

$$\vec{\Delta}=\nabla^*\nabla+\mathrm{Ric},$$
which allows one to look at $\vec{\Delta}$ as a generalized Schr\"odinger operator. The Bochner formula implies by a standard domination technique that, if the Ricci curvature is non-negative, then for every smooth, compactly supported $1$-form $\omega$ and all $t\geq0$,

\begin{equation}\label{domin}
||e^{-t\vec{\Delta}}\omega||\leq e^{-t\Delta}||\omega||.
\end{equation}
In fact, it is possible to show that \eqref{domin} is equivalent to the Ricci curvature being non-negative. From \eqref{domin}, it is also possible to recover the fact that \eqref{G} holds if the Ricci curvature is non-negative: indeed, the {\em scalar} heat kernel has Gaussian upper-estimates by \eqref{UE}, hence the domination \eqref{domin} implies that the heat kernel of $\vec{\Delta}$ also has Gaussian estimates, and from this and \eqref{UE} we deduce \eqref{G}. Hence, a proof of \eqref{G} that does not use strictly speaking the Li-Yau gradient inequality \eqref{Grad-LY}, only domination and \eqref{UE}. Recently, in \cite{D2} and later \cite{CDS}, estimates of Gaussian type for the heat kernel of the Hodge Laplacian on $1$-forms have been characterized in a quite general setting, including manifolds that can have a certain amount of negative Ricci curvature, provided it is not ``too negative at infinity''. On these manifolds, the estimates of Gaussian type for the heat kernel of the Hodge Laplacian on $1$-forms have been shown to be equivalent to the fact that $\mathrm{Ker}_{L^2}(\vec{\Delta})=\{0\}$. In particular $\mathrm{Ker}_{L^2}(\vec{\Delta})=\{0\}$ is a sufficient condition for \eqref{G} on these manifolds. This result is interesting, because the condition $\mathrm{Ker}_{L^2}(\vec{\Delta})=\{0\}$ can be interpreted in terms of $L^2$ cohomology, which has a topological/geometrical interpretation in many cases. For example, in the case $M$ has dimension $n\geq3$ and is asymptotically Euclidean, then $\mathrm{Ker}_{L^2}(\vec{\Delta})=\{0\}$ implies that $M$ has only one end. However, one expects that the condition $\mathrm{Ker}_{L^2}(\vec{\Delta})=\{0\}$ is too strong and that \eqref{G} could be obtained under weaker conditions. In the present article we study the validity of \eqref{G} and other related inequalities under conditions that are weaker than $\mathrm{Ker}_{L^2}(\vec{\Delta})=\{0\}$. In essentially the same class of manifolds that was considered in \cite{CDS}, we obtain a general, almost optimal criterion (Theorem \ref{gradient}) for the validity of \eqref{G}. As a consequence of this general criterion, and as a highlight of our results, let us quote right now the following theorem, which will be proved at the end of this article:

\begin{Thm}\label{main_intro}

Let $(M,g)$ be a Riemannian manifold of dimension $n>8$, that is isometric to the cone over a compact, connected manifold $(X,\bar{g})$ at infinity. Assume that the Ricci curvature on $X$ is bounded from below by $(n-2)\bar{g}$, and that $X$ is not isometric to the $(n-1)$-dimensional Euclidean sphere. Then, the gradient heat kernel estimates \eqref{G} hold on $M$. 

\end{Thm}
In the case the manifold $X$ is the Euclidean sphere $S^{n-1}$, i.e. $M$ is isometric to the Euclidean space at infinity, we fail to obtain \eqref{G}, however we can still prove that \eqref{G} ``almost holds''. The unnatural restriction $n>8$ is needed for technical reasons, to guarantee that $L^2$ harmonic $1$-forms decay fast enough at infinity. Another feature of our approach is that it is flexible enough to yield estimates for the heat kernel of operators such as $\vec{\Delta}_k$, the Hodge Laplacian on $k$-forms, and actually we will also obtain sufficient criteria for the uniform boundedness of the heat operator $e^{-t\vec{\Delta}_k}$ on some $L^p$ spaces. We refer to Section 1.3 for a detailed discussion of the results obtained in this article.

\subsection{Preliminaries}

After this quick introduction, let us introduce the setting. Once this has been done, our results will be presented in more details, in the next subsection. Let $M$ be a complete, connected, non-compact Riemannian manifold, endowed with a positive measure $\mu=e^{f}\nu$, absolutely continuous with respect to the Riemannian measure $\nu$. We assume that $f$ is smooth. We denote by $\nabla$ the Riemannian gradient, by $\Delta_\mu u=-\div(\nabla u)-\langle \nabla f,\nabla u\rangle$ the weighted non-negative Laplace operator. In the sequel, we will often denote $\Delta_\mu$ simply by $\Delta$, the dependence to the measure $\mu$ being thus implicit. Let $d$ denote the geodesic distance, $B(x,r)$ the open  ball for  $d$ with centre $x\in M$ and radius $r>0$, and $V(x,r)$ its volume $\mu\left(B(x,r)\right)$.

We will use the notation $h\lesssim g$ to indicate that there exists a constant $C$ (independent of the important parameters) such that $h\leq Cg$, $h\gtrsim g$ if $g\lesssim h$,  and $h\simeq g$ if $h\lesssim g$ and $h\gtrsim g$. 

The weighted manifold $(M,d,\mu)$ will be said to satisfy the volume doubling property if
  \begin{equation}\label{d}\tag{$V\!D$}
     V(x,2r)\lesssim  V(x,r),\quad \forall~x \in M,~r > 0.
    \end{equation}
It follows easily from \eqref{d} that there exists  $\nu>0$  such that
     \begin{equation*}\label{dnu}\tag{$V\!D_\nu$}
      \frac{V(x,r)}{V(x,s)}\lesssim \left(\frac{r}{s}\right)^{\nu} ,\quad \forall~ x \in M,~r \geq s>0.
    \end{equation*}
It is known (see  \cite[Theorem 1.1]{G1}) that if $M$ is  connected, non-compact, and satisfies \eqref{d}, then the following reverse doubling condition holds:
    there
exists
 $0<\nu'\leq \nu$  such that,
for all $r\geq s>0$ and $x\in M$,
\begin{equation}\label{rnu}\tag{$RD_{\nu'}$}
 \frac{V(x,r)}{V(x,s)}\gtrsim\left(\frac{r}{s}\right)^{\nu'}.
\end{equation}
Let us introduce the following volume lower bound: for some $x_0\in M$  there exists  $\kappa>0$ 
such that   
  \begin{equation}\label{dd}\tag{$V\!L_{\kappa}$}
  V(x_0,r)\gtrsim  r^{\kappa},\qquad\forall r\geq1.
  \end{equation}
From \eqref{d}, clearly this condition does not depend on the choice of 
$x_0$. Taking $s=1$ in  \eqref{rnu} shows that \eqref{dd} always holds at least for  $\kappa=\nu'$. We introduce an hypothesis of uniformity for the volume of balls, that will be assumed in many results of the paper: there is a constant $C>0$ such that, for all $x\in M$, $y\in M$, $t>0$,

\begin{equation}\label{vol_u}\tag{VU}
C^{-1}V(y,t)\leq V(x,t)\leq CV(y,t).
\end{equation}
If \eqref{vol_u} holds, then one fixes a point $x_0\in M$ once and for all, and one lets

$$V(t):=V(x_0,t).$$
The doubling assumption \eqref{d} then implies that for all $t>0$,

$$V(2t)\leq V(t).$$
Independently of the validity of \eqref{vol_u}, we will denote $V_{\sqrt{t}}$ the operator of multiplication by the function $V(x,\sqrt{t})$

Sometimes, we will in addition assume that $(M,d,\mu)$ satisfies the Sobolev inequality with parameter $n>2$, that is

\begin{equation}\label{Sob}\tag{$\mathrm{Sob}_n$}
||u||_{\frac{2n}{n-2}}\lesssim ||\nabla u||_2,\quad \forall u\in C_0^\infty(M).
\end{equation}
Note that $n$ does not have to be equal to the dimension of $M$. In the Riemannian case, i.e. when $\mu$ is the Riemannian measure, it is known that $n$ has to be greater or equal to the dimension of $M$ (see \cite{C9}) but it may happen that the inequality be strict.


We now introduce the heat kernel. Let $e^{-t\Delta}$ be the heat operator associated to $\Delta$, and  $p_t(x,y)$ its kernel, so that, for any  compactly supported smooth function $f$ on $M$, there holds:
$$e^{-t\Delta}u(x)=\int_M p_t(x,y)u(y)d\mu(y).$$
It is classical that $p_t$ is smooth, positive, satisfies $p_t(x,y)=p_t(y,x)$, and that under \eqref{d}, $\int_Mp_t(x,y)\,d\mu(y)= 1$ for all $x$, in other words $M$ is stochastically  complete. 

On and off-diagonal estimates of the heat kernel $p_t(x,y)$ have been studied in detail in the past thirty years and are well understood. Let us start by introducing the on-diagonal estimate:

\begin{equation}\tag{$DU\!E$}
p_{t}(x,x)\lesssim
\frac{1}{V(x,\sqrt{t})}, \quad \forall~t>0,\,\forall\,x\in
 M. \label{due}
\end{equation}
Under  \eqref{d}, \eqref{due} self-improves into an (off-diagonal) Gaussian upper estimate (\cite[Theorem 1.1]{Gr1}, see also \cite[Section 4.2]{CS}):
\begin{equation}\tag{$U\!E$}
p_{t}(x,y)\lesssim
\frac{1}{V(x,\sqrt{t})}\exp
\left(-\frac{d^{2}(x,y)}{Ct}\right), \quad \forall~t>0,\, \mbox{a.e. }x,y\in
 M,\label{UE}
\end{equation}
for some $C>0$. It is known (see \cite{G2}) that the Gaussian upper estimate \eqref{UE} is equivalent to an $L^2$  isoperimetric-type inequality called the relative Faber-Krahn inequality, known to hold if the Ricci curvature is non-negative. Beyond this assumption, \eqref{UE} is known to hold for small times, under an integral assumption on the negative part of the Ricci curvature (see \cite{R1}). Let us also introduce the upper and lower Gaussian estimates for the heat kernel (sometimes called Li-Yau estimates):
\begin{equation}\tag{$LY$}
p_{t}(x,y)\simeq
\frac{1}{V(x,\sqrt{t})}\exp
\left(-\frac{d^{2}(x,y)}{Ct}\right), \quad \forall~t>0,\, \mbox{a.e. }x,y\in
 M,\label{LY}
\end{equation}
where $C$ denotes a possibly different constant in the upper and the lower bound. By the famous work of Li and Yau \cite{LY}, such estimates are known to hold on a Riemannian manifold having non-negative Ricci curvature. By a theorem of  Saloff-Coste (see \cite{SC} and the references therein, see also \cite{G1} for an alternative approach of the main implication, as well as the more recent \cite{BCF}), it is known that under \eqref{d}, \eqref{LY} is equivalent to the following family of scale-invariant $L^2$ Poincar\'{e} inequalities: there is a constant $C$ such that, for every geodesic ball $B=B(x,r)$, and every $u\in C^\infty(B)$,
\begin{equation}\label{P}\tag{$P$}
\int_{B}|u-u_{B}|^2\,d\mu\leq Cr^2\int_{B}|\nabla u|^2\,d\mu,
\end{equation}
where $u_{B}=\frac{1}{\mu(B)}\int_Bu\,d\mu$ denotes the average of $u$ over $B$. By \cite{LY}, \eqref{LY} holds on a complete manifold with non-negative Ricci curvature endowed with its Riemannian measure.

%
%
We now recall the notion of non-parabolicity (see for instance \cite[Section 5]{G4} for more information).  One says that $M$ is non-parabolic if
$$\int_1^\infty p_t(x,y)\,dt<+\infty$$
for some (all) $x, y\in M$.  In this case, $G(x,y)$ defined by $$G(x,y)=\int_0^{+\infty} p_t(x,y)\,dt$$ is finite for all $x\neq y$, and is the positive, minimal Green function of $\Delta$.
If furthermore \eqref{d} and \eqref{UE} hold, then the non-parabolicity of $M$ is equivalent to
\begin{equation}\label{vol1}
\int_1^{+\infty}\frac{dt}{V(x_0, \sqrt{t})}<+\infty,
\tag{$V^\infty$}
\end{equation}
for some  $x_0\in M$ (this uses the fact that under \eqref{d} and \eqref{UE}, the heat kernel has an on-diagonal lower bound $p_t(x,x)\geq \frac{C}{V(x,\sqrt{t})}$; see \cite[Theorem 11.1]{G4}). We also introduce an additional, related integral volume growth condition:
\begin{equation}\label{vol2}
\int_1^{+\infty}\frac{dt}{\left[V(x_0, \sqrt{t})\right]^{1-\frac{1}{p}}}<+\infty\tag{$V^p$}
\end{equation}
for some  $x_0\in M$.
It follows from condition  \eqref{d} that \eqref{vol1} and \eqref{vol2} do not depend on the 
choice of $x_0$. Notice that \eqref{dd} with $\kappa>2$ implies \eqref{vol2} for all $p\in \left(\frac{\kappa}{\kappa-2},+\infty\right]$, hence non-parabolicity under  \eqref{d} and \eqref{UE}.

\vskip5mm

We will also consider the heat kernel of elliptic operators of Schr\"{o}dinger type, acting on sections of a vector bundle over $M$. The main example we have in mind is the heat kernel of the (weighted) Hodge Laplacian $\vec{\Delta}_{k,\mu}=dd_\mu^*+d_\mu^*d$, acting on $k$-forms, which we describe now. Here, we have denoted $d_\mu^*$ the formal adjoint of $d$ with respect to the measure $\mu$. Denote by $e^{-t\vec{\Delta}_{k,\mu}}$ the associated heat operator, and by $\vec{p_t}^{k,\mu}(x,y)$ its kernel. When $k=1$ and $\mu$ is the Riemannian measure, we will simply write $\vec{\Delta}$ and $\vec{p_t}(x,y)$, and in order not to make notations too heavy, we shall often make the measure $\mu$ implicit and simply write $\vec{p_t}^{k}(x,y)$. It will in practice be clear what reference measure $\mu$ has been taken. Thus, for every $x$ and $y$ in $M$, $\vec{p_t}^{k,}(x,y)$ is a linear endomorphism from $\Lambda^kT^*_yM$ to $\Lambda^k T^*_xM$, where $\Lambda^kT^*M$ is the vector bundle of $k$-forms on $M$. By definition, for every compactly supported  smooth $k$-forms $\omega$ and $\eta$,  there holds:

$$\langle e^{-t\vec{\Delta}_{k}} \omega,\eta\rangle=\int_M (\vec{p_t}^{k}(x,y)\omega(y),\eta(x))_x\,d\mu(x)d\mu(y).$$
We consider the Gaussian estimates for $\vec{p_t}^{k}(x,y)$:

\begin{equation}\tag{$\vec{U\!E}_k$}
\|\vec{p_t}^{k}(x,y)\|_{y,x}\lesssim
\frac{1}{V(x,\sqrt{t})}\exp
\left(-\frac{d^{2}(x,y)}{Ct}\right), \quad \forall~t>0,\, \mbox{a.e. }x,y\in
 M,\label{vecUE}
\end{equation}
for some $C>0$.
Here $\|\cdot\|_{y,x}$  denotes the Hilbert-Schmidt norm of  the operator $\vec{p_t}^{k}(x,y)$ from $\Lambda^kT^*_yM$ to $\Lambda^kT^*_xM$ endowed with the Riemannian metrics at $y$ and $x$. In the case $k=1$, we will simply write ($\vec{U\!E}$) instead of ($\vec{U\!E}_k$). It turns out that even in the weighted setting, there is a Bochner formula for the Laplacian on $k$-forms (see \cite[Appendix]{CDS}):

\begin{equation}\label{Boch-k}
\vec{\Delta}_{k,\mu}=\nabla^*\nabla + \mathscr{R}_{k,\mu},
\end{equation}
where $\nabla$ is the Riemannian connection, $\nabla^*$ its formal adjoint with respect to the measure $\mu$, and for every $x\in M$, $\mathscr{R}_{k,\mu}(x)$ is a symmetric endomorphism of the fiber of $\Lambda^kT^*M$ at $x$. Hence, $\vec{\Delta}_{k,\mu}$ can be seen as a {\em generalised Schr\"{o}dinger operator}, with potential $\mathscr{R}_{k,\mu}$; the term $\mathscr{R}_{k,\mu}$ can furthermore be expressed in terms of the curvature tensor. For instance, for $k=1$, it identifies naturally with the weighted Ricci curvature $\mathrm{Ric}_\mu:=\mathrm{Ric}-\mathscr{H}_f$, if $\mu=e^f\nu$, where $\mathscr{H}_f$ is the Hessian operator of $f$ (see \cite[Appendix]{CDS}).

In \cite{CDS}, Gaussian estimates for the heat kernel of generalised Schr\"odinger operators were studied, and in the present article we will work in the same setting, which we introduce now. We consider a generalised Schr\"{o}dinger operator
$$\mathcal{L}=\nabla^*\nabla + \mathcal{R},$$
acting on a finite-dimensional Riemannian  bundle $E\rightarrow M$,  that is a finite-dimensional vector bundle equipped with a scalar product $\left(\cdot, \cdot\right)_x$ depending continuously on $x\in M$ (see for instance \cite[Section E]{Be}). Here $\nabla$ is a connection on $E\rightarrow M$ which is  compatible with the metric, and $\nabla^*\nabla$ is the so-called ``rough Laplacian'', also denoted by $\bar{\Delta}$. Of course, the formal adjoint $\nabla^*$ depends on the measure $\mu$, and we really should write $\nabla^*_\mu$ instead of $\nabla^*$, but to keep notations light we prefer to keep this dependance implicit. The ``potential'' $\mathcal{R}$ is by definition a $L^\infty_{loc}$ section of the vector bundle $\mathrm{End}(E)$, that is, for all $x\in M$, $\mathcal{R}(x)$ is a symmetric endomorphism of $E_x$, the fiber at $x$. Notice that if $E$ is the trivial $M\times \R$, a generalised Schr\"{o}dinger operator on $E$ is just a {\em scalar} Schr\"{o}dinger operator $\Delta+V$, where $V: M\rightarrow \R$ is a real potential. Since $\mathcal{R}$ is in $L^\infty_{loc}$, by standard elliptic regularity, solutions of $\mathcal{L}\omega=0$ are contained in $C_{loc}^{1,\alpha}$ for all $\alpha\in (0,1)$. For a.e. $x\in M$, one can diagonalize $\mathcal{R}(x)$  in an orthonormal basis of $E_x$.
Denote by $\mathcal{R}_+(x)$ the endomorphism corresponding to the non-negative eigenvalues, and  by $-\mathcal{R}_-(x)$ the one corresponding to the negative eigenvalues, so that $\mathcal{R}_+(x)$, $\mathcal{R}_-(x)$ are a.e. non-negative symmetric endomorphisms acting on the fiber $E_x$ and

$$\mathcal{R}=\mathcal{R}_+-\mathcal{R}_-.$$
Notice also that $\mathcal{R}_+$ and $\mathcal{R}_-$ belong to $L^\infty_{loc}$. Denote by $|\cdot|_x$ the norm on $E_x$ derived from $\left(\cdot, \cdot\right)_x$ and by $\|\cdot\|_x$ the induced norm on    $\mathrm{End}(E_x)$. In particular,

$$\|\mathcal{R}_-(x)\|_x=\sup_{v\in E_x,|v|_x=1}|\left(\mathcal{R}_-(x)v,v\right)_x|=\max \sigma(\mathcal{R}_-(x)),$$
where $\sigma(\mathcal{R}_-(x))$ is the (finite)  set of eigenvalues of $\mathcal{R}_-(x)$. Let  $C^\infty(E)$ (resp. $C_0^\infty(E)$) be the set of smooth sections of $E$ (resp. of smooth ``compactly supported'' sections, that is sections which coincide with the zero section outside a compact set). For $p\geq 1$ we will consider the $L^p$-norm on sections of $E$:
$$||\omega||_p=\left(\int_M|\omega(x)|_x^p\,d\mu(x)\right)^{1/p}$$
with the usual extension for $p=\infty$.
We shall denote  by $L^p(E)$, or simply $L^p$ when  no confusion is possible,  the set of sections of $E$ with finite $L^p$-norm, modulo equality a.e.. Sometimes, depending on the context, $L^p$ will simply refer to real-valued functions. We will denote by 
$$\langle\omega_1,\omega_2\rangle:=\int_M(\omega_1(x),\omega_2(x))_x\,d\mu(x)$$ the scalar product in $L^2(E)$.

From an obvious adaptation of Strichartz's proof that the Laplacian is self-adjoint on a complete manifold (see Theorem 3.13 in \cite{PRS}), we know that if $\mathcal{R}_-$ is bounded, then $\mathcal{L}=\bar{\Delta}+\mathcal{R}_+-\mathcal{R}_-$ is essentially self-adjoint on $C_0^\infty(E)$. If $\mathcal{R}_-$ is not bounded, then we will consider the Friedrichs extension of $\mathcal{L}$, which is an unbounded, self-adjoint operator on $L^2(E)$. We will make a crucial assumption on $\mathcal{L}$: we will assume that it is {\em non-negative}. By this we mean that the associated quadratic form to $\mathcal{L}$ is non-negative, that is for every $\omega\in C_0^\infty(E)$,

$$\int_M|\nabla \omega|^2+\langle\mathcal{R}\omega,\omega\rangle\geq0.$$
By the spectral theorem, one can then consider $e^{-t\mathcal{L}}$, which is a contraction semigroup on $L^2$. By elliptic estimates again, there exists a kernel $p_t^{\mathcal{L}}$ for $e^{-t\mathcal{L}}$, that is, for every $x,y\in M$, $p_t^{\mathcal{L}}(x,y)$ is a linear map from $E_y$ to $E_x$ such that for every $\omega_i\in L^2(E)$, $i=1,2$,

$$\langle e^{-t\mathcal{L}}\omega_1,\omega_2\rangle=\int_{M\times M}\left(p_t^{\mathcal{L}}(x,y)\omega_1(y),\omega_2(x)\right)_x\,d\mu(x)d\mu(y).$$
By self-adjointness of $\mathcal{L}$, the adjoint of $p_t^{\mathcal{L}}(x,y)$ is $p_t^{\mathcal{L}}(y,x)$ a.e. In addition the non-negativity of $\mathcal{L}$ implies that for a.e. $x\in M$, the operator $p_t^{\mathcal{L}}(x,x)\in \mathrm{End}(E_x)$ is symmetric and non-negative. Denote by $\|\cdot\|_{y,x}$ the Hilbert-Schmidt norm on $\mathrm{End}(E_y,E_x)$, that is

$$||A||^2_{y,x}:=\mathrm{Tr}(A^*A)=\mathrm{Tr}(AA^*),\quad A\in \mathrm{End}(E_y,E_x).$$
We will say that the heat kernel of $\mathcal{L}$ satisfies Gaussian estimates if for every $t>0$ and a.e $x,y\in M$,

\begin{equation}\label{UEL}\tag{$U\!E_{\mathcal{L}}$}
||p_t^{\mathcal{L}}(x,y)||_{y,x}\lesssim
\frac{1}{V(x,\sqrt{t})}\exp
\left(-\frac{d^{2}(x,y)}{Ct}\right).
\end{equation}

As we mentioned before, an important example of such a generalised Schr\"{o}dinger operator $\mathcal{L}$ is provided by $\Delta_k=d^*_{k+1}d_k+d_{k-1}d^*_k$, the Hodge Laplacian acting on differential $k$-forms. It obviously follows  from its definition that $\Delta_k$ is non-negative. It turns out that if $\mu$ is not the Riemannian measure, the corresponding {\em weighted} Hodge Laplacian are also of interest: such operators have been considered by E. Witten \cite{Wit} and J-M. Bismut \cite{Bis}, in order to give a new proof of the Morse inequalities, and later by E. Bueler \cite{Bue} to study the cohomology of non-compact manifolds.

\bigskip

In \cite{CDS}, Gaussian estimates for the heat kernel of such a generalised Schr\"odinger operator have been characterized, under the assumption that the negative part of the potential $\mathcal{R}$ is ``small at infinity'' in a certain sense. More precisely, we shall say that condition \eqref{condik} is satisfied by a  section $\mathcal{V}\in L^\infty_{loc}$ of the vector bundle $\mathrm{End}(E)$ on $M$ if  $M$ is non-parabolic and  there is a compact subset $K_0$ of $M$ such that
\begin{equation}\label{condik}\tag{$K$}
\sup_{x\in M}\int_{M\setminus K_0}G(x,y)\|\mathcal{V}(y)\|_y\,d\mu(y)<1,\end{equation}
where $G$ is the Green function for the Laplace operator on functions.
In a more compact way, condition \eqref{condik}  may be formulated as $\|\Delta^{-1}(\|\mathcal{V}\|\mathbf{1}_{M\setminus K_0})\|_{\infty}<1$. The standing assumption that has been made in \cite{CDS} in order to study the heat kernel of a generalised Schr\"odinger operator $\mathcal{L}=\nabla^*\nabla+\mathcal{L}$ is that $\mathcal{R}_-$ satisfies \eqref{condik}. Actually, \eqref{condik} is a generalization of the more familiar Kato class at infinity $K^\infty(M)$, defined by: $\mathcal{V}\in K^\infty(M)$ if

\begin{equation}\label{Kato}
\lim_{R\to\infty}\sup_{x\in M}\int_{M\setminus B(x_0,R)}G(x,y)\|\mathcal{V}(y)\|_y\,d\mu(y)=0,
\end{equation}
for some (all) $x_0\in M$. This is a slightly more restrictive condition than our condition \eqref{condik}. The notion of Kato class at infinity is often used as an effective tool to obtain Gaussian estimates for the heat kernel of Schr\"{o}dinger operators (see \cite{T}). Also, in \cite{D3}, a sufficient condition for a potential to satisfy \eqref{Kato} is presented, in terms of some weighted $L^p$ spaces: if \eqref{UE} holds, and there is $\varepsilon>0$ such that 

$$\mathcal{V}\in L^{\frac{\nu}{2}-\varepsilon}\left(M,\frac{d\mu(x)}{V(x,1)}\right)\cap L^{\frac{\nu'}{2}+\varepsilon}\left(M,\frac{d\mu(x)}{V(x,1)}\right),$$
where $\nu$ and $\nu'$ are the exponents from \eqref{d} and \eqref{rnu},
then $\mathcal{V}$ belongs to the Kato class at infinity $K^\infty(M)$, and thus in particular satisfies \eqref{condik}. In the case where the volume growth is polynomial, that is

$$C^{-1}r^n\leq V(x,r)\leq Cr^n,\qquad \forall x\in M,\,\forall r>0,$$
with $n>2$, the above condition on $\mathcal{V}$ is the familiar condition $\mathcal{V}\in L^{\frac{n}{2}\pm \varepsilon}$. 

Let us now quote one of the main results of \cite{CDS}:

\begin{Thm}\label{main-CDS}\cite[Theorem 1.1]{CDS}

Let $(M,\mu)$ be  a complete, non-compact, connected weighted Riemannian manifold satisfying  \eqref{d} and \eqref{UE}. 
Assume that $(M,\mu)$ satisfies the volume lower bound \eqref{dd} with  $\kappa>4$.
Let $E$ be a vector bundle with  basis $M$ and a connection $\nabla$ compatible with the metric, and let 
$$\mathcal{L}=\nabla^*_\mu\nabla+\mathcal{R}_+-\mathcal{R}_-$$
be a generalised (weighted) Schr\"{o}dinger operator on $E$, such that $\mathcal{R}_-$ satisfies condition~$\eqref{condik}$. Then the following are equivalent:

\begin{enumerate}

\item[i)] the heat kernel of $\mathcal{L}$ satisfies \eqref{UEL}.

\item[ii)]  $\mathrm{Ker}_{L^2}(\mathcal{L})=\{0\}$. 

\end{enumerate}

\end{Thm}
 In the case where $\mathcal{L}$ is the Hodge Laplacian on $1$-forms, Theorem \ref{main-CDS} has important consequences for the gradient of the heat kernel on functions. First, let us say that that the gradient of the heat kernel on functions has Gaussian estimates if the following estimates hold:
 
\begin{equation}\label{G}\tag{$G$}
|\nabla_x p_t(x,y)|\lesssim \frac{1}{\sqrt{t}V(x,\sqrt{t})}e^{-c\frac{d^2(x,y)}{t}},\quad\forall t>0,\,\forall x,\,y\in M.
\end{equation} 
Then, one has the following result:
 
 \begin{Cor}\label{Cor-CDS}\cite[Corollary 1.13]{CDS}
 
 Let $(M,\mu)$ be  a complete, non-compact, connected weighted Riemannian manifold satisfying  \eqref{d} and \eqref{UE}. 
Recall the weighted Ricci tensor $\mathrm{Ric}_\mu:=\mathrm{Ric}-\mathscr{H}_f$, if $\mu=e^f\nu$.
Assume that $(\mathrm{Ric}_{\mu})_-$ satisfies condition~$\eqref{condik}$. Assume moreover that $M$ satisfies \eqref{dd} for some $\kappa>4$, and that $\mathrm{Ker}_{L^2}(\vec{\Delta_\mu})$, the kernel of the weighted Hodge Laplacian on $1$-forms, is trivial. Then, the Gaussian gradient estimates \eqref{G} for the heat kernel of the weighted Laplacian hold.

 \end{Cor}

\subsection{Our results}

Our work in this article was motivated by the following two questions. The first one concerns a possible extension of Theorem \ref{main-CDS}:

\begin{Ques}\label{Q1}

In the setting of Theorem \ref{main-CDS}, what happens if $\mathrm{Ker}_{L^2}(\mathcal{L})$ is not trivial? If $\Pi$ is the $L^2$-orthogonal projection onto $\mathrm{Ker}_{L^2}(\mathcal{L})$, is it true that the operator $e^{-t\mathcal{L}}(I-\Pi)$ has Gaussian estimates?

\end{Ques}
For example, in the case of $\vec{\Delta}$, the Hodge Laplacian on $1$-forms, one of the simplest manifolds for which Question \ref{Q1} is relevant is $\R^n\sharp\R^n$, a so-called {\em connected sum of two Euclidean spaces}, which is a manifold obtained by taking two disjoint copies of $\R^n$ on which the unit ball has be removed, and connecting them smoothly by a tube joining their boundaries. Indeed, if $n\geq3$, it is known that on $\R^n\sharp\R^n$, $\mathrm{Ker}_{L^2}(\vec{\Delta})\neq\{0\}$  (see e.g. \cite{CCH}), and hence the Gaussian estimate for the heat kernel of the Hodge Laplacian on $1$-forms ($\vec{U\!E}$) does not hold. Even in this particular case, it is an interesting open problem to understand the action of the semi-group $e^{-t\vec{\Delta}}$ on $L^p$ spaces (see the recent results in \cite{MO}). The second question concerns Corollary \ref{Cor-CDS}:

\begin{Ques}\label{Q2}

In the setting of Corollary \ref{Cor-CDS}, does the gradient estimate for the heat kernel \eqref{G} hold under a weaker assumption than $\mathrm{Ker}_{L^2}(\vec{\Delta})=\{0\}$? 

\end{Ques}
In fact, Questions \ref{Q1} and \ref{Q2} are intimately related: indeed, in our setting Gaussian estimates for $e^{-t\vec{\Delta}}(I-\Pi)$, where $\vec{\Delta}$ is the (weighted) Hodge Laplacian on $1$-forms, imply Gaussian gradient estimates for the scalar heat kernel. Actually, it turns out that an even weaker statement is true (see the proof of Theorem \ref{gradient}):

\begin{Pro}\label{forms_imply_grad}

Let $(M,\mu)$ be a complete weighted Riemannian manifold, satisfying \eqref{d}, \eqref{vol_u} and \eqref{UE}. Let $\Pi$ be the orthogonal projection onto $\mathrm{Ker}_{L^2}(\vec{\Delta}_\mu)$. Assume that

$$\sup_{t>0}||V_{\sqrt{t}}^{\frac{1}{2}}e^{-t\vec{\Delta}_\mu}(I-\Pi)||_{2\to\infty}<+\infty.$$
Then, the Gaussian gradient estimates \eqref{G} for the scalar heat kernel hold.

\end{Pro}
The result of Proposition \ref{forms_imply_grad} is reminiscent of the classical fact, first observed in \cite{CD2}, that Gaussian estimates for the heat kernel on forms have consequences for the gradient of the heat kernel on functions. It is known that the Gaussian gradient estimates for the heat kernel \eqref{G} cannot hold on $\R^n\sharp\R^n$ (see Remark \ref{grad_con_sum}), so the answer to Question \ref{Q1} is {\em no in general}. Therefore, in order to proceed one needs to make an extra assumption. For $q\in [1,2)$, we introduce the following assumption for the generalised Schr\"odinger operator $\mathcal{L}$:

\begin{equation}\label{Ker}\tag{$\mathrm{Ker}_q$}
\mathrm{Ker}_{L^2}(\mathcal{L})\subset L^q
\end{equation} 
In \cite[Proposition 11]{D2}, it was established that \eqref{Ker} holds for all $q\in (\frac{n}{n-2},2)$ under the assumptions that the Sobolev inequality \eqref{Sob} holds, that the negative part of $\mathcal{R}$ is in $L^{\frac{n}{2}\pm\varepsilon}$ for some $\varepsilon>0$, and that the volume of balls is Euclidean. Also, the results of \cite{GS} show that the behaviour at infinity of $L^2$ harmonic $k$-forms govern the boundedness of the Riesz transform $(d+d^*)\vec{\Delta}_k^{-1/2}$ on $k$-forms, and we will see that the same kind of phenomenon occurs for the heat operator. 

For us, the assumption \eqref{Ker} will be useful for all $q$'s in the range $[1,2)$, and the lower $q$, the better. We will explain later in details the significance and relevance of \eqref{Ker}, but for the moment let us present our first main result, which partially answers Question \ref{Q1}:

\begin{theorem}\label{pi_heat}

Assume that $M$ satisfies \eqref{d}, \eqref{UE} and \eqref{vol_u}. Furthermore, assume that $M$ satisfies the volume lower bound \eqref{dd} with $\kappa>8$. Let $E$ be a vector bundle with basis $M$, endowed with a connection $\nabla$ compatible with the metric, and let

$$\mathcal{L}=\nabla^*\nabla+\mathcal{R}_+-\mathcal{R}_-$$
be a generalised Schr\"odinger operator on $E$, such that $\mathcal{R}_-$ satisfies condition \eqref{condik}. Assume that \eqref{Ker} holds, and let $\Pi$ be the $L^2$ orthogonal projection onto $\mathrm{Ker}_{L^2}(\mathcal{L})$. Then for all $q\leq r\leq  s\leq q'$,

\begin{equation}\label{VPiV}
\sup_{t>0} ||V_{\sqrt{t}}^{\frac{1}{r}-\frac{1}{s}}e^{-t\mathcal{L}}(I-\Pi)||_{r\to s}<\infty.
\end{equation}
Furthermore, if \eqref{Ker} holds for $q=1$ and the kernel $k_t(x,y)$ of $e^{-t\mathcal{L}}(I-\Pi)$ is continuous, then it has on-diagonal estimates:

\begin{equation}\label{on-diago}
||k_t(x,x)||_{x,x}\lesssim \frac{1}{V(x,\sqrt{t})}.
\end{equation}

\end{theorem}

\begin{Rem}
{\em 
However, we do not know how to pass from the on-diagonal estimate \eqref{on-diago} to the full Gaussian estimates for the kernel of $e^{-t\mathcal{L}}(I-\Pi)$. Indeed, passing from on-diagonal to off-diagonal estimates usually requires finite speed of propagation for the associated wave operator (or, equivalently, Davies-Gaffney estimates), see \cite{Sik}, but in the proof we use the operator $\mathcal{L}+\alpha \Pi$, whose wave propagator does not have finite speed of propagation, because of the non-locality of the projection $\Pi$. This non-locality feature of the operator $\mathcal{L}+\alpha \Pi$ is also the reason why in Theorem \ref{pi_heat} we need to assume the uniformity of the volume growth \eqref{vol_u}. However, we feel that the result of Theorem \ref{pi_heat} should hold without this extra assumption.
}
\end{Rem}
Theorem \ref{pi_heat} has for consequence the following result:

\begin{Cor}\label{unif_bdd}

Under the assumptions of Theorem \ref{pi_heat}, the semi-group $e^{-t\mathcal{L}}$ is uniformly bounded on $L^p$, for every $p\in [q,q']$.

\end{Cor}
Now, we turn to the consequences of Theorem \ref{pi_heat} for Question \ref{Q2}. We first introduce weaker integral estimates for the gradient of the scalar heat kernel:

\begin{equation}\label{Gp}\tag{$G_p$}
||\nabla e^{-t\Delta}||_{p\to p}\lesssim \frac{1}{\sqrt{t}},\quad\forall t>0.
\end{equation}
It is known that if $M$ is complete, then \eqref{Gp} always holds for $p\in (1,2]$ (see \cite{Chen}; the statement is made for unweighted manifolds, but the proof extends without difficulty to the weighted case as well). Let us also mention that the weaker $L^p$ estimates $||\nabla e^{-t\Delta}||_{p\to p}\lesssim t^{-1/p}$ have been recently established in \cite{MO}, under some local uniformity of the negative part of the Ricci curvature. In the case $p=\infty$, the gradient estimate that we will consider is a bit different. It writes:

\begin{equation}\label{Ginf}\tag{$G_\infty$}
\sup_{x,y\in M}|\nabla_x p_t(x,y)|\lesssim \frac{1}{\sqrt{t}V(y,\sqrt{t})},\quad \forall t>0.
\end{equation}
According to \cite{ACDH}, \eqref{Ginf} implies \eqref{Gp}, for all $p\in (2,\infty)$, and moreover by \cite[Theorem 4.11]{CS}, \eqref{Ginf} is equivalent to the full Gaussian upper bound \eqref{G} for the gradient of the heat kernel. The $L^p$ estimates \eqref{Gp} for the gradient of the heat kernel have been introduced in \cite{ACDH}, in order to study the $L^p$ boundedness of the Riesz transform $d\Delta^{-1/2}$. Indeed, a well-known argument involving the analyticity of the semigroup $e^{-t\Delta}$ on $L^p$ shows that if the Riesz transform $d\Delta^{-1/2}$ is bounded on $L^p$ for some $p\in (1,\infty)$, then \eqref{Gp} holds (see \cite{ACDH}). The optimal range of boundedness for the Riesz transform is known in a few cases, for instance it is known for asymptotically Euclidean -or more generally, conical- manifolds (see \cite{CCH}, \cite{GH}). This indirectly yields \eqref{Gp} for some values of $p$. If $p\in (2,\infty)$ and $M$ satisfies the Poincar\'e inequalities \eqref{P}, by the main result of \cite{ACDH} the boundedness on $L^p$ of the Riesz transform and \eqref{Gp} are essentially equivalent. However, it is an open problem to provide weaker conditions than \eqref{P} for the converse that would apply to manifolds such as the connected sum of two Euclidean spaces (see however the recent \cite{BCF}). Also, on the connected sum of two Euclidean spaces, the Riesz transform is bounded on $L^p$ if and only if $p\in (1,n)$ ($p\in (1,2]$ if $n=2$). While it is known that \eqref{Gp} cannot hold for $p>n$, what happens in the limiting case $p=n$, $n\geq3$ is still unknown. In addition, if one is to study gradient heat kernel estimates through the Riesz transform, a serious limitation that one faces is that in this way, one can only get results in the range $p\in (1,\infty)$, but the case $p=\infty$ is excluded. Hence, the Gaussian gradient heat kernel estimates \eqref{G} are out of reach by this method. In the contrary, our method in the present article allows us to deal with the case $p=\infty$, as will be demonstrated by our next result. But first, let us introduce another family of inequalities for the gradient of the heat kernel:

\begin{equation}\label{grad_r_s}\tag{$G_{r,s}$}
\sup_{t>0}\left\| V_{\sqrt{t}}^{\frac{1}{r}-\frac{1}{s}}\sqrt{t}\nabla e^{-t\Delta}\right\|_{r\to s} <+\infty,
\end{equation}
and its off-diagonal counterpart

\begin{equation}\label{grad_off}\tag{$G^{\mathrm{off}}_{r,s}$}
V(x,\sqrt{t})^{\frac{1}{r}-\frac{1}{s}}\left\| \chi_{B(x,\sqrt{t})}\sqrt{t}\nabla e^{-t\Delta}\chi_{B(y,\sqrt{t})}\right\|_{r\to s} \lesssim e^{-\frac{d^2(x,y)}{Ct}},\quad \forall x,\,y\in M,\,t>0,
\end{equation}
where $\chi_B$ denotes the characteristic function of the set $B\subset M$. Since the gradient of the heat kernel satisfies the off-diagonal Davies-Gaffney estimates (see \cite[Lemma 3.8]{AMR}), $G^{\mathrm{off}}_{2,2}$ holds. Using the tools developed in \cite{CS}, one can show:

\begin{Pro}\label{grad_off_equiv}

Let $(M,\mu)$ be  a complete, non-compact, connected weighted Riemannian manifold satisfying  \eqref{d} and \eqref{UE}.  Then,

\begin{itemize}

\item[(i)] if $1\leq r\leq s\leq \infty$, then \eqref{grad_r_s} implies \eqref{grad_off}.

\item[(ii)] if $2\leq q'\leq \infty$, then $(G^{\mathrm{off}}_{2,q'})$ implies \eqref{Gp} for any $p \in [2,q']$.

\item[(iii)] if $1\leq q\leq 2$, then $(G^{\mathrm{off}}_{q',2})$ implies \eqref{Gp} for any $p \in [q,2]$.

\end{itemize}
As a consequence, if $q\in [1,2)$ (resp. $q'\in (2,\infty)$), then $(G_{q,2})$ (resp. $(G_{2,q'})$) implies \eqref{Gp} for any $p \in [q,2]$ (resp. $p\in [2,q']$).

\end{Pro}
After presenting these facts, we obtain as a consequence of Theorem \ref{pi_heat}:

\begin{Thm}\label{gradient}

 Let $(M,\mu)$ be  a complete, non-compact, connected weighted Riemannian manifold satisfying  \eqref{d}, \eqref{vol_u} and \eqref{UE}. 
Recall the weighted Ricci tensor $\mathrm{Ric}_\mu:=\mathrm{Ric}-\mathscr{H}_f$, if $\mu=e^f\nu$.
Assume that $(\mathrm{Ric}_{\mu})_-$ satisfies condition~$\eqref{condik}$. Assume moreover that $M$ satisfies \eqref{dd} for some $\kappa>8$, and that for some $q\in [1,2)$, \eqref{Ker} holds for $\mathcal{L}=\vec{\Delta}_\mu$, the Hodge Laplacian on $1$-forms. Then \eqref{grad_r_s} holds for any $q\leq r\leq  s\leq q'$, and

\begin{itemize}

\item if $q>1$, then \eqref{Gp} holds for any $p\in (1,q']$.

\item if $q=1$, then \eqref{G}, the pointwise Gaussian gradient estimate for the heat kernel holds, and \eqref{Gp} holds for any $p\in [1,\infty]$.

\end{itemize}

\end{Thm}
Our approach actually yields results for the heat kernel on forms as well. In this case, one considers

\begin{equation}\label{fGp}\tag{$\vec{G}_{p}$}
\sup_{t>0}\left\| \sqrt{t}(d+d^*) e^{-t\vec{\Delta}_{k,\mu}}\right\|_{p\to p} <+\infty,
\end{equation}

\begin{equation}\label{fgrad_r_s}\tag{$\vec{G}_{r,s}$}
\sup_{t>0}\left\| V_{\sqrt{t}}^{\frac{1}{r}-\frac{1}{s}}\sqrt{t}(d+d^*) e^{-t\vec{\Delta}_{k,\mu}}\right\|_{r\to s} <+\infty,
\end{equation}
and its off-diagonal counterpart: for every $x,y\in M$ and $t>0$,

\begin{equation}\label{fgrad_off}\tag{$\vec{G}^{\mathrm{off}}_{r,s}$}
V(x,\sqrt{t})^{\frac{1}{r}-\frac{1}{s}}\left\| \chi_{B(x,\sqrt{t})}\sqrt{t}(d+d^*) e^{-t\vec{\Delta}_{k,\mu}}\chi_{B(y,\sqrt{t})}\right\|_{r\to s} \lesssim e^{-\frac{d^2(x,y)}{Ct}}.
\end{equation}
According to the off-diagonal Davies-Gaffney estimates from \cite[Lemma 3.8]{AMR}, $\vec{G}^{\mathrm{off}}_{2,2}$ holds. As before, using the tools developed in \cite{CS}, one can show:

\begin{Pro}\label{fgrad_off_equiv}

Let $(M,\mu)$ be  a complete, non-compact, connected weighted Riemannian manifold satisfying  \eqref{d} and \eqref{UE}.  Then,

\begin{itemize}

\item[(i)] if $2\leq  s\leq \infty$ (resp. $1\leq r\leq 2$), then $(\vec{G}_{2,s})$ (resp. $(\vec{G}_{r,2})$) implies $(\vec{G}^{\mathrm{off}}_{2,s})$ (resp. $(\vec{G}^{\mathrm{off}}_{r,2})$).

\item[(ii)] if $2\leq  q'\leq \infty$, then $(\vec{G}^{\mathrm{off}}_{2,q'})$ implies \eqref{fGp} for any $p \in [2,q']$.

\item[(iii)] if $1\leq q\leq  2$, then $(\vec{G}^{\mathrm{off}}_{q',2})$ implies \eqref{fGp} for any $p \in [q,2]$.

\end{itemize}

\end{Pro}

\begin{Rem}\label{fgrad_off_equiv2}
{\em 
A result similar to Proposition \ref{fgrad_off_equiv} also holds for the exact part $\sqrt{t}de^{-t\vec{\Delta}_{k,\mu}}$ or the co-exact part $\sqrt{t}d^*e^{-t\vec{\Delta}_{k,\mu}}$ alone. That is, for instance, if $q<2$, $L^q\to L^2$ estimates for $\sqrt{t}de^{-t\vec{\Delta}_{k,\mu}}$ imply their off-diagonal counterpart, which in turn imply that $\sqrt{t}de^{-t\vec{\Delta}_{k,\mu}}$ is uniformly bounded on $L^p$, $p\in [q,2]$. The proof is identical to that of Proposition \ref{fgrad_off_equiv} and details are left to the reader.
}
\end{Rem}
Here is a form analog of Theorem \ref{gradient}:

\begin{Thm}\label{forms}

 Let $(M,\mu)$ be  a complete, non-compact, connected weighted Riemannian manifold satisfying  \eqref{d}, \eqref{vol_u}, \eqref{UE}, and \eqref{dd} for some $\kappa>8$. Recall the  tensor $\mathscr{R}_{k,\mu}$ from the Bochner formula on $k$-forms, and let $k\geq1$. Assume that $\mathscr{R}_{k+1,\mu}$, $\mathscr{R}_{k,\mu}$ and $\mathscr{R}_{k-1,\mu}$ all satisfy condition $\eqref{condik}$, and that, for some $q\in [1,2)$, \eqref{Ker} holds for the (weighted) Hodge Laplacian acting on $k-1$, $k$ and $k+1$-forms. Then, for every $q\leq r\leq s\leq q'$, \eqref{fgrad_r_s} for $\vec{\Delta}_{k,\mu}$ holds. In particular, \eqref{fGp} for $\vec{\Delta}_{k,\mu}$ holds for all $p\in [q,q']$.

\end{Thm}

\begin{Rem}
{\em
Variants of Theorem \ref{forms} holds for the operators $\sqrt{t}de^{-t\vec{\Delta}_{k,\mu}}$, or $\sqrt{t}d^*e^{-t\vec{\Delta}_{k,\mu}}$ alone. Also, one can obtain a version just for the interval $[q,2]$, or $[2,q']$. The details are left to the interested reader.
}
\end{Rem}
It will be interesting to understand how the estimates \eqref{fGp} of Theorem \ref{forms} are related to the boundedness of the Riesz transform on $k$-forms $(d+d^*)\vec{\Delta}_{k,\mu}^{-1/2}(I-\Pi_k)$. Contrary to the $L^p$ gradient estimates \eqref{Gp}, which can be deduced from the $L^p$ boundedness of the Riesz transform $d\Delta^{-1/2}$ using the analyticity of the heat semigroup $e^{-t\Delta}$, the heat kernel estimates of Theorem \ref{forms} do not follow directly from the $L^q$ and $L^{q'}$ boundedness of the Riesz transform on $k$-forms. Indeed, the analyticity of $e^{-t\vec{\Delta}_{k,\mu}}$ on $L^q$ or $L^{q'}$ is not known in the setting of Theorem \ref{forms}.

\bigskip

Finally, in the last part of the article, we discuss the relevance of our assumption \eqref{Ker}. We study its validity under additional geometric assumptions on $M$. The first such result is:

\begin{Thm}\label{Ker-Sob}

Let $M$ be an $n$-dimensional Riemannian manifold which satisfies \eqref{Sob}, the Sobolev inequality with parameter $n>4$. Assume that the negative part of the Ricci curvature $\mathrm{Ric}_-$ is in $L^{\frac{n}{2}\pm\varepsilon}$ for some $\varepsilon>0$. Then, \eqref{Ker} for $\mathcal{L}=\vec{\Delta}$, the Hodge Laplacian on $1$-forms, holds for all $q\in (\frac{n}{n-1},2)$.

\end{Thm}
This improves the result obtained in \cite[Proposition 11]{D2} in the case $\mathcal{L}=\vec{\Delta}$, the Hodge Laplacian on $1$-forms. Actually, a similar result holds for the Hodge Laplacian on $k$-forms:

\begin{Thm}\label{Ker-Sob_k}

Let $M$ be an $n$-dimensional Riemannian manifold which satisfies \eqref{Sob}, the Sobolev inequality with parameter $n>4$. Let $k\in[1, n-1]$ be an integer, and recall the tensor $\mathscr{R}_{k}$ from the Bochner formula of the Hodge Laplacian $\vec{\Delta}_{k}$ on $k$-forms. Assume that $\mathscr{R}_{k}$ is in $ L^{\frac{n}{2}\pm\varepsilon}$ for some $\varepsilon>0$. Assume also that

$$V(x,r)\simeq r^n,\,\quad \forall x\in M,\,r\geq1.$$
Then, \eqref{Ker} for $\mathcal{L}=\vec{\Delta}_{k}$ holds for all $q\in (q^*,2)$, where

$$q^*=\left\{\begin{array}{lcl}\frac{n(n-k-1)}{(n-2)(n-k)},\quad 1\leq k\leq \frac{n}{2}\\\\
\frac{n(k-1)}{k(n-2)},\quad \frac{n}{2}\leq k\leq n-1
\end{array}\right.$$

\end{Thm}
This result, together with THeorem \ref{pi_heat}, has consequence for the heat semi-group on $k$-forms. For $k=1,\ldots,n-1$, define the interval $I_k$ by

$$I_k=\left\{\begin{array}{lcl}\left(\frac{n(n-k-1)}{(n-2)(n-k)},\frac{n(n-k-1)}{n-2k}\right),\quad 1\leq k\leq \frac{n}{2}\\\\
\left(\frac{n(k-1)}{k(n-2)},\frac{n(k-1)}{2k-n}\right),\quad \frac{n}{2}\leq k\leq n-1
\end{array}\right.$$
and let also $I_0=I_n=[1,+\infty]$. Then, one has:

\begin{Cor}\label{unif_bdd2}

Let $M$ be an $n$-dimensional Riemannian manifold which satisfies \eqref{Sob}, the Sobolev inequality with parameter $n>8$. Recall the tensor $\mathscr{R}_{k}$ from the Bochner formula of the Hodge Laplacian $\vec{\Delta}_{k}$ on $k$-forms. Assume that for every integer $k\in [1,n-1]$, $\mathscr{R}_{k}$ is in $ L^{\frac{n}{2}\pm\varepsilon}$ for some $\varepsilon>0$. Assume also that

$$V(x,r)\simeq r^n,\,\quad \forall x\in M,\,r\geq1.$$
Then, for every $k=1,\cdots,n-1$, the heat semi-group $e^{-t\vec{\Delta}_k}$ of the Hodge Laplacian on $k$-forms is uniformly bounded on $L^p$, for every $p\in I_k$, and moreover, for all $r,\,s\in I_{k-1}\cap I_k\cap I_{k+1}$ such that $ r\leq s$,
 
$$\sup_{t>0}V_{\sqrt{t}}^{\frac{1}{r}-\frac{1}{s}}||\sqrt{t}(d+d^*)e^{-t\vec{\Delta}_{k,\mu}}||_{r\to s}<+\infty.$$

\end{Cor}
Corollary \ref{unif_bdd2} improves (under stronger geometric conditions) results obtained in \cite{MO}. As a consequence of Theorems \ref{Ker-Sob} and \ref{gradient}, one obtains:

\begin{Cor}\label{grad-Sob}

Let $M$ be an $n$-dimensional Riemannian manifold which satisfies \eqref{Sob}, the Sobolev inequality of exponent $n>8$. Assume that the negative part of the Ricci curvature $\mathrm{Ric}_-$ is in $ L^{\frac{n}{2}\pm\varepsilon}$ for some $\varepsilon>0$. Then, the $L^p$ gradient estimate for the heat kernel \eqref{Gp} holds for all $2\leq p< n$.

\end{Cor}

\begin{Rem}
{\em
In the setting of Theorem \ref{Ker-Sob} and Corollary \ref{grad-Sob}, it follows from \cite{C8} that one has the volume estimates:

$$V(x,r)\simeq r^n,\quad\forall x\in M,\,r>0.$$
}
\end{Rem}

\begin{Rem}
{\em 
If one assumes that the Ricci curvature is bounded from below, the result of Corollary \ref{grad-Sob} is actually already known, even under the weaker assumption $n>3$: indeed, this follows from the fact that under the assumptions of Corollary \ref{grad-Sob}, the Riesz transform $d\Delta^{-1/2}$ is bounded on $L^p$ for all $p\in (1,n)$ (see \cite[Theorem 16]{D2}; the required volume estimate follows from the recent \cite{C8}).
}
\end{Rem}

\begin{Rem}\label{grad_con_sum}
{\em 

As the example of the connected sum of Euclidean spaces demonstrates, the result of Corollary \ref{grad-Sob} is essentially optimal. Indeed, on $\R^n\sharp\R^n$, \eqref{Gp} cannot hold for $p>n$. The reason is that if \eqref{Gp} were true for some $p>n$, then by using Morrey inequalities, one could prove Gaussian heat kernel lower bounds, which are known to be false on $\R^n\sharp\R^n$ (for details, see the arguments in \cite[Section 5]{CD1} as well as \cite{D4}).

}
\end{Rem}
Under stronger assumptions on the geometry at infinity of the manifolds, the results of Corollary \ref{grad-Sob} can be improved. In conclusion of the article, we study the validity of \eqref{Ker} for $\mathcal{L}=\vec{\Delta}$, the Hodge Laplacian on $1$-forms, on manifolds that have a special cone structure at infinity. On such manifolds, there is a hope to find the optimal interval of $q$'s in $[1,2)$ such that \eqref{Ker} for $\mathcal{L}=\vec{\Delta}_k$ holds; indeed, harmonic $k$-forms on an conical manifold have an relatively simple description (see \cite{Che}). One expects that the answer in this case depends on the spectrum of the Hodge Laplacian of the basis of the cone. In \cite{GS}, the analysis of the behaviour at infinity of $L^2$ harmonic $k$-forms has been partially made, and as a consequence of the results of \cite{GS} (see especially Corollary 9 therein), in the case of $1$-forms it is possible to obtain the following result: assume that $M$ is asymptotically Euclidean (to high enough order), then one of the three following possibilities occurs: 

\begin{enumerate}

\item[(i)] \eqref{Ker} for $\mathcal{L}=\vec{\Delta}$ holds for all $q\in (1,2)$.

\item[(ii)] \eqref{Ker} for $\mathcal{L}=\vec{\Delta}$ holds for all $q\in (\frac{n}{n-1},2)$.

\item[(iii)] \eqref{Ker} for $\mathcal{L}=\vec{\Delta}$ holds for all $q\in (\frac{n}{n-2},2)$.

\end{enumerate}
However, there is a catch here: to determine whether one is in situation (i), (ii) or (iii), one needs to prove the existence or non-existence of $L^2$ harmonic $1$-forms of certain special type, which is not done in \cite{GS}. The result of Theorem \ref{Ker-Sob} shows that for $1$-forms, one is always in the case (i) or (ii). Our next result (Theorem \ref{main-cone}) implies in particular that on a manifold that is Euclidean outside a compact set, (i) occurs if and only if there is only one end. The case of forms of higher degree is more involved, and we leave it for future work. After this short discussion, let us introduce the setting for our last result.

\bigskip

If $(X,\bar{g})$ is a compact Riemannian manifold, one can consider $\mathcal{C}(X)$, the cone over $X$, defined by

$$\mathcal{C}(X)=(0,\infty)\times X,$$
endowed with the metric 

$$g=dr^2+r^2\bar{g}.$$
Let $k\geq1$ be an integer, and let $X_1,\ldots,X_k$ be compact Riemannian manifold having all the same dimension. Let $M$ be a complete Riemannian manifold. If there exists $U\subset M$ and $V_1\subset \mathcal{C}(X_1),\ldots,V_k\subset \mathcal{C}(X_k)$ open sets such that $M\setminus U$ is isometric to the disjoint union of the manifolds $\mathcal{C}(X_k)\setminus V_k$, we will say that $M$ is a {\em conical manifold at infinity}, with $k$ ends. In this case, we will write

$$M\simeq_\infty \sqcup_{i=1}^k\mathcal{C}(X_i).$$
It is important to note that for $n>2$, if $M$ is an $n$-dimensional manifold which is conical at infinity, then the Sobolev inequality with parameter $n$ holds. Furthermore, if $M$ has only one end, then the scaled $L^2$ Poincar\'e inequalities \eqref{P}, hold on $M$. See e.g. \cite[Section 7.1]{C1} and the references therein. Also, a direct computation shows that if $n$ is the dimension of $M\simeq_\infty\mathcal{C}(X)$, and if the Ricci curvature on $X$ has the lower bound

\begin{equation}\label{low_Ric}
\mathrm{Ric}_X\geq (n-2)\bar{g},
\end{equation}
then $M$ has {\em non-negative} Ricci curvature. Our study of the behaviour of harmonic $1$-forms on such manifolds will lead to the following result:

\begin{Thm}\label{main-cone}
Let $n>8$, and let $M\simeq_\infty \mathcal{C}(X)$ be an $n$-dimensional manifold that is conical at infinity with only one end. Assume that the Ricci curvature on $X$ has the lower bound \eqref{low_Ric}. Then,

\begin{itemize}

\item[(i)] if $X=S^{n-1}$, i.e. $M$ is {\em Euclidean} outside a compact set, then for every $p\in (1,\infty)$, the semi-group $e^{-t\vec{\Delta}}$ is bounded on $L^p$ and the gradient estimate \eqref{Gp} holds.

\item[(ii)] if $X\neq S^{n-1}$, then for every $p\in [1,\infty]$, the semi-group $e^{-t\vec{\Delta}}$ is bounded on $L^p$, and the Gaussian gradient estimate for the heat kernel \eqref{G} holds.

\end{itemize}

\end{Thm}

\begin{Rem}
{\em 

The gradient $L^p$ estimate \eqref{Gp} in (i) is already known, since on $M$ the Riesz transform is bounded on $L^p$ for all $p\in (1,\infty)$ (see \cite{D1}). However, our proof is the first direct proof, i.e. not using the Riesz transform. The $L^p$ boundedness of the heat semi-group on $1$-forms, as well as (ii), are entirely new.

}
\end{Rem}
In light of the result of Theorem \ref{main-cone}, an interesting open problem is the following:

\begin{Prob}
{\em 
Let $M$ be Euclidean at infinity, with only one end; does the Gaussian gradient estimate \eqref{G} hold on $M$?
}
\end{Prob}

\subsection{Comparison with existing results in the literature}

In this paragraph, we present a more detailed account on gradient estimates for the heat kernel that are available in the literature, and we compare them with our own results. We feel such a review is needed, because it is not so easy to extract the relevant results from the literature, some results being written under different assumptions, or in the compact setting. It will also put our own results in perspective. First, let us come back to the Li-Yau gradient estimate \eqref{Grad-LY} in the non-negative Ricci curvature setting. The idea of the proof of inequality \eqref{Grad-LY} is to compute $(\Delta+\partial_t)(\varphi G)$, with $G=\alpha\frac{\nabla u|^2}{u^2}-\frac{u_t}{u}$ for some $\alpha\in (0,1)$, $\varphi$ being an approriate (radial) cut-off function equal to $1$ on $B(x_0,r)$ and supported in $B(x_0,2r)$, and use a clever maximum principle technique to estimate then $G$ on $B(x_0,r)$. The radius $r$ is then sent to infinity, which provides a global estimate for $G$. The same approach has later been used first in \cite{ZZ}, \cite{R2} and \cite[Section 3]{C8}. In all these works, the idea of the proof is the same as Li and Yau: one computes $(\Delta+\partial_t)(\varphi G)$, where now $G$ is defined as $G=\alpha J\frac{\nabla u|^2}{u^2}-\frac{u_t}{u}$ for some $\alpha\in (0,1)$, and $\varphi$ is an appropriate cut-off function, then one uses the maximum principle technique to estimate $G$. The function $J=J(x,t)$ is a carefully chosen function, whose purpose is to balance the error terms introduced by the fact that the Ricci curvature is not globally non-negative. The cut-off function usually localizes the estimate in a ball $B(x_0,r)$ where the negative part of the Ricci curvature is small (in an integral sense); it has support in $B(x_0,2r)$, is equal to $1$ on $B(x_0,r)$, and should satisfy $|\nabla\varphi|^2+|\Delta\varphi|\lesssim r^2$. The negative part of the Ricci curvature is assumed to be small uniformly (in an integral sense), at the scale $r$, more precisely in \cite{ZZ} it is assumed that the following quantity is small enough:

$$\kappa(p,r)=r^2\sup_{x\in M}\left(\frac{1}{V(x,r)}\int_{B(x,r)}||\mathrm{Ric}_-(y)||^p\,dy\right)^{1/p},\,p>n/2.$$
The smallness assumption on $\mathrm{Ric}_-$ enters crucially at two places in the proof: first, to guarantee the existence of a good cut-off function $\varphi$ as above for any $x_0\in M$, and second, to control the function $J$. Actually, the existence of such cut-off functions at all scales is a strong assumption on $M$; it holds for example if the negative part of the Ricci curvature has quadratic decay, as follows from standard comparison theorems. In the first paper \cite{ZZ}, the obtained control on $J$ deteriorates as $t\to\infty$: more precisely, it is proved that 

$$\underline{j}_r(t)\leq J\leq 1,$$
where $\underline{j}_r(t)$ depends on $r$ (the radius of the considered ball), $t$, and $\kappa(p,r)$. Explicitly,

$$\underline{j}_r(t)=C\exp(-a\kappa r^2(1+b\kappa^{\frac{1}{2p-n}})t),$$
where $C$, $a$ and $b$ are positive constants. The obtained Li-Yau gradient inequality writes:

$$\alpha \underline{j}_r(t)\frac{|\nabla u|^2}{u^2}-\frac{u_t}{u}\leq \frac{\beta}{\underline{j}_r(t)}\left(\frac{1}{t}+\frac{1}{\underline{j}_r(t)r^2}\right).$$
However, 

$$\lim_{t\to\infty}\underline{j}_r(t)=0,$$
which means that the control we get on $G$ worsens as $\to\infty$. Also, contrary to $\kappa(\frac{n}{2},r)$, the quantity $\kappa(p,r)$ for $p>\frac{n}{2}$ is not scale invariant. Actually, assuming that $V(x,r)\simeq r^n$ for $r\geq1$, one has

$$\lim_{r\to \infty}\kappa(p,r)=+\infty,$$
unless $\mathrm{Ric}_-\equiv0$. Therefore, one cannot let $r\to \infty$ and get rid of the $r^{-2}$ term in the Li-Yau gradient inequality, contrary to what happens in the non-negative Ricci curvature setting.

It has been noticed by G. Carron in \cite[Section 3]{C8} that the control on $J$ could be improved. In fact, he obtains that $J$ is bounded above and below by two positive constants independently of $t$, under an approriate smallness assumption of $\mathrm{Ric}_-$. His observation leads to the following Li-Yau gradient estimate, which is not written down explicitely in \cite{C8} but can be shown by using the ideas in \cite[Section 3]{C8} (see in particular Proposition 3.17 therein): 

\begin{Pro}\label{Carron}

Let $M$ be a non-parabolic, complete Riemannian manifold of dimension $n$, $x_0\in M$, $r>0$, and assume that the negative part of the Ricci is small in an integral sense on $B(x_0,2r)$:

\begin{equation}\label{LSL}
\sup_{x\in B(x_0,2r)}\int_{B(x_0,2r)}G(x,y)||\mathrm{Ric}_-(y)||\,dy<\frac{1}{16n},
\end{equation}
where $G(x,y)$ is the positive minimal Green function on $M$. Assume also the existence of a good cut-off function, that is $\varphi\in C_0^\infty(B(x_0,2r))$ with $\varphi|_{B(x_0,r)}\equiv 1$ and $|\nabla \varphi|^2+|\Delta\varphi|\lesssim \frac{1}{r^2}$. Then, there exists positive constants $\alpha$, $\beta$ such that the following Li-Yau gradient estimate in $B(x_0,r)$, for $u$ positive solution of the heat equation, holds:

\begin{equation}\label{Grad-LY3}
\alpha \frac{|\nabla u|^2}{u^2}-\frac{u_t}{u}\leq \beta \left(\frac{1}{t}+\frac{1}{r^2}\right),\quad t\in (0,\infty).
\end{equation}

\end{Pro}
It is important to note that the gradient estimate \eqref{Grad-LY3} actually holds also on a manifold whose negative part of the Ricci curvature has a quadratic decay, as follows from the original Li and Yau result (\cite[Theorem 1.2]{LY}). Therefore, in some sense no better inequality is obtained under a local smallness of the Ricci curvature such as \eqref{LSL}. This is in contrast with the compact setting, in which a much better result is obtained under \eqref{LSL}. The issue with the non-compact setting is that one has to use a cut-off function to localize the argument and be able to apply the maximum principle technique and estimate $G$, while in the compact setting one can apply the maximum principle directly. As a consequence of the Li-Yau gradient estimate \eqref{Grad-LY3}, one obtains the following estimates for the heat kernel and its gradient (see \cite[Theorem 3.5]{C8} and \cite[Theorem 5.10]{R3}):

 \begin{equation}
p_{t}(x,y)\lesssim
\frac{1}{V(x,\sqrt{t})}\exp
\left(-\frac{d^{2}(x,y)}{Ct}\right), \quad \forall~t\in(0,r^2),\, x,y\in
 B(x_0,r),\label{LY3}
\end{equation}
and

\begin{equation}\label{G3}
|\nabla_x p_t(x,y)|\lesssim \frac{1}{\sqrt{t}V(x,\sqrt{t})}e^{-c\frac{d^2(x,y)}{t}},\quad \forall~t\in(0,r^2),\, x,y\in
 B(x_0,r).
\end{equation}
On the other hand, if one knows {\em a priori} that the heat kernel has Gaussian upper estimates for all times, i.e. \eqref{UE} holds, then using \eqref{Grad-LY3} one can improve \eqref{G3} into

\begin{equation}\label{G4}
|\nabla_x p_t(x,y)|\lesssim \max\left(\frac{1}{\sqrt{t}},\frac{1}{r}\right)\frac{1}{V(x,\sqrt{t})}e^{-c\frac{d^2(x,y)}{t}},\quad \forall~t>0,\, x,y\in
 B(x_0,r).
\end{equation}
One expects that the gradient estimate \eqref{G4} is close to be optimal, without further assumptions on the topology of $M$: indeed, it has been observed in \cite[Proposition 6.1]{CCH} that on the so-called connected sum of two Euclidean space, which is flat outside a compact set, as $t\to+\infty$ the correct polynomial time decay for the gradient of the heat kernel is $t^{-n/2}$ and not $t^{-\frac{n+1}{2}}$, as in the Euclidean space itself. Hence, taking the gradient does not improve the time decay at all, for large times! This is an intrinsic limitation of the method, as the assumption made on the negative part of the Ricci curvature does not permit to discriminate between a manifold such as the connected sum of two copies of $\R^n$ (on which the gradient of the heat kernel has the same decay in time as the heat kernel itself) and between a manifold isometric to a single copy of $\R^n$ outside a compact set (for which some extra decay in time of the gradient of the heat kernel is expected). 

An advantage of the result of Proposition \ref{Carron} over \cite[Theorem 1.1]{ZZ} is that from Proposition \ref{Carron}, one can get also a result for large time, under a global size bound for the negative part of the Ricci curvature. Indeed, if one has

\begin{equation}\label{GSL1}
\sup_{x\in M}\int_{M}G(x,y)||\mathrm{Ric}_-(y)||\,d\mu(y)<\frac{1}{16n},
\end{equation}
then provided that one can get good cut-off functions at all scales $r$, one can let $r\to\infty$ in \eqref{Grad-LY3} and obtain

\begin{equation}\label{Grad-LY4}
\alpha \frac{|\nabla u|^2}{u^2}-\frac{u_t}{u}\leq \beta \frac{1}{t},\quad t\in (0,\infty),
\end{equation}
which implies both \eqref{UE} and \eqref{G} for all times. Actually, this result goes back to \cite[Theorem 3.1]{CZ}, where \eqref{G} is proved under a non-explicit size bound of $\mathrm{Ric}_-$, without assuming the existence of good cut-off functions but assuming that \eqref{UE} and \eqref{d} hold. However the proof in \cite[Theorem 3.1]{CZ} uses the heat kernel on $1$-forms, and not a Li-Yau gradient inequality. As a consequence of the results in \cite{CDS}, one actually sees that \eqref{G} follows from the following weaker global size bound on $\mathrm{Ric}_-$:

\begin{equation}\label{GSL2}
\sup_{x\in M}\int_{M}G(x,y)||\mathrm{Ric}_-(y)||\,d\mu(y)<1.
\end{equation}
 However let us stress again that such a result is not satisfactory, since it does not allow one to determine whether \eqref{G} holds or not even on simple manifolds, such as asymptotically Euclidean ones. 
 
As a conclusive remark for this paragraph, let us compare the result obtained in our Theorem \ref{main_intro}, with the above-mentioned results from the literature. First, the manifolds considered in Theorem \ref{main_intro} are conical at infinity, hence they have good cut-off functions as needed in Proposition \ref{Carron}, at all scales. Also, on these manifolds, \eqref{UE} and \eqref{d} hold (see \cite[Section 2]{C1}). The Ricci curvature is non-negative outside a compact set $K\Subset M$, hence the local size limitation \eqref{LSL} is satisfied provided the ball $B(x_0,r)$ is included in $M\setminus K$, but the global size limitation \eqref{GSL1} or \eqref{GSL2} are not necessarily satisfied. For a point $x\in M$, denote by $r(x)$ the distance from $x$ to $K$. As mentioned above, on any asymptotically conical manifold, the fact that the decay of the Ricci curvature is quadratic implies by Li and Yau's original result \cite[Theorem 1.2]{LY} that

\begin{equation}\label{quad_grad}
\begin{array}{rcl}
|\nabla_xp_t(x,y)|&\leq& \max\left(\frac{1}{\sqrt{t}},\frac{1}{r(x)+1}\right)\frac{1}{V(x,\sqrt{t})}e^{-c\frac{d^2(x,y)}{t}},\\\\
&& \quad\forall~t>0,\, x,y\in M.
\end{array}
\end{equation}
This estimate can also be obtained as a consequence of Proposition \ref{Carron}. A straightforward computation then shows that \eqref{quad_grad} implies $||\nabla e^{-t\Delta}||_{p\to p}\lesssim t^{-1/2}$ only for $p<n$. This result is almost optimal if $M$ has several ends, as $||\nabla e^{-t\Delta}||_{p\to p}\lesssim t^{-1/2}$ cannot hold for $p>n$. However, it is far from being optimal if $M$ has one end: as a consequence of results from \cite{GH} concerning the Riesz transform, assuming that $M$ is isometric to the cone $\mathcal{C}(X)$ outside a compact set, it is known that $||\nabla e^{-t\Delta}||_{p\to p}\lesssim t^{-1/2}$ for all $p\in (1,p^*)$, with $p^*>n$ being defined in terms of $\lambda_1(X)$, the first eigenvalue of the cross-section of the cone. Explicitly, if $ \sqrt{\left(\frac{n-2}{2}\right)^2+\lambda_1(X)}<\frac{n}{2}$ then

$$p^*=\frac{n}{\frac{n}{2}-\sqrt{\left(\frac{n-2}{2}\right)^2+\lambda_1(X)}}>n,$$
and if $\sqrt{\left(\frac{n-2}{2}\right)^2+\lambda_1(X)}\geq\frac{n}{2}$ then $p^*=+\infty$. Comparatively, in our Theorem \ref{main_intro}, we improve this result and manage to show that the full pointwise Gaussian estimate gradient estimate \eqref{G} holds when $X$ has Ricci curvature bounded from below by $n-2$ and is not isometric to the Euclidean sphere $S^{n-1}$. 

Finally, for the sake of completeness, let us refer to \cite[Section 5]{R3} for more details and reference concerning Li-Yau gradient estimates in the Zhang and Zhu setting, and their consequences for heat kernels (see in particular Theorem 5.9 therein).

\subsection{Strategy for the proof of Theorem \ref{pi_heat}}

The proof is based on the approach developped in \cite{CDS}, which we briefly recall now. We strongly advise the reader to get familiar with \cite{CDS} before reading the full proof of Theorem \ref{pi_heat}, as similar ideas and notations will be used.

We consider the following decomposition of the operator $\mathcal{L}$: we write $\mathcal{L}$  as the rough Laplace operator $\bar{\Delta}$, plus $\mathcal{R}_+$, minus $\mathcal{R}_-$ outside a compact set (which is small in some sense thanks to condition $(K)$) perturbed by a compactly supported part of $\mathcal{R}_-$.

More precisely, let  $K_0$ be given by condition $(K)$. Let $\mathcal{W}_0$ and $\mathcal{W}_\infty$ be the sections of the vector bundle $\mathrm{End}(E)$ respectively given by $$x\to \mathcal{W}_0(x)=\mathbf{1}_{K_0}(x)\mathcal{R}_-(x)$$
and
$$x\to \mathcal{W}_\infty(x)=\mathbf{1}_{M\setminus {K_0}}\mathcal{R}_-(x).$$  We shall also denote by $\mathcal{W}_0$  and $\mathcal{W}_\infty$ the associated operators on sections of $E$.  Set
\begin{equation}\label{defh}
\mathcal{H}=\bar{\Delta}+\mathcal{R}_+-\mathcal{W}_\infty,
\end{equation}
so that 
$$\mathcal{L}=\mathcal{H}-\mathcal{W}_0.$$
That is, $\mathcal{L}$ can be seen as $\mathcal{H}$, perturbed by the compactly supported $\mathcal{W}_0$. According to \cite{CDS}, $e^{-t\mathcal{H}}$ has Gaussian estimates, and in \cite{CDS} it is proved that if the perturbation $\mathcal{W}_0$ substracted from $\mathcal{H}$ is ``subcritical'' --which was also shown to be equivalent to $\mathrm{Ker}_{L^2}(\mathcal{L})=\{0\}$ if $\kappa$, the exponent from \eqref{dd}, is greater than $4$--, then the operator $e^{-t\mathcal{L}}$ also had Gaussian estimates. The way the Gaussian estimates for $e^{-t\mathcal{L}}$ are proved in \cite{CDS} is through resolvent estimates: by a general functional analytic principle, in the context of \cite[Theorem 1.1]{CDS}, Gaussian estimates for $e^{-t\mathcal{L}}$ are equivalent to the resolvent estimates

\begin{equation}\label{resol1}
\sup_{t>0}||(I+t\mathcal{L})^{-1}V_{\sqrt{t}}^{1/p}||_{p\to\infty}<+\infty,\quad \forall \hat{p}\leq p<+\infty.
\end{equation}
In order to prove the resolvent estimates \eqref{resol1}, a perturbation formula for the resolvent is used:

\begin{equation}\label{perturb1}
(I+t\mathcal{L})^{-1}=(I-(I+t\mathcal{H})^{-1}t\mathcal{W}_0)^{-1}(I+t\mathcal{H})^{-1}.
\end{equation}
Then, it is shown that under the assumption $\mathrm{Ker}_{L^2}(\mathcal{L})=\{0\}$ (for $\kappa>4$), the operator $(I-(I+t\mathcal{H})^{-1}t\mathcal{W}_0)^{-1}$, defined as a Neumann series, is a uniformly bounded (in $t$) operator on $L^\infty$. This relies on the fact that for every $\lambda\geq0$, the operator $(\mathcal{H}+\lambda)^{-1}\mathcal{W}_0$ has a spectral radius bounded by $1-\varepsilon$, $\varepsilon >0$. Let us also mention that a crucial technical ingredient for this spectral radius estimate is that the operator $(\mathcal{H}+\lambda)^{-1}\mathcal{W}_0$ is compact on $L^\infty$.

\bigskip

Let us now present the main steps for the proof of Theorem \ref{pi_heat}. We will prove that the $L^r\to L^r$ estimates for the kernel of $e^{-t\mathcal{L}}(I-\Pi)$ follow from similar estimates for $e^{-t(\mathcal{L}+\alpha \Pi)}$, where $\alpha>0$ is a constant whose value will be chosen small enough later. The point is that by adding $\alpha \Pi$, we have gained positivity: indeed, the operator $\mathcal{L}+\alpha \Pi$ has a trivial $L^2$ kernel. Now, by a general functional-analytic principle similar to \cite[Theorem 2.2]{CDS}, in order to get estimates for the heat operator  $e^{-t(\mathcal{L}+\alpha \Pi)}$, \eqref{VPiV} follows from the following estimates for the resolvent of $\mathcal{L}+\alpha\Pi$:

\begin{equation}\label{resol2}
\sup_{t>0}||(I+t(\mathcal{L}+\alpha\Pi))^{-1}V_{\sqrt{t}}^{1/p}||_{p\to\infty}<+\infty,\quad \forall \hat{p}\leq p<+\infty.
\end{equation}
As in \cite{CDS}, in order to obtain \eqref{resol2}, we write a perturbation formula for the resolvent:

\begin{equation}\label{perturb2}
(I+t\mathcal{L})^{-1}=(I-(I+t\mathcal{H})^{-1}t(\mathcal{W}_0+\alpha \Pi))^{-1}(I+t\mathcal{H})^{-1}.
\end{equation}
Notice that $\Pi$ enters in the perturbation part now! The main step is then to prove the uniform boundedness on $L^s$, $2\leq s\leq q'$, of the operator $(I-(I+t\mathcal{H})^{-1}t(\mathcal{W}_0+\alpha \Pi))^{-1}$. This will follow from the fact that the spectral radius on $L^s$ of $(\mathcal{H}+\lambda)^{-1}(\mathcal{W}_0+\alpha \Pi)$ is bounded (uniformly in $\lambda\geq0$) by $1-\varepsilon$, $\varepsilon>0$. As before, a crucial ingredient is the compactness on $L^s$ of the operator $(\mathcal{H}+\lambda)^{-1}(\mathcal{W}_0+\alpha \Pi)$.

\subsection{Plan of the paper}

In Section 2, we prove Theorem \ref{pi_heat}. Section 3 is devoted to the proof of Theorems \ref{gradient} and \ref{forms}. In Section 4, we relate the condition \eqref{Ker} to the $L^2$ cohomology, and discuss its significance; Theorem \ref{Ker-Sob} is proved. Finally, in Section 5, we study the special case of manifolds that are locally Euclidean outside a compact set, and prove Theorem \ref{main-cone}.

\section{Proof of Theorem \ref{pi_heat}}

First, we recall some notations from \cite{CDS}. Introduce $Q_\mathcal{H}$, the quadratic form associated to $\mathcal{H}$:

$$Q_\mathcal{H}(\omega)=\int_M |\nabla \omega|^2+<\mathcal{R}_+\omega,\omega>-\int_{M\setminus K_0}\left(\mathcal{R}_-\omega,\omega\right),$$
and denote by $H_0^1$ the completion of $C_0^\infty(E)$ for the norm $\|\omega\|_{Q_\mathcal{H}}^2=Q_\mathcal{H}(\omega)$. Note that the space $H_0^1$ depends on $\mathcal{H}$, that is on $\mathcal{L}$ and on $K_0$. According to \cite[Section 3]{CDS}, there exists a smooth, positive function $\rho$ on $M$ such that

\begin{equation}\label{vNP}
\int_M\rho |\omega|^2\leq Q_\mathcal{H}(\omega),\quad \forall \omega\in C_0^\infty(E).
\end{equation}
In particular, this shows that $H_0^1$ injects into $L^2_{loc}$ so it is a function space. Note that as consequence of \eqref{vNP}, $\mathrm{Ker}_{L^2}(\mathcal{H})=\{0\}$. Indeed, by self-adjointness, every element $\omega$ of $\mathrm{Ker}_{L^2}(\mathcal{H})$ lies in the domain of the quadratic form $Q_\mathcal{H}$, and satisfy $Q_\mathcal{H}(\omega)=0$, which by \eqref{vNP} implies that $\omega\equiv0$. This allows us to define operators $\mathcal{H}^{-\alpha}$, $\alpha>0$, by means of the spectral theorem: these are defined as $f(\mathcal{H})$, where $f(x)=x^{-\alpha}$ for $x>0$, and $f(0)=0$; since $\mathrm{Ker}_{L^2}(\mathcal{H})=\{0\}$, the spectral measure does not charge $0$, and so the value of $f$ at $0$ does not matter. Equivalently, one can use the heat kernel to define $\mathcal{H}^{-\alpha}$:

$$\mathcal{H}^{-\alpha}=\frac{1}{\Gamma(\alpha)}\int_0^\infty t^{\alpha-1}e^{-t\mathcal{H}}\,dt.$$
In the case  $\alpha=\frac{1}{2}$, there is  yet another equivalent way to define $\mathcal{H}^{-1/2}$: the operator $\mathcal{H}^{-1/2}$, defined by the spectral theorem, is an isometric embedding from $C_0^\infty$ endowed with the $L^2$ norm, to $H_0^1$, and thus extends by density to a bounded operator from $L^2$ to $H_0^1$. It turns out that this operator is a bijective isometry from $L^2$ to $H_0^1$, and on $H_0^1\cap L^2$ it coincides with the operator $\mathcal{H}^{-1/2}$ defined by the spectral theorem. See \cite[Section 3]{D0} for more details; the proofs are written for the scalar Laplacian, but they easily adapt to our present context.

Note that under condition \eqref{vol1} (which we recall under assumptions \eqref{d} and \eqref{UE} holds if and only if $M$ is non-parabolic), the operator $\mathcal{H}^{-1}$ has a well-defined, finite kernel outside of the diagonal: indeed, using the fact that the heat kernel of $\mathcal{H}$ has Gaussian estimates, and the formula
$$\mathcal{H}^{-1}=\int_0^{+\infty} e^{-t\mathcal{H}}\,dt,$$
it follows that for every $x\neq y$,
$$||\mathcal{H}^{-1}(x,y)||_{y,x}\leq C\int_{d^2(x,y)}^{+\infty} \frac{dt}{V(x,\sqrt{t})}<+\infty.$$
A consequence of \eqref{Ker} is the following:

\begin{Lem}\label{Lp-spaces}

Assume that $(M,d,\mu)$ is non-parabolic and that \eqref{UE}, \eqref{d} and \eqref{dd} with $\kappa>4$ hold. Assume also that \eqref{Ker} holds for some $q\in [1,2)$. Then, $\mathrm{Ker}_{H_0^1}(\mathcal{L})=\mathrm{Ker}_{L^2}(\mathcal{L})$ and for every $p\in [q,\infty]$,

$$\mathrm{Ker}_{H_0^1}(\mathcal{L})\subset L^p.$$

\end{Lem}
\begin{proof}

The result of \cite[Lemma 3.1]{CDS} implies that $\mathrm{Ker}_{H_0^1}(\mathcal{L})=\mathrm{Ker}_{L^2}(\mathcal{L})$, and that

$$\mathrm{Ker}_{H_0^1}(\mathcal{L})\subset L^\infty.$$
The lemma now follows by interpolation.

\end{proof}
Furthermore, according to \cite[Proposition 1.2]{CDS}, one has:

\begin{Lem}\label{finite-dim}

Assume that $(M,d,\mu)$ satisfies \eqref{UE} and \eqref{d}. Recall the exponent $\kappa$ from \eqref{dd}. Then, if $\kappa>2$, $\mathrm{Ker}_{L^2}(\mathcal{L})$ is finite-dimensional.

\end{Lem}
We also recall a fact that was used in the proof of \cite[Lemma 3.1]{CDS}:

\begin{Lem}\label{minig}

Assume that $M$ satisfies \eqref{d}, \eqref{UE}, and $\kappa>4$. Then, for all $\omega\in \mathrm{Ker}_{L^2}(\mathcal{L})$,

$$\omega=-\mathcal{H}^{-1}\mathcal{W}_0\omega.$$

\end{Lem}
Having recalled these notations and results from \cite{CDS}, we can turn to the proof of Theorem \ref{pi_heat}. In light of Lemma \ref{finite-dim}, under the assumptions of Theorem \ref{pi_heat} one can take a finite $L^2$-orthonormal basis $\left\{\omega_i\right\}_{i=1}^N$ of $\mathrm{Ker}_{L^2}(\mathcal{L})$. Then, the orthogonal projection $\Pi$ onto $\mathrm{Ker}_{L^2}(\mathcal{L})$ writes: 

$$\Pi=\sum_{i=1}^N \langle \omega_i ,\cdot\rangle \omega_i$$
where $\langle\cdot,\cdot\rangle$ is the $L^2$-scalar product.

To prove Theorem \ref{pi_heat}, we follow the strategy of the proof of \cite[Theorem 1.1]{CDS}. The crucial step will be to get resolvent estimates for the operator $\mathcal{L}+\alpha \Pi$, where $\alpha$ is a small constant to be determined later. The main technical result that will be used to get these resolvent estimates is the following:

\begin{Pro}\label{SpecRad}

There exists $\alpha>0$ such that, for all $\lambda\geq0$, and all $s\in [2,q']$, the operator $(I-(\mathcal{H}+\lambda)^{-1}(\mathcal{W}_0-\alpha\Pi))^{-1}$ is well-defined on $L^{s}$, and

$$\sup_{\lambda>0} ||(I-(\mathcal{H}+\lambda)^{-1}(\mathcal{W}_0-\alpha\Pi))^{-1}||_{s\to s}<+\infty.$$

\end{Pro} 
The proof of Proposition \ref{SpecRad} is similar in spirit to the one of \cite[Proposition 4.1]{CDS}. For $\lambda>0$, we define two operators:

$$\mathcal{A}_\lambda=(\mathcal{H}+\lambda)^{-1/2}(\mathcal{W}_0-\alpha\Pi)(\mathcal{H}+\lambda)^{-1/2},$$
and

$$\mathcal{B}_\lambda=(\mathcal{H}+\lambda)^{-1}(\mathcal{W}_0-\alpha\Pi).$$
Note that the definition of $\mathcal{A}_\lambda$ differs slightly from the one in \cite[Section 4]{CDS}, and is more adapted to our present purposes. In all the proof, $s$ will denote a real number that belongs to the interval $(2,q']$. First, we present a lemma that will be used several times in the course of the proof of Proposition \ref{SpecRad}:

\begin{Lem}\label{Halph}

Let $\psi$ be a function in $L^\infty$, compactly supported in $K\Subset M$, $\alpha>0$, and assume that for some (all) $y_0\in K$, and some $\infty\geq r\geq s\geq 1$

$$\int_{1}^{+\infty} \frac{t^{\alpha-1}}{V(y_0,\sqrt{t})^{1-\frac{1}{s}}}\, dt< {+\infty}.$$
Denote by $m_\psi$ the operator of multiplication by $\psi$. Then,

$$\sup_{\lambda\geq 0} ||(\mathcal{H}+\lambda)^{-\alpha}m_\psi ||_{r\to s}<\infty.$$
In particular, if \eqref{vol2} is satisfied with $p=s$, then

$$\sup_{\lambda\geq0} ||(\mathcal{H}+\lambda)^{-1}\mathcal{W}_0||_{s\to s}<\infty.$$

\end{Lem}

\begin{proof}

Write
$$(\mathcal{H}+\lambda)^{-\alpha}m_\psi=\int_0^{+\infty} e^{-t\lambda} e^{-t\mathcal{H}} m_\psi\,t^{\alpha-1}dt.$$
We split the integral into $\int_0^1+\int_1^{+\infty}$, and estimate both terms. By H\"older, the integral $\int_0^1$ is estimated by

$$\mathrm{Vol(K)}^{\frac{rs}{r-s}}||\psi||_\infty\int_0^1 ||e^{-t\mathcal{H}}||_{s\to s} \,t^{\alpha-1}dt.$$
Since $e^{-t\mathcal{H}}$ has Gaussian estimates, 

$$\sup_{t\geq0}||e^{-t\mathcal{H}}||_{s\to s}<+\infty,$$
and it follows that the integral $\int_0^1$ is bounded by a finite constant, independent of $\lambda$. Concerning the integral $\int_1^\infty$, it is bounded by

$$\mathrm{Vol(K)}^{\frac{r}{r-1}}\int_1^{+\infty} \|e^{-t\mathcal{H}}\|_{L^1(K)\to L^{s}}||\psi||_\infty \,t^{\alpha-1}dt .$$
Recall that according to \cite{BCS}, since the heat kernel of $e^{-t\mathcal{H}}$ has Gaussian estimates, 

$$\sup_{t>0}||e^{-t\mathcal{H}}V_{\sqrt{t}}^{1-\frac{1}{s}}||_{1\to s}<+\infty.$$
It follows that

$$||e^{-t\mathcal{H}}||_{L^1(K)\to L^s} \leq \frac{C}{V(y_0,\sqrt{t})^{1-\frac{1}{s}}},$$
where $y_0\in K$ is arbitrary. Then,
$$\int_{1}^{+\infty} \|e^{-t\mathcal{H}}\|_{L^1(K)\to L^{s}}\,t^{\alpha-1}dt \leq C_{K} \int_{1}^{+\infty} \frac{t^{\alpha-1}}{V(y_0,\sqrt{t})^{1-\frac{1}{s}}}\, dt< {+\infty}. $$
Hence, the integral $\int_1^\infty$ is also bounded by a finite constant, independent of $\lambda$.

\end{proof}
We now turn to the proof of Proposition \ref{SpecRad}, which will be divided into a sequence of claims.\\

\vskip5mm

{\em {\em CLAIM 1:} for every $s\in [2,q']$ and $\lambda\geq0$, the operator $\mathcal{B}_\lambda$ is compact on $L^s$, the map $\lambda\mapsto \mathcal{B}_\lambda\in \mathscr{L}(L^s,L^s)$ is continuous, and}

$$\sup_{\lambda\geq 0}||\mathcal{B}_\lambda||_{s\to s}<+\infty,\,\,\sup_{\lambda\geq 0}||\mathcal{B}_\lambda||_{s\to 2}<+\infty.$$

\bigskip

{\em Proof of Claim 1:}  we write

$$\mathcal{B}_\lambda=(\mathcal{H}+\lambda)^{-1}\mathcal{W}_0-\alpha (\mathcal{H}+\lambda)^{-1}\Pi.$$
The fact that $\kappa>4$ implies that for some $\varepsilon>0$, \eqref{vol2} holds for any $p\in (\frac{1}{2}-\varepsilon,\infty]$. Hence, according to Lemma \ref{Halph}

$$\sup_{\lambda\geq 0}||(\mathcal{H}+\lambda)^{-1}\mathcal{W}_0||_{s\to s}<+\infty,\,\,\sup_{\lambda\geq 0}||(\mathcal{H}+\lambda)^{-1}\mathcal{W}_0||_{s\to 2}<+\infty,$$
and

$$\sup_{\lambda\geq0}||(\mathcal{H}+\lambda)^{-1}\mathcal{W}_0||_{2-\varepsilon\to 2-\varepsilon}<\infty.$$
It was proved in \cite[Lemma 4.3]{CDS} that for every $\lambda\geq0$, the operator $(\mathcal{H}+\lambda)^{-1}\mathcal{W}_0$ is compact on $L^\infty$, and that the map $\lambda\mapsto (\mathcal{H}+\lambda)^{-1}\mathcal{W}_0\in \mathcal{L}(L^\infty,L^\infty)$ is continuous. By an interpolation argument (see \cite[Theorem 1.6.1]{Da} for the compactness part), for every $\lambda\geq0$, the operator $(\mathcal{H}+\lambda)^{-1}\mathcal{W}_0$ is compact on $L^s$, and moreover the map $\lambda\mapsto (\mathcal{H}+\lambda)^{-1}\mathcal{W}_0\in \mathscr{L}(L^s,L^s)$ is continuous.

We now turn to the term $(\mathcal{H}+\lambda)^{-1}\Pi$. This is a finite rank operator, so its compactness follows from its boundedness. By Lemma \ref{Lp-spaces}, 

$$\omega_i\in L^{s'}\cap L^s,\quad \forall i=1,\ldots,N.$$
The fact that $\omega_i\in L^{s'}$ implies that $\langle \omega_i,\cdot\rangle$ is a bounded linear form on $L^s$. In order to prove that

$$\sup_{\lambda\geq0}||(\mathcal{H}+\lambda)^{-1}\Pi||_{s\to s}<+\infty,\,\,\sup_{\lambda\geq0}||(\mathcal{H}+\lambda)^{-1}\Pi||_{s\to 2}<+\infty,$$
it remains to prove that

\begin{equation}\label{E1}
\sup_{\lambda\geq0} \left(||(\mathcal{H}+\lambda)^{-1}\omega_i||_s+||(\mathcal{H}+\lambda)^{-1}\omega_i||_2\right)<+\infty,\quad \forall i=1,\ldots,N.
\end{equation}
Note that for $\lambda>0$,

$$\mathcal{H}(\mathcal{H}+\lambda)^{-1}=I-\lambda(\mathcal{H}+\lambda)^{-1}.$$
Given that for every $p\in [1,\infty]$, $\sup_{t\geq0}||e^{-t\mathcal{H}}||_{p,p}<+\infty$ (which follows from the Gaussian estimates of $e^{-t\mathcal{H}}$), and using the formula

$$(\mathcal{H}+\lambda)^{-1}=\int_0^\infty e^{-\lambda t}e^{-t\mathcal{H}}\,dt,$$ 
one sees easily that

$$\sup_{\lambda\geq 0}\left(\lambda ||(\mathcal{H}+\lambda)^{-1}||_{s\to s}+\lambda ||(\mathcal{H}+\lambda)^{-1}||_{2\to 2}\right)<+\infty.$$
Therefore,

$$\sup_{\lambda\geq 0}\left(||\mathcal{H}(\mathcal{H}+\lambda)^{-1}||_{s\to s}+||\mathcal{H}(\mathcal{H}+\lambda)^{-1}||_{2\to 2}\right)<+\infty.$$
Consequently, writing $(\mathcal{H}+\lambda)^{-1}=\mathcal{H}(\mathcal{H}+\lambda)^{-1}\mathcal{H}^{-1}$, one sees that \eqref{E1} is equivalent to 

\begin{equation}\label{E2}
\mathcal{H}^{-1}\omega_i\in L^2\cap L^s,\quad \forall i=1,\ldots,N.
\end{equation}
According to Lemma \ref{minig},

$$\mathcal{H}^{-1}\omega_i=-\mathcal{H}^{-2}\mathcal{W}_0\omega_i.$$
By Lemma \ref{Halph}, the operator $\mathcal{H}^{-2}\mathcal{W}_0$ is bounded on $L^s$ and on $L^2$, provided that

$$\int_1^\infty \frac{t}{V(x_0,\sqrt{t})^{1/2}}\,dt<+\infty.$$
Since $\kappa>8$, the above integral is finite; consequently, \eqref{E2} is proved, and one has shown that

$$\sup_{\lambda\geq0}||(\mathcal{H}+\lambda)^{-1}\Pi||_{s\to s}<+\infty,\,\,\sup_{\lambda\geq0}||(\mathcal{H}+\lambda)^{-1}\Pi||_{s\to 2}<+\infty.$$
We finally prove that the map $\lambda\mapsto (\mathcal{H}+\lambda)^{-1}\Pi\in \mathscr{L}(L^s,L^s)$ is continuous. Clearly, one needs to prove that for all $i=1,\ldots,N$, the map $\lambda\mapsto (\mathcal{H}+\lambda)^{-1}\omega_i \in L^s$ is continuous. Continuity at $\lambda_0>0$ is easy enough: one starts with writing

\begin{equation}\label{E3}
(\mathcal{H}+\lambda)^{-1}-(\mathcal{H}+\lambda_0)^{-1}=(\lambda-\lambda_0)(\mathcal{H}+\lambda)^{-1}(\mathcal{H}+\lambda_0)^{-1}.
\end{equation}
As was already used above, the fact that $e^{-t\mathcal{H}}$ has Gaussian estimates implies that 

$$\sup_{\lambda\geq0}\lambda ||(\mathcal{H}+\lambda)^{-1}||_{s\to s}<+\infty.$$
Therefore, there is a constant $C(\lambda_0)>0$ such that, for all $\lambda\geq\frac{\lambda_0}{2}$,

$$||(\mathcal{H}+\lambda)^{-1}(\mathcal{H}+\lambda_0)^{-1}||_{s\to s}\leq C(\lambda_0).$$
Then, \eqref{E3} implies continuity of the map $\lambda\mapsto (\mathcal{H}+\lambda)^{-1}\Pi\in \mathscr{L}(L^s,L^s)$ at $\lambda_0>0$. For the continuity at $0$, write that $\omega_i=-\mathcal{H}^{-1}\mathcal{W}_0\omega_i$, so

$$(\mathcal{H}+\lambda)^{-1}\omega_i=-(\mathcal{H}+\lambda)^{-1}\mathcal{H}^{-1}\mathcal{W}_0\omega_i.$$
Thus,

$$(\mathcal{H}+\lambda)^{-1}\omega_i-\mathcal{H}^{-1}\omega_i=\left(\lambda^{-1}[(\mathcal{H}+\lambda)^{-1}-\mathcal{H}^{-1}]+\mathcal{H}^{-2}\right)\mathcal{W}_0\omega_i.$$
We claim that 

$$\lim_{\lambda\to 0}||\left(\lambda^{-1}[(\mathcal{H}+\lambda)^{-1}-\mathcal{H}^{-1}]+\mathcal{H}^{-2}\right)\mathcal{W}_0||_{s\to s}=0.$$
To prove this, we write

$$\left(\lambda^{-1}[(\mathcal{H}+\lambda)^{-1}-\mathcal{H}^{-1}]+\mathcal{H}^{-2}\right)\mathcal{W}_0=\int_0^\infty \left(\frac{e^{-\lambda t}-1}{\lambda}+t\right)e^{-t\mathcal{H}}\mathcal{W}_0\,dt.$$
One has

$$\left|\frac{e^{-\lambda t}-1}{\lambda}+t\right|\leq 2t,\quad \forall \lambda>0,\,t>0.$$
According to the proof of Lemma \ref{Halph}, the fact that $s\geq 2$ and the condition $\kappa>8$ imply that

$$\int_0^\infty ||e^{-t\mathcal{H}}\mathcal{W}_0||_{s\to s}\,\,t\,dt<\infty.$$
Clearly, for every $t>0$,

$$\lim_{\lambda\to 0}\frac{e^{-\lambda t}-1}{\lambda}+t=0,$$ so the dominated convergence theorem implies that

$$\lim_{\lambda\to 0}||\left(\lambda^{-1}[(\mathcal{H}+\lambda)^{-1}-\mathcal{H}^{-1}]+\mathcal{H}^{-2}\right)\mathcal{W}_0||_{s\to s}=0.$$
Hence, the continuity at $\lambda=0$ is proved.
$\Box$

\vskip5mm

{\em {\em CLAIM 2:} For every $\lambda\geq0$, the operator $\mathcal{A}_\lambda=(\mathcal{H}+\lambda)^{-1/2}(\mathcal{W}_0-\alpha\Pi)(\mathcal{H}+\lambda)^{-1/2}$ is self-adjoint, compact on $L^2$.}\\

{\em Proof of Claim 2:} According to \cite[Lemma 4.5]{CDS}, the operator $(\mathcal{H}+\lambda)^{-1/2}\mathcal{W}_0(\mathcal{H}+\lambda)^{-1/2}$ is self-adjoint, compact on $L^2$. Thus, it is enough to show that the operator $(\mathcal{H}+\lambda)^{-1/2}\Pi(\mathcal{H}+\lambda)^{-1/2}$ is self-adjoint, compact on $L^2$. Since $\Pi$ is a projection, $\Pi^2=\Pi$, so

$$(\mathcal{H}+\lambda)^{-1/2}\Pi^2(\mathcal{H}+\lambda)^{-1/2}.$$
The operator $\Pi (\mathcal{H}+\lambda)^{-1/2}$ having finite rank, its compactness follows from its boundedness, and therefore it is enough to see that the two operators $\Pi (\mathcal{H}+\lambda)^{-1/2}$ and $ (\mathcal{H}+\lambda)^{-1/2}\Pi$ are bounded on $L^2$ and adjoint one to another. By an argument similar to the proof of Claim 1, using Lemma \ref{Halph} one gets that for all $\omega\in \mathrm{Ker}_{L^2}(\mathcal{L})$,

$$\sup_{\lambda\geq0} ||(\mathcal{H}+\lambda)^{-1/2}\omega||_2<+\infty,$$
provided

$$\int_1^\infty \frac{\sqrt{t}}{V(y_0,\sqrt{t})^{1/2}}\,dt<+\infty.$$
Since $\kappa>8$, the above integral is finite. Thus, the operator $(\mathcal{H}+\lambda)^{-1/2}\Pi$ is bounded on $L^2$. Concerning the operator $\Pi (\mathcal{H}+\lambda)^{-1/2}$, one has, for all $\varphi\in C_0^\infty(E)$,

$$\Pi(\mathcal{H}+\lambda)^{-1/2}\varphi=\sum_{i=1}^N\langle \omega_i,(\mathcal{H}+\lambda)^{-1/2}\varphi\rangle \omega_i.$$
Since $\varphi$ and $\omega_i$, $i=1,\ldots,N$ belong to $L^2\cap H_0^1$, \cite[Lemma 3.1]{D0} implies that

$$\langle \omega_i,(\mathcal{H}+\lambda)^{-1/2}\varphi\rangle=\langle (\mathcal{H}+\lambda)^{-1/2}\omega_i,\varphi\rangle.$$
Note that \cite[Lemma 3.1]{D0} is written for {\em scalar} Schr\"odinger-type operators, however the arguments adapt easily to treat operators such as $\mathcal{H}$, if one uses Kato's inequality

$$|\nabla |\eta\|\leq |\nabla \eta|,\quad a.e.,\,\forall \eta\in C_0^\infty(E).$$  
Since $(\mathcal{H}+\lambda)^{-1/2}\omega_i\in L^2$, for all $i=1,\ldots,N$, one gets by density of $C_0^\infty(E)$ in $L^2(E)$ that the operator $\Pi(\mathcal{H}+\lambda)^{-1/2}$ is bounded on $L^2$, and that for every $\varphi\in L^2(E)$, 

$$\Pi(\mathcal{H}+\lambda)^{-1/2}\varphi=\sum_{i=1}^N\langle (\mathcal{H}+\lambda)^{-1/2}\omega_i,\varphi\rangle \omega_i.$$
Now, if $\psi\in L^2(E)$,

$$\begin{array}{rcl}
\langle \Pi(\mathcal{H}+\lambda)^{-1/2}\varphi,\psi\rangle&=&\sum_{i=1}^N\langle (\mathcal{H}+\lambda)^{-1/2}\omega_i,\varphi\rangle \langle \omega_i,\psi\rangle\\\\
&=&\langle \varphi,(\mathcal{H}+\lambda)^{-1/2}\Pi \psi\rangle.
\end{array}$$
This shows that the operators $\Pi (\mathcal{H}+\lambda)^{-1/2}$ and $ (\mathcal{H}+\lambda)^{-1/2}\Pi$ are adjoint one to another, which concludes the proof of Claim 2.

$\Box$

\vskip5mm

{\em {\em CLAIM 3:} If $\alpha>0$ is chosen small enough, there exists $\eta\in (0,1)$ such that, for all $\lambda\geq0$,

$$||\mathcal{A}_\lambda||_{2\to 2}\leq 1-\eta.$$
}\\

{\em Proof of Claim 3:}
Write

$$\mathcal{A}_\lambda=U_\lambda^*(\mathcal{H}^{-1/2}(\mathcal{W}_0-\alpha\Pi)\mathcal{H}^{-1/2})U_\lambda,$$
where $U_\lambda=\mathcal{H}^{1/2}(\mathcal{H}+\lambda)^{-1/2}$. Notice that, since

$$\sqrt{\frac{x}{x+\lambda}}\leq 1,\quad \forall x\geq 0,\,\forall \lambda\geq0,$$
by the Spectral Theorem the operator $U_\lambda$ is contractive on $L^2$. Therefore, it is enough to prove that for some $\eta\in (0,1)$ and $\alpha>0$ small enough,

$$||\mathcal{H}^{-1/2}(\mathcal{W}_0-\alpha\Pi)\mathcal{H}^{-1/2}||_{2\to 2}\leq 1-\eta.$$
By Claim 2, the operator $\mathcal{A}_0=\mathcal{H}^{-1/2}(\mathcal{W}_0-\alpha\Pi)\mathcal{H}^{-1/2}$ is self-adjoint and compact on $L^2$, therefore

$$||\mathcal{A}_0||_{2\to 2}=\max_{||\varphi||_2=1}|\langle \mathcal{A}_0\varphi,\varphi\rangle|.$$
But

$$\begin{array}{rcl}
\langle \mathcal{A}_0\varphi,\varphi\rangle&=&\langle \mathcal{H}^{-1/2}\mathcal{W}_0\mathcal{H}^{-1/2}\varphi,\varphi\rangle-\alpha \langle \mathcal{H}^{-1/2}\Pi\mathcal{H}^{-1/2}\varphi,\varphi\rangle\\\\
&=&\langle \mathcal{H}^{-1/2}\mathcal{W}_0\mathcal{H}^{-1/2}\varphi,\varphi\rangle-\alpha ||\Pi\mathcal{H}^{-1/2}\varphi||_2^2,
\end{array}$$
where in the last equality we have used that 

$$\begin{array}{rcl}
\mathcal{H}^{-1/2}\Pi\mathcal{H}^{-1/2}&=&\mathcal{H}^{-1/2}\Pi^2\mathcal{H}^{-1/2}\\\\
&=& (\Pi\mathcal{H}^{-1/2})^*\Pi\mathcal{H}^{-1/2}
\end{array}$$
(see the proof of Claim 2). Since $\mathcal{W}_0$ is non-negative, the operator $\mathcal{H}^{-1/2}\mathcal{W}_0\mathcal{H}^{-1/2}$ is non-negative, and therefore, for $\varphi\in L^2(E)$,

$$\langle \mathcal{A}_0\varphi,\varphi\rangle\geq -\mu ||\Pi\mathcal{H}^{-1/2}\varphi||_2^2\geq -\alpha ||\Pi\mathcal{H}^{-1/2}||_{2\to 2}^2||\varphi||_2^2 .$$
Take 

$$\alpha:=\frac{1}{2}(1+||\Pi\mathcal{H}^{-1/2}||_{2\to 2}^2)^{-1},$$
then one gets that for every $\varphi\in L^2(E)$ such that $||\varphi||_2=1$,

$$\langle \mathcal{A}_0\varphi,\varphi\rangle\geq -\frac{1}{2}.$$
Thus, in order to conclude the proof, it is enough to see that with the above choice for $\alpha$, there exists $\eta\in (0,1)$ such that, for all $\varphi\in L^2(E)$,

$$\langle \mathcal{H}^{-1/2}\mathcal{W}_0\mathcal{H}^{-1/2}\varphi,\varphi\rangle-\alpha ||\Pi\mathcal{H}^{-1/2}\varphi||_2^2\leq (1-\eta)||\varphi||_2^2.$$
Let $\varphi_0\in L^2(E)\setminus \{0\}$, $||\varphi_0||_2=1$, such that

$$\langle \mathcal{A}_0\varphi_0,\varphi_0\rangle=\max_{||\varphi||_2=1}\langle \mathcal{A}_0\varphi,\varphi\rangle.$$
Such a $\varphi_0$ exists by compactness of $\mathcal{A}_0$. Now, for any $\varphi\in L^2(E)$, if one lets $\psi=\mathcal{H}^{-1/2}\varphi\in H_0^1$, then

$$\begin{array}{rcl}
\langle \mathcal{H}^{-1/2}\mathcal{W}_0\mathcal{H}^{-1/2}\varphi,\varphi\rangle \leq \delta \langle\varphi,\varphi\rangle &\Leftrightarrow & \langle \mathcal{H}^{-1/2}\mathcal{W}_0 \psi,\mathcal{H}^{1/2}\psi\rangle \leq \delta \langle \mathcal{H}^{1/2}\psi,\mathcal{H}^{1/2}\psi\rangle\\\\
&\Leftrightarrow& \langle \mathcal{W}_0\psi,\psi\rangle\leq \delta q_{\mathcal{H}}(\psi).
\end{array}$$
The latter inequality is true with $\delta=1$ since $\mathcal{L}=\mathcal{H}-\mathcal{W}_0$ is non-negative. Since $\mathcal{H}^{-1/2}\mathcal{W}_0\mathcal{H}^{-1/2}$ is self-adjoint, this implies that $||\mathcal{H}^{-1/2}\mathcal{W}_0\mathcal{H}^{-1/2}||_{2\to 2}\leq 1$.  The operator $\mathcal{H}^{-1/2}\mathcal{W}_0\mathcal{H}^{-1/2}$ being self-adjoint and compact on $L^2$ by \cite[Lemma 4.5]{CDS}, either there is $\eta\in (0,1)$ such that

$$\langle \mathcal{H}^{-1/2}\mathcal{W}_0\mathcal{H}^{-1/2}\varphi_0,\varphi_0\rangle \leq (1-\eta),$$
or 

$$(I-\mathcal{H}^{-1/2}\mathcal{W}_0\mathcal{H}^{-1/2})\varphi_0=0.$$
In the former case, the result of Claim 3 follows, since

$$\langle \mathcal{A}_0\varphi_0,\varphi_0\rangle\leq \langle \mathcal{H}^{-1/2}\mathcal{W}_0\mathcal{H}^{-1/2}\varphi_0,\varphi_0\rangle.$$
In the latter case, one recalls from \cite[Lemma 1]{D2} that $\mathcal{H}^{-1/2}$ is an isometry from $\mathrm{Ker}_{L^2}(I-\mathcal{H}^{-1/2}\mathcal{W}_0\mathcal{H}^{-1/2})$ to $\mathrm{Ker}_{H_0^1}(\mathcal{L})$; let us limit ourselves here to mentioning that this relies on the formula:

$$\mathcal{L}=\mathcal{H}^{1/2}(I-\mathcal{H}^{-1/2}\mathcal{W}_0\mathcal{H}^{-1/2})\mathcal{H}^{1/2}.$$
Therefore, if one lets $\psi_0:=\mathcal{H}^{-1/2}\varphi_0$, then $\psi_0\in \mathrm{Ker}_{H_0^1}(\mathcal{L})=\mathrm{Ker}_{L^2}(\mathcal{L})$ (Lemma \ref{Lp-spaces}), so $\Pi\psi_0=\psi_0$, and

$$\begin{array}{rcl}
\langle \mathcal{A}_0\varphi_0,\varphi_0\rangle &=& 1-\alpha ||\Pi \psi_0||_2^2\\\\
&=& 1-\alpha ||\psi_0||_2^2.
\end{array}$$
In this case, one lets 

$$\eta=\min(\frac{1}{2},\alpha ||\psi_0||_2^2)\in (0,1),$$
then

$$\langle \mathcal{A}_0\varphi_0,\varphi_0\rangle \leq 1-\eta,$$
which implies by definition of $\varphi_0$ that

$$\max_{||\varphi||_2=1}\langle \mathcal{A}_0\varphi,\varphi\rangle \leq 1-\eta.$$
This concludes the proof of Claim 3.

$\Box$

\vskip5mm

{\em {\em CLAIM 4:} For every $\lambda\geq0$, the spectral radius $r_s$ of $\mathcal{B}_\lambda$ on $L^s$ satisfies

$$r_s\leq 1-\eta.$$
}\\

{\em Proof of Claim 4:}
Since $\mathcal{B}_\lambda$ is compact on $L^s$, there exists $\varphi\in L^s\setminus \{0\}$ such that

$$\mathcal{B}_\lambda \varphi=\gamma \varphi,\quad |\gamma|=r_s.$$
That is,

$$(\mathcal{H}+\lambda)^{-1}(\mathcal{W}_0-\alpha\Pi)\varphi=\gamma \varphi.$$
According to Claim 1, $\mathcal{B}_\lambda$ is bounded from $L^s$ to $L^2$, therefore $\varphi\in L^2$. Next,

$$\gamma (\mathcal{H}+\lambda)^{1/2}\varphi=(\mathcal{H}+\lambda)^{-1/2}(\mathcal{W}_0-\alpha\Pi)\varphi.$$
It has been proved in the course of the proof of Claim 2 that the operator $(\mathcal{H}+\lambda)^{-1/2}\Pi$ is bounded on $L^2$, and according to \cite[Lemma 4.5]{CDS}, the operator $(\mathcal{H}+\lambda)^{-1/2}\mathcal{W}_0$ is bounded on $L^2$. Therefore, $(\mathcal{H}+\lambda)^{1/2}\varphi\in L^2$. Then, one writes

$$\begin{array}{rcl}
\gamma (\mathcal{H}+\lambda)^{1/2}\varphi&=&\left((\mathcal{H}+\lambda)^{-1/2}(\mathcal{W}_0-\alpha\Pi)(\mathcal{H}+\lambda)^{-1/2}\right)(\mathcal{H}+\lambda)^{1/2}\varphi\\\\
&=& \mathcal{A}_\lambda (\mathcal{H}+\lambda)^{1/2}\varphi.
\end{array}$$
Letting $\psi=(\mathcal{H}+\lambda)^{1/2}\varphi\in L^2$, one gets

$$\gamma \psi=\mathcal{A}_\lambda\psi.$$
By Claim 3, one obtains $r_s=|\gamma|\leq 1-\eta$.

$\Box$

\vskip5mm

\noindent{\em Proof of Proposition $\ref{SpecRad}$:} The proof now follows along the lines of \cite[Proposition 4.1]{CDS}. For the reader's convenience, we write the details. One needs to show that

\begin{equation}\label{Neumann}
\sup_{\lambda\geq0}||(I-\mathcal{B}_\lambda)^{-1}||_{s\to s}<+\infty.
\end{equation}
By Claim 4, for every $\lambda\geq0$, the spectral radius of $\mathcal{B}_\lambda$ on $L^s$ is less than $1-\eta$. In particular, $1$ does not belong to the $L^s$-spectrum of $\mathcal{B}_\lambda$, so $(I-\mathcal{B}_\lambda)^{-1}$ is well-defined and is a bounded operator on $L^s$. Since $e^{-t\mathcal{H}}$ is uniformly bounded on $L^s$, it follows from the formula

$$(\mathcal{H}+\lambda)^{-1}=\int_0^\infty e^{-t\lambda}e^{-tH}\,dt$$
that

$$\sup_{\lambda>0}\lambda\|(\mathcal{H}+\lambda)^{-1}\|_{s\to s}<+\infty.$$
Since $\mathcal{W}_0$ is bounded, there exists $\Lambda>0$ such that for every $\lambda\geq\Lambda$,
\begin{equation}\label{largelambda}
\|\mathcal{B}_{\lambda}\|_{s\to s}=\|(\mathcal{H}+\lambda)^{-1}\mathcal{W}_0\|_{s\to s}\leq \frac{1}{2}.
\end{equation}
It follows that for every $\lambda\geq\Lambda$,
$$\|(I-\mathcal{B}_\lambda)^{-1}\|_{s\to s}\leq 2.$$
\bigskip
It remains to show that $\|(I-\mathcal{B}_\lambda)^{-1}\|_{s\to s}$ is uniformly bounded for $\lambda\in [0,\Lambda]$. For this, it is enough to prove that $\lambda\mapsto (I-\mathcal{B}_\lambda)^{-1}\in \mathscr{L}(L^s,L^s)$ is continuous on $[0,\infty)$. 

Let $\lambda_0\geq 0$. By Claim 4,
$$\lim_{n\to +\infty}\|\mathcal{B}_{\lambda_0}^n\|_{s\to s}^{1/n}\leq 1-\eta,$$
so there exists $N$ such that $\|\mathcal{B}_{\lambda_0}^N\|_{s\to s}^{1/N}\leq 1-\frac{\eta}{2}$. Hence for some $\delta\in (0,1)$,

$$\|\mathcal{B}_{\lambda_0}^N\|_{s\to s}\leq 1-\delta.$$
Now we know by Claim 1 that the map $\lambda\mapsto \mathcal{B}_\lambda\in \mathscr{L}(L^s,L^s)$ is continuous on $[0,\infty)$.
Therefore, for $\lambda$ close enough to $\lambda_0$, one has
$$\|\mathcal{B}_{\lambda}^N\|_{s\to s}\leq 1-\frac{\delta}{2}.$$
By Claim 1, there is a constant $C\geq 1$ such that for every $\lambda\geq 0$,
$$\|\mathcal{B}_\lambda\|_{s\to s}\leq C.$$
Thus, for every $n\geq N$ and $\lambda$ close  to $\lambda_0$,
$$\|\mathcal{B}_\lambda^n\|_{s\to s}\leq C^N \left(1-\frac{\delta}{2}\right)^{[n/N]}.$$
Thus, the series $\sum_{n\geq0}\|\mathcal{B}_\lambda^n\|_{s\to s}$ converges uniformly for $\lambda$ close  to $\lambda_0$, and since $\lambda\mapsto \mathcal{B}_\lambda\in \mathscr{L}(L^s,L^s)$ is continuous, this implies the continuity of $\lambda\mapsto (I-\mathcal{B}_\lambda)^{-1}$ at $\lambda_0$. The proof of Proposition \ref{SpecRad} is complete.

\cqfd

\noindent{\em Proof of Theorem \ref{pi_heat}:} Given the result of Proposition \ref{SpecRad}, the result of Theorem \ref{pi_heat} follows from an iterative argument similar to the one used in the proof of \cite[Theorem 2.1]{CDS}. For $t>0$, let us write the perturbation formula:

$$(I+t(\mathcal{L}+\alpha\Pi))^{-1}=(I-(I+t\mathcal{H})^{-1}t(\mathcal{W}_0-\alpha\Pi))^{-1}(I+t\mathcal{H})^{-1}.$$
According to \cite[Lemma 2.5]{CDS},

$$\sup_{t>0}||(I+t\mathcal{H})^{-1}V_{\sqrt{t}}^{\frac{1}{r}-\frac{1}{s}}||_{r\to s}<+\infty,$$
for all $1\leq r\leq s\leq \infty$ with $\frac{1}{r}-\frac{1}{s}<\frac{2}{\nu}$. It follows from Proposition \ref{SpecRad} that for all $s\in [2,q']$,

$$\sup_{t>0}||(I-(I+t\mathcal{H})^{-1}t(\mathcal{W}_0-\alpha \Pi))^{-1}||_{s\to s}<+\infty.$$
Hence, for all $2\leq r\leq s\leq q'$ such that $\frac{1}{r}-\frac{1}{s}<\frac{2}{\nu}$,

$$\sup_{t>0}||(I+t(\mathcal{L}+\alpha\Pi))^{-1}V_{\sqrt{t}}^{\frac{1}{r}-\frac{1}{s}}||_{r\to s}<\infty.$$
Thanks to the uniform volume growth assumption \eqref{vol_u}, this is equivalent to

$$\sup_{t>0}V_{\sqrt{t}}^{\frac{1}{r}-\frac{1}{s}}||(I+t(\mathcal{L}+\alpha\Pi))^{-1}||_{r\to s}<\infty,$$
for all $2\leq r\leq s\leq q'$ such that $\frac{1}{r}-\frac{1}{s}<\frac{2}{\nu}$
By duality, one also gets

$$\sup_{t>0}V_{\sqrt{t}}^{\frac{1}{r}-\frac{1}{s}}||(I+t(\mathcal{L}+\alpha\Pi))^{-1}||_{r\to s}<\infty,$$
for all $q\leq r\leq s\leq 2$ such that $\frac{1}{r}-\frac{1}{s}<\frac{2}{\nu}$. By interpolation, one obtains

\begin{equation}\label{iterate}
\sup_{t>0}V_{\sqrt{t}}^{\frac{1}{r}-\frac{1}{s}}||(I+t(\mathcal{L}+\alpha\Pi))^{-1}||_{r\to s}<\infty,
\end{equation}
for all $q\leq r\leq s\leq q'$ such that $\frac{1}{r}-\frac{1}{s}<\frac{2}{\nu}$. Now, fix $2\leq s\leq q'$. We claim that

\begin{equation}\label{iterate2}
\sup_{t>0}V_{\sqrt{t}}^{\frac{1}{2}-\frac{1}{s}}||(I+t(\mathcal{L}+\alpha\Pi))^{-n}||_{2\to s}<\infty,
\end{equation}
where $n$ is any fixed integer such that

$$\frac{1}{2}-\frac{1}{s}<\frac{2n}{\nu}.$$
Indeed, if $n=1$, i.e. $\frac{1}{2}-\frac{1}{s}<\frac{2}{\nu}$, then by \eqref{iterate} we are done. If not, let $r_0=2<r_1<\cdots<r_n=s$ be an increasing sequence such that for every $k=0,\ldots,n-1$,

$$\frac{1}{r_k}-\frac{1}{r_{k+1}}<\frac{2}{\nu}.$$
Thanks to \eqref{iterate}, one has for every $k=0,\ldots,n-1$,

$$\sup_{t>0}V_{\sqrt{t}}^{\frac{1}{r_k}-\frac{1}{r_{k+1}}}||(I+t(\mathcal{L}+\alpha\Pi))^{-1}||_{r_k\to r_{k+1}}<\infty.$$
Hence,

$$\sup_{t>0}V_{\sqrt{t}}^{\frac{1}{r}-\frac{1}{s}}||(I+t(\mathcal{L}+\alpha\Pi))^{-n}||_{2\to s}<\infty,$$
and \eqref{iterate2} is proved. Now, since the operator $\mathcal{L}+\alpha\Pi$ is self-adjoint, the Spectral Theorem implies that the semigroup $e^{-t(\mathcal{L}+\alpha\Pi)}$ is analytic on the half-plane $\{z\in \mathbb{C}\,;\,\mathfrak{Re}(z)>0\}$, and as a consequence 

$$\sup_{t>0}||(I+t(\mathcal{L}+\alpha\Pi))^{n}e^{-t(\mathcal{L}+\alpha\Pi)}||_{2\to 2}<+\infty.$$
Writing 

$$e^{-t(\mathcal{L}+\mu\Pi)}=(I+t(\mathcal{L}+\alpha \Pi))^{-n} (I+t(\mathcal{L}+\alpha\Pi))^{n}e^{-t(\mathcal{L}+\alpha\Pi)},$$
one gets from \eqref{iterate2} that

\begin{equation}\label{iterate3}
\sup_{t>0}V_{\sqrt{t}}^{\frac{1}{2}-\frac{1}{s}}||e^{-t(\mathcal{L}+\alpha\Pi)}||_{2\to s}<\infty,
\end{equation}
for all $2\leq s\leq q'$. 



Since $\mathcal{L}$ and $\Pi$ commute, and $\Pi^2=\Pi$, one finds

$$\begin{array}{rcl}
e^{-t(\mathcal{L}+\alpha\Pi)}&=&e^{-t\mathcal{L}}e^{-t\alpha\Pi}\\\\
&=&e^{-t\mathcal{L}}(I+(e^{-t\alpha}-1)\Pi)
\end{array}$$
Notice that

$$I-\Pi=(I+(e^{-t\alpha}-1)\Pi)(I-\Pi),$$
so \eqref{iterate3} implies that

\begin{equation}\label{iterate6}
\sup_{t>0}V_{\sqrt{t}}^{\frac{1}{2}-\frac{1}{s}}||e^{-t\mathcal{L}}(I-\Pi)||_{2\to s}<\infty,\quad \forall 2\leq s\leq q'.
\end{equation}
By duality and the fact that $\mathcal{L}$ and $\Pi$ commute, \eqref{iterate6} yields

\begin{equation}\label{iterate7}
\sup_{t>0}V_{\sqrt{t}}^{\frac{1}{r}-\frac{1}{2}}||e^{-t\mathcal{L}}(I-\Pi)||_{r\to 2}<\infty,\quad \forall q\leq r\leq 2.
\end{equation}
By composition of \eqref{iterate6} and \eqref{iterate7} and doubling, one obtains

\begin{equation}\label{iterate8}
\sup_{t>0}V_{\sqrt{t}}^{\frac{1}{r}-\frac{1}{s}}||e^{-t\mathcal{L}}(I-\Pi)||_{r\to s}<\infty,\quad \forall q\leq r\leq 2 \leq s\leq q'.
\end{equation}
CLAIM: the semi-group $e^{-t\mathcal{L}}$ is bounded on $L^p$, for any $p\in [q,q']$. 

\medskip

Postponing for the moment the proof of the above claim (see Proposition \ref{Lp-bdd}), let us finish the proof of the first part of Theorem \ref{pi_heat}. The claim, together with the fact that the projector $\Pi$ is bounded on $L^p$ for $p\in [q,q']$ implies that $e^{-t\mathcal{L}}(I-\Pi)$ is bounded on $L^p$ for $p\in [q,q']$. Interpolating this with \eqref{iterate8} yields

$$\sup_{t>0}V_{\sqrt{t}}^{\frac{1}{r}-\frac{1}{s}}||e^{-t\mathcal{L}}(I-\Pi)||_{r\to s}<\infty,\quad \forall q\leq r \leq s\leq q',$$
which proves the first part of Theorem \ref{pi_heat}.

\medskip

Concerning the second part, we assume that $q=1$, and we denote by $k_t(x,y)$ the kernel of $e^{-t\mathcal{L}}(I-\Pi)$. Since 

$$e^{-t\mathcal{L}}(I-\Pi)=(I-\Pi)e^{-t\mathcal{L}}(I-\Pi)=(I-\Pi)e^{-t\mathcal{L}},$$
the kernel $k_t(x,y)$ satisfies the semigroup property, and is symmetric in the following sense:

$$k_t(x,y)=k_t(y,x)^*.$$
Thus,

$$\begin{array}{rcl}
\int_M||k_t(x,y)||_{y,x}^2\,d\mu(y)&=&\int_M \mathrm{Tr} \left(k_t(x,y)k_t(x,y)^*\right)\,d\mu(y)\\\\
&=& \int_M \mathrm{Tr} \left(k_t(x,y)k_t(y,x)\right)\,d\mu(y)\\\\
&=&\mathrm{Tr}\left(\int_M k_t(x,y)k_t(y,x)\,d\mu(y)\right)\\\\
&=&\mathrm{Tr}\left(k_{2t}(x,x)\right)
\end{array}$$
Also, $k_t(x,x)$ is non-negative and self-adjoint, therefore

$$\mathrm{Tr}\,k_t(x,x)\geq ||k_t(x,x)||_{x,x}.$$
As a consequence, by \eqref{d} the $L^2\to L^\infty$ estimate:

\begin{equation}\label{L2-Linf}
\sup_{t>0} ||V_{\sqrt{t}}^{-\frac{1}{2}}e^{-t\mathcal{L}}(I-\Pi)||_{2\to \infty}<\infty
\end{equation}
is equivalent to the on-diagonal estimates \eqref{on-diago} for the kernel $k_t(x,y)$ of the operator $e^{-t\mathcal{L}}(I-\Pi)$. The proof is complete.

\cqfd
The following proposition has been used in the proof of Theorem \ref{pi_heat}:

\begin{Pro}\label{Lp-bdd}

Assume that $M$ satisfies \eqref{d} and \eqref{UE}. Let $E$ be a vector bundle with basis $M$, endowed with a connection $\nabla$ compatible with the metric, and let

$$\mathcal{L}=\nabla^*\nabla+\mathcal{R}_+-\mathcal{R}_-$$
be a non-negative generalised Schr\"odinger operator on $E$. Let $\Pi$ be the $L^2$ orthogonal projection onto $\mathrm{Ker}_{L^2}(\mathcal{L})$. Assume that for some $p\in [1,2)$,

$$\sup_{t>0}\left\|V_{\sqrt{t}}^{\frac{1}{p}-\frac{1}{2}}e^{-t\mathcal{L}}(I-\Pi)\right\|_{p\to 2}<\infty.$$
Assume also that $\Pi$ is bounded on $L^p$. Then, the semi-group $e^{-t\mathcal{L}}$ is uniformly bounded on $L^p$.

\end{Pro}

\begin{proof}

The proof relies on ideas from \cite{CS} and \cite{BCS}. The starting point is the following formula (see the proof of \cite[Proposition 4.1.6]{BCS}):

\begin{equation}\label{subord}
e^{-t\mathcal{L}}=\int_0^\infty s^{a+\frac{1}{2}}F_a(\sqrt{st\mathcal{L}})e^{-\frac{s}{4}}\,ds,
\end{equation}
where the function $F_a$ is (up to a multiplicative constant) the inverse Fourier transform of the function $(1-t^2)_+^a$, $a>0$. Since the function $(1-t^2)_+^a$ is even and has support included in $[-1,1]$, it follows from the finite speed ($=1$) propagation property for the wave equation of $\mathcal{L}$ (see \cite[Theorem 2]{Sik} and \cite[Section 8.3]{CDS}) that for every $r>0$, 

\begin{equation}\label{FSP}
\mathrm{supp}(F_a(r\sqrt{\mathcal{L}}))\subset \mathcal{D}_r,
\end{equation}
where $\mathcal{D}_r$ is the following neighborhood of the diagonal:

$$\mathcal{D}_r=\{(x,y)\in M\times M\,;\, d(x,y)\leq r\}.$$
See \cite[Lemma 4.1.3]{BCS}. Clearly, according to \eqref{subord}, it is enough to prove that for some $a>0$,

\begin{equation}\label{Lp-bd}
\sup_{r>0}||F_a(r\sqrt{L})||_{p\to p}<+\infty.
\end{equation}
First, we prove that for $N>\frac{\nu}{2}\left(\frac{1}{p}-\frac{1}{2}\right)$, where $\nu$ is the doubling exponent from \eqref{dnu}, there holds:

\begin{equation}\label{resol_p_to_2}
\sup_{t>0}\left\|V_{\sqrt{t}}^{\frac{1}{p}-\frac{1}{2}}(I+t\mathcal{L})^{-N}(I-\Pi)\right\|_{p\to 2}<+\infty.
\end{equation}
For this, we write

$$\begin{array}{rcl}
V_{\sqrt{t}}^{\frac{1}{p}-\frac{1}{2}}(I+t\mathcal{L})^{-N}&=&\frac{1}{\Gamma(N)}\int_0^\infty V_{\sqrt{t}}^{\frac{1}{p}-\frac{1}{2}}e^{-s(I+t\mathcal{L})}\,s^{N-1}\,ds\\\\
&=&\frac{1}{\Gamma(N)}\int_0^\infty \left(\frac{V_{\sqrt{t}}}{V_{\sqrt{ts}}}\right)^{\frac{1}{p}-\frac{1}{2}}V_{\sqrt{st}}^{\frac{1}{p}-\frac{1}{2}}e^{-st\mathcal{L}}\,s^{N-1}e^{-s}\,ds
\end{array}$$
By doubling,

$$\frac{V_{\sqrt{t}}}{V_{\sqrt{ts}}}\leq \max(1,s^{-\nu/2})$$
(note that the estimate is trivial when $s\geq1$). Let $\delta=\frac{\nu}{2}\left(\frac{1}{p}-\frac{1}{2}\right)$. Using the assumption on the heat semi-group, one gets

$$\left\|V_{\sqrt{t}}^{\frac{1}{p}-\frac{1}{2}}(I+t\mathcal{L})^{-N}(I-\Pi)\right\|_{p\to 2}\lesssim \int_0^\infty s^{N-1}\max(s^{-\delta},1)\,e^{-s}\,ds<+\infty,$$
since $N$ has been assumed to be strictly greater than $\delta$, hence \eqref{resol_p_to_2}. Now, write, for some fixed $N>\delta$,

$$V_{r}^{\frac{1}{p}-\frac{1}{2}}F_a(r\sqrt{\mathcal{L}})(I-\Pi)=\left(V_{r}^{\frac{1}{p}-\frac{1}{2}}(I+r^2\mathcal{L})^{-N}(I-\Pi)\right)\left(F_a(r\sqrt{\mathcal{L}})(I+r^2\mathcal{L})^{+N}\right).$$
According to \eqref{resol_p_to_2}, the first factor is bounded as an operator from $p$ to $2$, uniformly in $r>0$. In order to deal with the second factor, note that if $a>2N$, the function $(1-t^2)_+^a$ is in $C_0^{2N+1}(\mathbb{R})$, hence its inverse Fourier transform $F_a(\lambda)$ is smooth and decays at a polynomial rate of order $2N$, hence

$$\sup_{\lambda\geq0} |F_a(\lambda)(1+\lambda^2)^N|<+\infty.$$
The Spectral Theorem then implies that

$$\sup_{r>0}\left\|F_a(r\sqrt{\mathcal{L}})(I+r^2\mathcal{L})^N\right\|_{2\to 2}<+\infty.$$
Finally, one gets that for $a>2N$,

\begin{equation}\label{F_a}
\sup_{r>0}\left\|V_{r}^{\frac{1}{p}-\frac{1}{2}}F_a(r\sqrt{\mathcal{L}})(I-\Pi)\right\|_{2\to p}<+\infty.
\end{equation}
We are now in position to prove the uniform boundedness on $L^p$ of $F_a(r\sqrt{\mathcal{L}})$, following the ideas in the proof of \cite[Corollary 4.16]{CS}. Take a maximal sequence $(x_n)_{n\in\mathbb{N}}$ such that the balls $B(x_k,\frac{r}{2})$ are disjoint. By doubling, there exists $D\in\mathbb{N}^*$ such that every $x\in M$ is contained in at most $D$ balls $B(x_k,r)$. Let $\chi_k$ be the characteristic function of the set $B_k:=B(x_k,r)\setminus \cup_{i=0}^{k-1}B(x_i,r)$. Clearly, the sets $B_k$ are disjoint, and 

$$M=\cup_{k=0}^\infty B_k.$$
We also set $r_{kl}=d(B_k,B_l)$. Notice that since $\Pi$ is bounded on $L^p$, and $F_a(r\sqrt{\mathcal{L}})\Pi=F_a(0)\Pi$ is bounded uniformly in $r>0$, the uniform boundedness on $L^p$ of $F_a(r\sqrt{\mathcal{L}})$ is equivalent to that of $F_a(r\sqrt{\mathcal{L}})(I-\Pi)$. Using that $\mathbf{1}=\sum_{i=0}^\infty \chi_i$ and the supports of $\chi_i$ are disjoint, one has by Minkowski's inequality:

$$\left\|F_a(r\sqrt{\mathcal{L}})(I-\Pi)\omega \right\|_{p}^p\leq \sum_{k=0}^\infty \left(\sum_{l=0}^\infty \left\|\chi_k F_a(r\sqrt{\mathcal{L}})(I-\Pi)\chi_l\right\|_{p\to p}||\chi_l \omega||_p\right)^p.$$
Applying the following elementary inequality (see \cite[Equation 4.42]{CS}):

$$\sum_{k=0}^\infty \left(\sum_{l=0}^\infty |c_{k,l}||a_l|\right)^p \leq \max\left(  \sup_l\sum_{k=0}^\infty |c_{k,l}|,\sup_k\sum_{l=0}^\infty |c_{k,l}| \right)\sum_{l=0}^\infty |a_l|^p,$$
with the choices $c_{k,l}=\left\|\chi_k F_a(r\sqrt{\mathcal{L}})(I-\Pi)\chi_l\right\|_{p\to p}$, $a_l=||\chi_l \omega||_p$, one gets 

$$\begin{array}{rcl}
\left\|F_a(r\sqrt{\mathcal{L}})(I-\Pi)\omega \right\|_{p}^p&\leq& ||\omega||_p^p\max\Big(  \sup_l\sum_{k=0}^\infty \left\|\chi_k e^{-t\mathcal{L}}(I-\Pi)\chi_l\right\|_{p\to p},\\
&&\sup_k\sum_{l=0}^\infty \left\|\chi_k F_a(r\sqrt{\mathcal{L}})(I-\Pi)\chi_l\right\|_{p\to p} \Big),
\end{array}$$
Hence,

$$\begin{array}{rcl}
\left\|F_a(r\sqrt{\mathcal{L}})(I-\Pi) \right\|_{p\to p}^p&\leq&\max\Big(  \sup_l\sum_{k=0}^\infty \left\|\chi_k F_a(r\sqrt{\mathcal{L}})(I-\Pi)\chi_l\right\|_{p\to p},\\
&&\sup_k\sum_{l=0}^\infty \left\|\chi_k F_a(r\sqrt{\mathcal{L}})(I-\Pi)\chi_l\right\|_{p\to p} \Big),
\end{array}$$
and in order to prove uniform boundedness of $F_a(r\sqrt{\mathcal{L}})(I-\Pi)$ on $L^p$, it is enough to prove that the above maximum is uniformly bounded. The proofs that the supremum in $k$ or in $l$ is finite being similar, we choose to explain only why the supremum in $l$ is finite. For this, we split the sum in $k$ into two parts, an on-diagonal one and an off-diagonal one:

$$\sum_{k=0}^\infty \left\|\chi_k F_a(r\sqrt{\mathcal{L}})(I-\Pi)\chi_l\right\|_{p\to p}=\sum_{r_{kl}\leq 2r}(\cdots)+\sum_{r_{kl}>2r}(\cdots)=I+II.$$
The on-diagonal part will be taken care of using \eqref{F_a}. Notice that by doubling, the number of terms in the sum $I$ is finite, and bounded independently of $r$. Hence, it is enough to control individually each term $\left\|\chi_k F_a(r\sqrt{\mathcal{L}})(I-\Pi)\chi_l\right\|_{p\to p}$, $r_{kl}\leq 2r$ of the sum. By H\"older inequality and doubling,

$$\begin{array}{rcl}
\left\|\chi_k F_a(r\sqrt{\mathcal{L}})(I-\Pi)\chi_l\right\|_{p\to p}&\leq& V(x_k,r)^{\frac{1}{p}-\frac{1}{2}}\left\|\chi_k F_a(r\sqrt{\mathcal{L}})(I-\Pi)\chi_l\right\|_{p\to 2}\\\\
&\lesssim & \left\|V_{r}^{\frac{1}{p}-\frac{1}{2}} F_a(r\sqrt{\mathcal{L}})(I-\Pi)\chi_l\right\|_{p\to 2},
\end{array}$$
which is uniformly bounded by assumption. Hence, the on-diagonal term $I$ is uniformly bounded.

Let us now deal with the off-diagonal part $II$. Notice that, since the sets $B_k$ are disjoint,

$$\sum_{k=0}^\infty \left\|\chi_k \Pi\chi_l\right\|_{p\to p}=\left\|\Pi\chi_l\right\|_{p\to p}\leq \left\|\Pi\right\|_{p\to p}<+\infty,$$
so that, using $F_a(r\sqrt{\mathcal{L}})\Pi=F_a(0)\Pi$, one gets

$$\sum_{k\geq 0,\,r_{kl}>2r} \left\|\chi_k F_a(r\sqrt{\mathcal{L}})(I-\Pi)\chi_l\right\|_{p\to p}\leq |F_a(0)|\left\|\Pi\right\|_{p\to p}+\sum_{k\geq 0,\,r_{kl}>2r} \left\|\chi_k F_a(r\sqrt{\mathcal{L}})\chi_l\right\|_{p\to p}.$$
However, according to \eqref{FSP}, if $r_{kl}>r$ then $\chi_k F_a(r\sqrt{\mathcal{L}})\chi_l\equiv 0$. Hence, one gets

$$\sum_{k\geq 0,\,r_{kl}>2r} \left\|\chi_k F_a(r\sqrt{\mathcal{L}})(I-\Pi)\chi_l\right\|_{p\to p}\leq |F_a(0)|\left\|\Pi\right\|_{p\to p},$$
which is finite, and bounded independently of $r$. Hence, the off-diagonal term $II$ is uniformly bounded. Finally, we have proved that for $a>\nu\left(\frac{1}{p}-\frac{1}{2}\right)$,

$$\sup_{r>0}\left\|F_a(r\sqrt{\mathcal{L}})\right\|_{p\to p}<+\infty,$$
and it follows that

$$\sup_{t>0}\left\|e^{-t\mathcal{L}}\right\|_{p\to p}<+\infty.$$
This concludes the proof.

\end{proof}

\section{Proof of Theorems \ref{gradient} and \ref{forms}}

Before embarking on the proofs, let us present the proof of Propositions \ref{fgrad_off_equiv} and \ref{grad_off_equiv}. We start with Proposition \ref{fgrad_off_equiv}.

\medskip

\noindent {\em Proof of Proposition \ref{fgrad_off_equiv}:}

\medskip

In all the proof, for simplicity of notations one denotes $\vec{\Delta}$ for the Hodge Laplacian acting on $\Lambda^k T^*M$, without specifying the measure $\mu$ or the degree of the form $k$. We start with the proof of (i). Thus, we assume that $(\vec{G}_{2,s})$ holds for some $s>2$, and one wishes to prove its off-diagonal counterpart $(\vec{G}^{\mathrm{off}}_{2,s})$. Let $x,y\in M$, denote $r=4d(x,y)$, and let $f_1\in  L^2$, $f_2\in L^{s'}\cap L^2$ smooth, and $\mathrm{supp}(f_1)\subset B(y,r)$, $\mathrm{supp}(f_2)\subset B(x,r)$. Consider the analytic function

$$F(z):=\langle f_2,(d+d^*)e^{-z\vec{\Delta}} f_1\rangle,\quad \mathfrak{Re}(z)\geq0.$$
(here, if $f_1$ is a $k$-form, then $f_2=\eta_2\oplus \omega_2$ is the sum of a $k-1$ and a $k+1$-form). Since $\vec{\Delta}$ is self-adjoint, it follows from the Spectral Theorem that for some constant $A>0$,

$$|F(z)|\leq A||f_1||_2||d\eta_2+d^*\omega_2||_2,\quad\mathfrak{Re}(z)\geq0.$$
Therefore, $F(z)$ is bounded in the right half-plane. According to the Davies-Gaffney estimates (see \cite[Lemma 3.8]{AMR}), one can assume that $A$ is chosen large enough so that, for some constant $C>0$,

$$|F(t)|\leq \frac{A}{\sqrt{t}}e^{-\frac{r^2}{Ct}}||f_1||_2||f_2||_2,\quad \forall t>0.$$
Let $\gamma=\frac{r^2}{C}$. Up to increasing the value of the constants $A$ and $C$, one can assume that

$$|F(t)|\leq Ar^{-1}e^{-\frac{\gamma}{t}}||f_1||_2||f_2||_2,\quad \forall 0<t\leq \gamma.$$
For $\mathfrak{Re}(z)>0$, write $z=t+is$, and denote $\alpha=\frac{1}{2}-\frac{1}{s}>0$; then,

$$\begin{array}{rcl}
F(z)&=&\langle f_2,(d+d^*)e^{-t\vec{\Delta}} e^{is\vec{\Delta}}f_1\rangle.\\\\
&=&\langle V_{\sqrt{t}}^{-\alpha}f_2,V_{\sqrt{t}}^{\alpha}(d+d^*)e^{-t\vec{\Delta}} e^{is\vec{\Delta}}f_1\rangle\\\\
&=& V(x,r)^{-\alpha} \langle \left(\frac{V(x,r)}{V_{\sqrt{t}}}\right)^{\alpha}f_2,V_{\sqrt{t}}^{\alpha}(d+d^*)e^{-t\vec{\Delta}} e^{is\vec{\Delta}}f_1\rangle
\end{array}$$
Write $z=t+is$, and let us assume that $\mathfrak{Re}\left(\frac{\gamma}{z}\right)\geq 1$. This implies in particular that $t\leq \gamma$. By doubling, one then has

$$\frac{V(x,r)}{V(w,\sqrt{t})}\lesssim \left(\frac{r^2}{t}\right)^{\nu/2},\quad \forall w\in \mathrm{supp}(f_2)\subset B(x,r).$$
Since $||e^{is\vec{\Delta}}||_{2\to 2}\leq 1$ by self-adjointness, it follows from $(\vec{G}_{2,s})$ that, for some constant $A>0$, and for $\mathfrak{Re}\left(\frac{\gamma}{z}\right)\geq1$,

$$\begin{array}{rcl}
|F(z)|&\leq& \frac{A}{\sqrt{t}}\left(\frac{r^2}{t}\right)^{\alpha\nu/2}V(x,r)^{-\alpha}||f_1||_{2}||f_2||_{s'}\\\\
&\leq& Ar^{-1}\left(\frac{r^2}{t}\right)^{(\alpha\nu+1)/2}V(x,r)^{-\alpha}||f_1||_{2}||f_2||_{s'}
\end{array}$$
Hence,

$$|F(z)|\leq Ar^{-1}V(x,r)^{-\alpha}\left(\frac{\mathfrak{Re}(z)}{\gamma}\right)^{-(\nu\alpha+1)/2}||f_1||_{2}||f_2||_{s'},\quad \mathfrak{Re}\left(\frac{\gamma}{z}\right)\geq1.$$
According to \cite[Proposition 2.3]{CS}, one gets

$$|F(z)|\lesssim Ar^{-1}V(x,r)^{-\alpha}\left(\frac{\gamma}{|z|}\right)^{\alpha \nu+1}\exp\left(-\mathfrak{Re}\left(\frac{\gamma}{z}\right)\right)||f_1||_{2}||f_2||_{s'},\quad\mathfrak{Re}\left(\frac{\gamma}{z}\right)\geq1.$$
Specializing to $z=t>0$, one obtains

$$|F(t)|\lesssim r^{-1}V(x,r)^{-\alpha}\left(\frac{r^2}{t}\right)^{-\nu\alpha-1}e^{-\frac{r^2}{Ct}}||f_1||_{2}||f_2||_{s'},\quad 0<t\leq \frac{r^2}{C},$$
which easily yields by doubling (with a different constant $C$)

$$|F(t)|\lesssim \frac{1}{\sqrt{t}}V(x,\sqrt{t})^{-\alpha}e^{-\frac{r^2}{Ct}}||f_1||_{2}||f_2||_{s'},\quad 0<t\leq \frac{r^2}{C},$$
Hence,

$$\left|\langle f_2,(d+d^*)e^{-t\vec{\Delta}} f_1\rangle\right| \lesssim \frac{1}{\sqrt{t}}V(x,\sqrt{t})^{-\alpha}e^{-\frac{r^2}{Ct}}||f_1||_{2}||f_2||_{s'}.$$
Since the set of smooth $f_2\in L^2\cap L^{s'}$ with support in $B(x,r)$ is dense in $L^{s'}(B(x,r))$, one finds that

$$||\sqrt{t}(d+d^*)e^{-t\vec{\Delta}}||_{L^r(B(y,r))\to L^s(B(x,r))}\lesssim V(x,\sqrt{t})^{-\alpha}e^{-\frac{r^2}{Ct}},\quad 0<t\leq \frac{r^2}{C}.$$
Since $B(x,\sqrt{Ct})\subset B(x,r)$ and $B(y,\sqrt{Ct})\subset B(y,r)$ for $t\leq \frac{r^2}{Ct}$, changing $t$ into $C^{-1}t$ easily yields $(\vec{G}^{\mathrm{off}}_{2,s})$ for $t\leq r^2$. However, $(\vec{G}^{\mathrm{off}}_{2,s})$ for $t\geq r^2$ follows directly from $(\vec{G}_{2,s})$ , so (i) is proved.

\medskip

The proof of (ii) and (iii) is identical to the one of \cite[Corollary 4.16]{CS}, replacing $e^{-zL}$ there by $\sqrt{t}(d+d^*)e^{-t\vec{\Delta}}$. The only point to notice is that the off-diagonal Davies-Gaffney estimates for $\sqrt{t}(d+d^*)e^{-t\vec{\Delta}}$ hold according to \cite[Lemma 3.8]{AMR}. We leave the details to the reader. Contrary to \cite[Corollary 4.16]{CS}, one cannot get the full interval $[q,q']$ for \eqref{fGp}, because the duality argument is not applicable.

$\Box$

\medskip

\noindent {\em Proof of Proposition \ref{grad_off_equiv}:}

\medskip

Proposition \ref{grad_off_equiv} follows from the more general Proposition \ref{fgrad_off_equiv}, except for the first point (i), where it is not assumed that either $r$ nor $s$ is equal to $2$. For the proof of (i), we follow the proof of Proposition \ref{fgrad_off_equiv}, (i), keeping the same notations. We assume that \eqref{grad_r_s} holds for some $1\leq r\leq s\leq \infty$, and one wishes to prove its off-diagonal counterpart \eqref{grad_off}. Let $x,y\in M$, denote $r=4d(x,y)$, and let $f\in  L^r\cap L^2$, $X\in L^{s'}\cap L^2$, $X$ smooth vector field, and $\mathrm{supp}(f)\subset B(y,r)$, $\mathrm{supp}(X)\subset B(x,r)$. Consider the complex-valued function

$$F(z):=\langle X,\nabla e^{-z\Delta} f\rangle,\quad \mathfrak{Re}(z)\geq0.$$
By argument similar to the proof of Proposition \ref{fgrad_off_equiv}, one has

$$|F(z)|\leq A||f||_||\div X||_{},\quad\mathfrak{Re}(z)\geq0,$$
and

$$|F(t)|\leq Ar^{-1}e^{-\frac{\gamma}{t}}||f||_{2}||X||_{2},\quad \forall 0<t\leq \gamma,$$
where where $\gamma=\frac{r^2}{C}$. We claim that the following estimate holds:

\begin{equation}\label{contract}
|F(z)|\leq Ar^{-1-\nu/2}V(x,r)^{-\alpha} |z|^\nu \left(\frac{\mathfrak{Re}(z)}{\gamma}\right)^{-(\nu\alpha+2\nu+1)/2}||f||_{r}||X||_{s'},\quad \mathfrak{Re}\left(\frac{\gamma}{z}\right)\geq1.
\end{equation}
In order to prove \eqref{contract}, we will use the fact that, according to \cite[Corollary 4.5]{CS}, if $p_z(x,y)$ is the heat kernel of $e^{-z\Delta}$, $\mathfrak{Re}(z)\geq0$, then by \eqref{dnu} and \eqref{UE}, for every $v,w\in M$ and $\mathfrak{Re}(z)>0$,

\begin{equation}\label{complex}
|p_z(v,w)|\lesssim \left(\frac{|z|}{\mathfrak{Re}(z)}\right)^\nu \frac{1}{V(w,(\mathfrak{Re}(1/z))^{-1/2})}\exp\left(-\mathfrak{Re}\left(\frac{d^2(v,w)}{Cz}\right)\right).
\end{equation}
The estimate \eqref{complex} implies by standard arguments using \eqref{d} that for every $p\in[1,\infty]$,

\begin{equation}\label{uniform}
||e^{-z\Delta}||_{p\to p}\lesssim \left(\frac{|z|}{\mathfrak{Re}(z)}\right)^\nu,\quad \mathfrak{Re}(z)>0.
\end{equation}
We now assume that $\mathfrak{Re}\left(\frac{\gamma}{z}\right)\geq1$, write $z=t+is$, then

$$F(z)=\langle f_2,\nabla e^{-\frac{t}{2}\Delta}e^{-(\frac{t}{2}+is)\Delta} f_1\rangle.$$
By \eqref{uniform} and \eqref{grad_r_s}, we get as in the proof of Proposition \ref{fgrad_off_equiv}, that for $\mathfrak{Re}\left(\frac{\gamma}{z}\right)\geq1$,

$$\begin{array}{rcl}
|F(z)| &\leq & A\left(\frac{|z|}{\mathfrak{Re}(z)}\right)^\nu r^{-1}V(x,r)^{-\alpha}\left(\frac{\mathfrak{Re}(z)}{\gamma}\right)^{-(\nu\alpha+1)/2}||f||_{r}||X||_{s'} \\\\
&\leq & Ar^{-1-2\nu}V(x,r)^{-\alpha} |z|^\nu \left(\frac{\mathfrak{Re}(z)}{\gamma}\right)^{-(\nu\alpha+2\nu+1)/2}||f||_{r}||X||_{s'}
\end{array}$$
Hence, \eqref{contract} is proved. According to \cite[Proposition 2.4]{CS} with the choice $\exp(g(z))=z^\nu$ for $\mathfrak{Re}(z)\geq0$, one gets for $\mathfrak{Re}\left(\frac{\gamma}{z}\right)\geq1$,

$$\begin{array}{rcl}
|F(z)|&\lesssim & A |z|^\nu r^{-1-2\nu}V(x,r)^{-\alpha}\left(\frac{\gamma}{|z|}\right)^{\alpha \nu+2\nu+      1}\exp\left(-\mathfrak{Re}\left(\frac{\gamma}{z}\right)\right)||f||_{r}||X||_{s'}\\\\
&\lesssim & Ar^{-1} V(x,r)^{-\alpha}\left(\frac{\gamma}{|z|}\right)^{\alpha \nu+\nu+1}\exp\left(-\mathfrak{Re}\left(\frac{\gamma}{z}\right)\right)||f||_{r}||X||_{s'}
\end{array}$$
From this point, the rest of the proof is identical to that of Proposition \ref{fgrad_off_equiv}, and is skipped.

$\Box$

\medskip

In the proof of Theorem \ref{forms}, the following lemma will be key:

\begin{Lem}\label{form1}

Let $(M,\mu)$ be a complete weighted Riemannian manifold, $k\geq1$ be an integer, and $\alpha\in L^2(\Lambda^kT^*M)\cap C^\infty(\Lambda^kT^*M)$. Let $\Pi_k$ be the $L^2$ orthogonal projection onto $\mathrm{Ker}_{L^2}(\vec{\Delta}_{k,\mu})$. Then, for every $t\geq 0$,

$$\Pi_k(de^{-t\vec{\Delta}_{k,\mu}}\alpha )=0,$$
and 

$$\Pi_{k-1}(d^*e^{-t\vec{\Delta}_{k,\mu}}\alpha )=0.$$

\end{Lem}

\begin{proof}

Let $\eta=e^{-t\vec{\Delta}_{k,\mu}}\alpha$, $\omega=d\eta$ and $\chi=d^*\eta$. Then, $\eta\in L^2$, and $\eta$ is in the domain of $\vec{\Delta}_{k,\mu}$. But since $M$ is complete, one has

$$(\vec{\Delta}_{k,\mu}\eta,\eta)_{L^2}=||d\eta||_2^2+||d^*\eta||_2^2.$$
Hence, $\omega\in L^2$ and $\chi\in L^2$. We claim that $\omega$ belongs to $\overline{dC_0^\infty(\Lambda^kT^*M)}^{L^2}$, and $\chi$ belongs to $\overline{d^*C_0^\infty(\Lambda^kT^*M)}^{L^2}$. To see this, let $p\in M$ and $\{\varphi_n\}_{n\in\mathbb{N}}$ be a sequence of smooth, compactly supported functions on $M$, such that $\varphi_n|_{B(p,n)}=1$, $0\leq \varphi_n\leq 1$, and such that for some constant $C>0$, $|d\varphi_n|\leq C$. Let $\omega_n=d(\varphi_n\eta)$, and $\chi_n=d^*(\varphi_n\eta)$. Since

$$\omega_n=d\varphi_n\wedge \eta+\varphi_n\omega,$$
and

$$\chi_n=\varphi_n\chi-\mathrm{int}_{d\varphi_n}\eta,$$
it is easy to see that $\omega_n\to \omega$ and $\chi_n\to\chi$ in $L^2$. But $\omega_n$ belongs to $dC_0^\infty(\Lambda^kT^*M)$, and $\chi_n$ belongs to $d^*C_0^\infty(\Lambda^kT^*M)$, hence the claim is proved. 

Since $M$ is complete, the Hodge-De Rham decomposition writes:

$$L^2(\Lambda^kT^* M)=\mathrm{Ker}_{L^2}(\vec{\Delta}_{k,\mu})\oplus_\perp \overline{dC_0^\infty(M)}^{L^2}\oplus_\perp \overline{d^*C_0^\infty(\Lambda^2T^*M)}^{L^2}.$$
The projection $\Pi_k$ is precisely the projection on the first factor of the above orthogonal decomposition. Since $\omega$ belongs to $\overline{dC_0^\infty(\Lambda^kT^*M)}^{L^2}$, it follows that 

$$\Pi_k\omega=0.$$
Using the Hodge-De Rham decomposition for $L^2(\Lambda^{k-1}T^* M)$, we see in a similar way that $\Pi_{k-1}\chi=0.$

\end{proof}

\noindent{\em Proof of Theorem \ref{forms}:}

The assumptions together with Theorem \ref{pi_heat} imply in particular that for all $q\leq r\leq s\leq q'$,

\begin{equation}\label{Eq1}
\sup_{t>0}||V_{\sqrt{t}}^{\frac{1}{r}-\frac{1}{s}}e^{-t\vec{\Delta}_{k+1,\mu}}(I-\Pi_{k+1})||_{r\to s}<+\infty
\end{equation}
and

\begin{equation}\label{Eq2}
\sup_{t>0}||V_{\sqrt{t}}^{\frac{1}{r}-\frac{1}{s}}e^{-t\vec{\Delta}_{k,\mu}}(I-\Pi_{k})||_{r\to s}<+\infty.
\end{equation}
We first claim that for every $s\in [2,q']$,

\begin{equation}\label{Eq3}
\sup_{t>0}||V_{\sqrt{t}}^{\frac{1}{2}-\frac{1}{s}}\sqrt{t}de^{-t\vec{\Delta}_{k,\mu}}||_{2\to s}<+\infty,
\end{equation}
and

\begin{equation}\label{Eq4}
\sup_{t>0}||V_{\sqrt{t}}^{\frac{1}{2}-\frac{1}{s}}\sqrt{t}d^*e^{-t\vec{\Delta}_{k,\mu}}||_{2\to s}<+\infty.
\end{equation}
Let $\alpha\in C_0^\infty(\Lambda^{k}T^*M)$. According to \eqref{Eq1},

\begin{equation}\label{Eq5}
||V_{\sqrt{t}}^{\frac{1}{2}-\frac{1}{s}}e^{-t\vec{\Delta}_{k+1,\mu}}(I-\Pi_{k+1})\left[\sqrt{t}de^{-t\vec{\Delta}_{k,\mu}}(I-\Pi_{k})\alpha\right]||_s\lesssim ||\sqrt{t}de^{-t\vec{\Delta}_{k,\mu}}(I-\Pi_{k})\alpha||_2.
\end{equation}
But one has

$$\begin{array}{rcl}
e^{-t\vec{\Delta}_{k+1,\mu}}(I-\Pi_{k+1})\left[de^{-t\vec{\Delta}_{k,\mu}}(I-\Pi_{k})\alpha\right]&=&e^{-t\vec{\Delta}_{k+1,\mu}}(I-\Pi_{k+1})\left[de^{-t\vec{\Delta}_{k,\mu}}\alpha\right]\\\\
&=&de^{-2t\vec{\Delta}_{k,\mu}}\alpha,
\end{array}$$
where in the first equality we have used that $e^{-t\vec{\Delta}_{k,\mu}}\Pi_{k}=\Pi_{k}$ and $d\Pi_{k}=0$, and in the second inequality we have used that $\Pi_{k+1} de^{-t\vec{\Delta}_{k,\mu}}\alpha=0$ according to Lemma \ref{form1}. On the other hand, since $de^{-t\vec{\Delta}_{k,\mu}}\Pi_{k}=e^{-t\vec{\Delta}_{k+1,\mu}}d\Pi_{k}=0$, one has

$$\begin{array}{rcl}
||\sqrt{t}de^{-t\vec{\Delta}_{k,\mu}}(I-\Pi_{k})\alpha||_2&=& ||\sqrt{t}de^{-t\vec{\Delta}_{k,\mu}}\alpha||_2\\\\
&\leq& \sqrt{t}\left(||de^{-t\vec{\Delta}_{k,\mu}}\alpha||_2^2+||d^*_{\mu} e^{-t\vec{\Delta}_{k,\mu}}\alpha||_2^2\right)^{1/2}\\\\
&\leq& \langle t\vec{\Delta}_{k,\mu}e^{-t\vec{\Delta}_{k,\mu}}\alpha,\alpha\rangle^{1/2}\\\\
&\lesssim& ||\alpha||_2.
\end{array}$$
Therefore, according to \eqref{Eq5}, one obtains

$$||V_{\sqrt{t}}^{\frac{1}{2}-\frac{1}{s}}\sqrt{t}de^{-2t\vec{\Delta}_{k-1,\mu}}\alpha||_s\lesssim ||\alpha||_2,$$
and using the doubling property, we get \eqref{Eq3}. The proof of \eqref{Eq4} is completely similar, using \eqref{Eq2} instead of \eqref{Eq1}, and is skipped. 

As a consequence of \eqref{Eq3} and \eqref{Eq4}, one concludes that $(\vec{G}_{2,s})$ holds, for all $s\in [2,q']$. By a duality argument, using the hypothesis on the kernel of harmonic forms of degree $k-1$, one concludes that $(\vec{G}_{r,2})$ holds too, for all $r\in [q,2]$. According to Proposition \ref{fgrad_off_equiv}, one concludes first that the off-diagonal versions $(\vec{G}^{\mathrm{off}}_{2,s})$ and $(\vec{G}^{\mathrm{off}}_{r,2})$ hold, for any $q\leq r\leq s\leq q'$, and then that for every $p\in [q,q']$, \eqref{fGp} holds. By interpolation between \eqref{fGp} and $(\vec{G}_{2,s})$, $(\vec{G}_{r,2})$, $q\leq r\leq 2\leq s\leq q'$, $p\in [q,q']$, one obtains that \eqref{fgrad_r_s} holds for all $q\leq r\leq s\leq q'$. The proof is complete.

\cqfd

{\em Proof of Theorem \ref{gradient}:}

\medskip

The first part of Theorem \ref{gradient}, that is the part for $q>1$, is a particular case of Theorem \ref{forms}. It remains to explain why \eqref{G} holds true, in the case $q=1$. First of all, in this case, according to Theorem \ref{forms}, one gets $(G_{1,\infty})$. According to Proposition \ref{grad_off_equiv}, $(G^{\mathrm{off}}_{1,\infty})$ follows. That is,

$$V(x,\sqrt{t})||\chi_{B(x,\sqrt{t})}\sqrt{t}\nabla e^{-t\Delta}\chi_{B(y,\sqrt{t})}||_{1\to \infty}\lesssim e^{-\frac{d^2(x,y)}{Ct}},\quad t>0,\,x,y\in M.$$
But using the well-known fact that $L^1\to L^\infty$ estimates are equivalent to pointwise estimates on the kernel, one concludes that \eqref{G} holds. 

\cqfd

\section{$L^2$-cohomologies and the condition \eqref{Ker}}

In this section, we explain how the assumption \eqref{Ker} is related to the $L^2$ cohomology. The classical De Rham theorem asserts that, on a smooth compact Riemannian manifold, the cohomology with real coefficients can be computed using differential forms, and more precisely that $H^k(M,\R)$ is isomorphic to the quotient

$$\frac{\{\alpha\in C^\infty(\Lambda^k T^*M)\,;\,d\alpha=0\}}{dC^\infty(\Lambda^{k-1}T^*M)}.$$
Furthermore, by the Hodge theorem if $M$ is compact without boundary, then the space of ($L^2$) harmonic $k$-forms is isomorphic to $H^k(M,\R)$.
However, if $M$ is non-compact, in general these cohomology spaces have infinite dimension. Thus, alternative cohomologies have to be defined. We recall the definition of two cohomologies for $M$, the reduced $L^2$ cohomology and the cohomology with compact support, for more details and references see \cite{C6}. Firstly, $H_c^k(M)$, the space of {\em cohomology with compact support} of degree $k$, is defined as the quotient of 

$$Z_c^k(M)=\mathrm{Ker}\{d:C_0^\infty(\Lambda^k T^*M)\}$$
by 

$$B_c^k(M)=dC_0^\infty(\Lambda^k T^* M).$$
Secondly, $H_{(2)}^k(M)$, the space of $L^2$ {\em reduced cohomology} of degree $k$, is defined as the quotient of

$$Z_{(2)}^k(M)=\{\alpha\in L^2(\Lambda^k  T^*M)\,;\,d\alpha=0\}$$
by 

$$B_{(2)}^k(M)=\overline{dC_0^\infty(\Lambda^k T^*  M)}^{L^2}.$$
Note that in the definition of $Z_{(2)}^k(M)$, the equation $d\alpha=0$ is intended in the weak sense. Given that 

$$B_c^k(M)\subset B_{(2)}^k(M),$$
there is a natural map from $H_c^k(M)$ into $H_{(2)}^k(M)$, that associates to each $[\eta]_c$, $\eta\in C_0^\infty$, its class $[\eta]_{(2)}$ in $L^2$ reduced cohomology. In \cite{C6}, one finds geometric conditions ensuring that the map $H_c^1(M)\to H_{(2)}^1(M)$ is injective, for example it is the case if all the ends of $M$ are non-parabolic. In the spirit of Hodge theorem, the spaces of $L^2$ reduced cohomology have an interpretation in terms of harmonic forms. Let $\mathscr{H}^k(M)$ denotes the space of $L^2$ harmonic $k$-forms, i.e.

$$\mathscr{H}^k(M)=\{\alpha\in L^2(\Lambda^k T^* M)\,;\,d\alpha=0,\,d^*\alpha=0\}.$$ 
It is well-known that if $M$ is complete, then there is a Hodge decomposition:

$$L^2(\Lambda^k T^* M)=\mathscr{H}^k(M)\oplus_\perp \overline{dC_0^\infty(\Lambda^{k-1}T^*M)}^{L^2}\oplus_\perp \overline{d^*C_0^\infty(\Lambda^{k+1} T^*M)}^{L^2},$$
and that moreover

$$\mathscr{H}^k(M)=\{\alpha\in \Lambda^k T^*M\,;\,\vec{\Delta}_k\alpha=0\},$$
where $\vec{\Delta}_k=dd^*+d^*d$ is the Hodge Laplacian acting on $k$-forms. The Hodge decomposition readily implies that

$$\mathscr{H}^k(M)\simeq H^k_{(2)}(M).$$
In the case $M$ has a boundary, additional boundary conditions for harmonic forms have to be introduced in order that the above equality holds (see \cite{C6}). Moreover, it turns out that all the above-mentioned results concerning $L^2$-cohomology (including the Hodge decomposition) also hold in the weighted setting (see \cite{Bue}).

\bigskip

\noindent {\em Proof of Theorems \ref{Ker-Sob} and \ref{Ker-Sob_k}:} Under the assumptions of Theorem \ref{Ker-Sob}, the volume estimates

\begin{equation}\label{vol_eucl}
V(x,r)\simeq r^n,\,\forall r>0
\end{equation}
hold, according to \cite{C8}. It is thus clear that Theorem \ref{Ker-Sob} follows from Theorem \ref{Ker-Sob_k}. So, it is enough to prove Theorem \ref{Ker-Sob_k}. 

According to \cite[Proposition 6.7]{D3}, under our assumptions the potential $W(x)=||\mathscr{R}_{k}(x)||_x$ belongs to the Kato class at infinity, i.e.

$$\lim_{r\to\infty} \sup_{x\in M}\int_{M\setminus B(x_0,r)}G(x,y)W(y)\,dy=0.$$
Let us fix a radius $R$ such that

$$\sup_{x\in M}\int_{M\setminus B(x_0,R)}G(x,y)W(y)\,dy<1,$$
and denote

$$\mathcal{R}_{-,0}(x)=\mathbf{1}_{B(x_0,R)}(x)\mathscr{R}_{k}(x),$$
and

$$\mathcal{R}_{-,\infty}(x)=\mathbf{1}_{M\setminus B(x_0,R)}(x)\mathscr{R}_{k}(x).$$
As in \cite{CDS}, let 

$$\mathcal{H}=\nabla^*\nabla+\mathcal{R}_+-\mathcal{R}_{-,\infty},$$
and 

$$\mathcal{W}_0=\mathcal{R}_{-,0},$$
Let $\omega\in \mathrm{Ker}_{L^2}(\vec{\Delta})$. Then, by Lemma \ref{minig},

$$\omega=-\mathcal{H}^{-1}\mathcal{W}_0\omega.$$
By \cite[Proposition 2.3]{CDS}, the semi-group $e^{-t\mathcal{H}}$ has Gaussian heat kernel estimates. The volume assumption \eqref{vol_eucl} on $M$ implies by integration that the Green operator $\mathcal{H}^{-1}$ has a kernel $h(x,y)$ satisfying:

$$||h(x,y)||_{y,x}\lesssim d(x,y)^{2-n},\quad \forall x,\,y,\,d(x,y)\geq1.$$
Thus, since $\mathcal{W}_0$ is compactly supported,

\begin{equation}\label{asym-form}
|\omega(x)|\lesssim d(x_0,x)^{2-n}.
\end{equation}
By the volume assumption \eqref{vol_eucl}, we deduce from the above estimate that $\omega$ is in $L^q$ for all $q\in (\frac{n}{n-2},2)$. We claim that one can improve this to get the interval $(\frac{n}{n-1},2)$. The Bochner formula 

$$\vec{\Delta}_{k}\omega=-\mathrm{Tr}\nabla^2\omega+\mathscr{R}_{k}\omega$$
implies, since $\omega$ is harmonic, that

\begin{equation}\label{B1}
-\frac{1}{2}\Delta |\omega|^2=|\nabla \omega|^2+(\mathscr{R}_k\omega,\omega).
\end{equation}
Recall the refined Kato inequality for harmonic $k$-forms (see \cite[Theorem 3.8]{CiZ}),

$$|\nabla \omega|^2\geq \kappa|\nabla |\omega||^2,$$
with 

$$\kappa=\left\{ \begin{array}{lcl}
\frac{n-k+1}{n-k},\, 1\leq k\leq \frac{n}{2}\\\\
\frac{k+1}{k},\, \frac{n}{2}\leq k\leq n-1
\end{array}\right.$$
Thus, by \eqref{B1},

\begin{equation}\label{B2}
|\omega|\Delta |\omega|+(\mathscr{R}_k\omega,\omega)\leq (1-\kappa)|\nabla |\omega||^2.
\end{equation}
Let $\alpha>0$,  $\varphi_\alpha:=|\omega|^{\alpha}$, and $V_\alpha(x):=\alpha ||\mathscr{R}_k(x))||_{x,x}$. Then one gets from \eqref{B2} that

$$(\Delta-V_\alpha)\varphi_\alpha\leq \alpha \left((1-\kappa)-(\alpha-1)\right)|\omega|^{\alpha-2}|\nabla|\omega||^2.$$
Hence, for $\alpha=2-\kappa$, denoting $V=V_{2-\kappa}$ and $\varphi=\varphi_{2-\kappa}$, we get

\begin{equation}\label{B3}
(\Delta-V)\varphi\leq 0.
\end{equation}
Let $V_\infty(x):=\mathbf{1}_{B(x_0,r)}(x)V(x)$, where the radius $r>0$ is chosen so that

$$\sup_{x\in M}\int_{M\setminus B(x_0,r)}G(x,y)V(y)\,dy<1,$$
which is possible since $V$ belongs to the Kato class at infinity. Then, $\varphi$ is a subsolution of $P=\Delta-V_\infty$ outside a compact set, and furthermore by \eqref{asym-form},

$$\lim_{x\to\infty}\varphi(x)=0.$$
We claim that there exists a constant $C>0$ such that

\begin{equation}\label{min-growth}
\varphi(x)\leq C G_P(x,x_0),\quad d(x,x_0)\geq 2r,
\end{equation}
where $G_P$ is the positive minimal Green function of $P$. Note that by \cite[Proposition 2.3]{CDS}, the operator $P$ is non-negative and the semi-group $e^{-tP}$ has Gaussian heat kernel estimates, thus by the volume assumption \eqref{vol_eucl}, $G_P$ exists and is finite, and one has

\begin{equation}\label{Green_est}
G_P(x,y)\lesssim d(x,y)^{2-n}.
\end{equation}
According to \cite[Theorem 3.2]{D3}, there exists $h$ such that $(\Delta-V_\infty)h=0$ and for some constant $C>0$,

$$C^{-1}\leq h\leq C.$$
Let $U$ be a smooth domain such that $B(x_0,r)\subset U\subset B(x_0,2r)$.  Denote $w(x)=G_P(x,x_0)$, and let $\{\Omega_n\}_{n\in\mathbb{N}}$ be an exhaustion of $M$ by smooth domains, such that $U\subset \Omega_0$. Let $w_n$ be the Green function of $P$ on $\Omega_n$ with pole at $x_0$ and Dirichlet boundary conditions on $\partial \Omega_n$. By the maximum principle, the sequence $\{w_n\}_{n\in\mathbb{N}}$ is increasing, and converges pointwise to $w$. Denote

$$c=\max_{x\in\partial U}\frac{\varphi(x)}{w(x)},$$
and let $\varepsilon>0$. Since $h\simeq 1$ and $\varphi$ tends to zero at infinity, there exists $N\in\mathbb{N}$ such that, for $n\geq N$,

$$\varphi(x)\leq \varepsilon h(x),\quad\forall x\in \partial \Omega_n.$$
By the maximum principle, we thus obtain that for every $n\geq N$,

$$\varphi(x)\leq (c+1)w_n(x)+\varepsilon h(x),\quad\forall x\in \Omega_n\setminus U.$$
Letting $n\to\infty$, we obtain

$$\varphi(x)\leq (c+1)w(x)+\varepsilon h(x),\quad \forall x\in M\setminus U.$$
Letting $\epsilon\to 0$, we get \eqref{min-growth} with $C=c+1$.

Now, \eqref{min-growth} and the estimate \eqref{Green_est} of $G_P$ imply that

$$|\omega(x)|\lesssim d(x,x_0)^{\frac{2-n}{2-\kappa}},$$
which, taking into account the fact that $M$ has Euclidean volume growth, easily yields

$$\omega \in L^q, \quad\forall q\in (q^*,2).$$
Thus, we have proved that

$$\mathrm{Ker}_{L^2}(\vec{\Delta})\subset L^q,\quad\forall q\in \left(q^*,2\right).$$
\cqfd

In the rest of the article, we focus on the gradient estimates of the scalar heat kernel. As explained in the introduction, our goal will be to improve the result of Corollary \ref{grad-Sob}, under additional assumptions on the geometry of the manifolds at infinity. This will be done in the next section. For the moment, we content ourselves to present general results pertaining to the condition \eqref{Ker} for the Hodge Laplacian on $1$-forms. We first present a kind of converse to Theorem \ref{gradient}:

\begin{Pro}\label{converse}

Assume that $(M,d,\mu)$ is a complete, non-parabolic weighted Riemannian manifold satisfying \eqref{d} and \eqref{P}. Assume that \eqref{Gp} holds for some $p\in (2,\infty]$, and that every class $[\alpha]_{(2)}$ in the cohomology space $H_{(2)}^1(M)$ has a representative $\alpha$ in $L^2\cap L^{p'}$. Then, for every $q\in (p',2)$, \eqref{Ker} for $\mathcal{L}=\vec{\Delta}_{\mu}$, the weighted Hodge Laplacian on $1$-forms, holds, i.e.

$$\mathrm{Ker}_{L^2}(\vec{\Delta}_\mu )\subset L^q.$$

\end{Pro}

\begin{Rem}
{\em 

\begin{enumerate}

\item[(i)] The assumption on the cohomology classes in particular holds if the map $H_c^1(M)\to H^1_{(2)}(M)$ is surjective. This is true for example if $M$ is Euclidean outside a compact set (see Theorem \ref{coho-comp}).

\item[(ii)] As the proof will show, instead of assuming \eqref{P} and \eqref{Gp}, one can assume instead that the Riesz transform is bounded on $L^p$ and $L^{p'}$.

\end{enumerate}
}
\end{Rem}

\begin{proof}

Let $\omega\in \mathrm{Ker}_{L^2}(\vec{\Delta})=\mathscr{H}^1(M)$, then the assumption made on the cohomology classes implies that there exists $\eta \in L^2\cap L^{p'}$ such that

$$[\eta]_{(2)}=[\omega]_{(2)}.$$
Equivalently, if one lets $\varphi=\eta-\omega$, then

$$\varphi \in \overline{dC_0^\infty(M)}^{L^2}.$$
So,

$$\eta=\omega+\varphi$$
is the Hodge decomposition of $\eta$, and this implies that

$$\mathscr{P}\eta=\varphi,$$
where 

$$\mathscr{P}=d\Delta^{-1}d^*$$
is the Hodge projector onto the space of exact forms $\overline{dC_0^\infty(M)}^{L^2}.$ Notice that one can write $\mathscr{P}$ as

$$\mathscr{P}=(d\Delta^{-1/2})(d\Delta^{-1/2})^*.$$
According to \cite{CD1} and \cite{ACDH}, the assumptions of Proposition \ref{converse} imply that the Riesz transform $d\Delta^{-1/2}$ is bounded on $L^s$ for all $s\in (1,p)$. Thus, $\mathscr{P}$ is bounded on $L^q$ for all $q\in (p',2)$. Since $\eta\in L^q$ for all $q\in (p',2)$, one obtains that

$$\varphi=\mathscr{P}\eta \in L^q,\quad \forall q\in (p',2).$$
Thus, $\omega\in L^q$, for all $q\in (p',2).$ This shows that \eqref{Ker} holds for all $q\in (p',2)$.

\end{proof}
We now further investigate the validity of assumption \eqref{Ker} for the Hodge Laplacian on $1$-forms. We start by recalling a result which follows directly from the proof of \cite[Lemma 3.4]{C5}, even if it was not stated explicitly there:

\begin{Lem}\label{Car}

Let $M$ be a complete Riemannian manifold satisfying the Poincar\'e inequalities \eqref{P}, and whose volume growth is Euclidean at infinity: there exists $n>2$ such that

$$V(x,t)\simeq t^n,\quad \forall x\in M,\,\forall t\geq 1.$$
Let $p$ be a point in $M$, and let $u$ be a smooth function on $M$ such that $du$ is in $L^2$. Then, there exist two constants $c\in \R$, $C>0$, such that, for every $r\geq 1$,

$$\int_{B(p,r)}|u-c|^2\leq Cr^2.$$

\end{Lem}
As a consequence of Lemma \ref{Car}, one obtains the following result:

\begin{Lem}\label{harmo-out}

Let $M$ be a complete Riemannian manifold satisfying the Poincar\'e inequalities \eqref{P}, the Sobolev inequality of exponent $n>2$, and whose volume growth is Euclidean at infinity: there exists a constant $C>0$ such that

$$C^{-1}t^n\leq V(x,t)\leq C t^n,\quad \forall x\in M,\,\forall t\geq 1.$$
Let $u$ be a function on $M$ which is harmonic outside a compact set and such that $du$ is in $L^2$. Let $r(x)$ denotes the distance in $M$ to a reference point $p$. Then, there exists a constant $c\in \R$ such that, as $x\to\infty$,

$$u(x)-c=O(r(x)^{2-n}).$$

\end{Lem}

\begin{proof}

Denote by $K$ a compact set such that $\Delta u=0$ outside of $K$. One first claims that there exists $c\in \mathbb{R}$ such that

$$\lim_{x\to\infty} u(x)=c.$$
The Sobolev inequality together with the Euclidean volume growth allows one to get the mean value inequality for subharmonic functions by the Moser iteration scheme. This implies that, for every $\delta<1$, and for every ball $B$ of radius $r(B)\geq 1$, disjoint from $K$, 

$$\sup_{x\in\delta B}|u(x)-c|\lesssim \frac{1}{r(B)^n}\int_B |u-c|^2.$$
Assume that $K$ is contained in $B(x_0,R)$ for some $R>0$, and denote $r(x)=d(x,x_0)$. For $x\in M\setminus B(x_0,2R)$, applying the above inequality to $B=B(x,\frac{r(x)}{2})$, on which $u$ is harmonic, one obtains that

$$|u(x)-c|^2\lesssim \frac{1}{r(x)^n}\int_{B(x_0,r(x))}|u-c|^2\lesssim r(x)^{2-n},$$
where in the last inequality we have used Lemma \ref{Car}. Since $n>2$, this implies that

$$\lim_{x\to\infty} u(x)=c.$$
Now, up to substracting a constant to $u$, one can assume that $c=0$. Let

$$g(x)=\int_MG(x,y)\Delta u(y)\,dy=\int_{B(x_0,R)}G(x,y)\Delta u(y)\,dy.$$
Since the Sobolev inequality implies that 

$$G(x,y)\lesssim r(x)^{2-n},\quad\forall y\in B(x_0,R),$$
in order to conclude the proof, it is enough to see that 

$$g\equiv f.$$
But $v:=g-f$ is harmonic on $M$, and

$$\lim_{x\to \infty}v(x)=0,$$
which implies by the maximum principle that $v\equiv0$. Hence, the lemma is proved.

\end{proof}
We can now analyze more precisely the behaviour of harmonic $1$-forms at infinity:

\begin{Pro}\label{harmo}

Assume that $M$ is a complete Riemannian manifold on which the Poincar\'e inequalities \eqref{P} hold, as well as the Sobolev inequality of exponent $n$, and whose volume growth is Euclidean at infinity:

$$V(x,t)\simeq t^n,\quad \forall x\in M,\,\forall t\geq1.$$
Assume also that the natural map

$$H_c^1(M)\to H_{(2)}^1(M)$$
is an surjective. Then, every harmonic $1$-form $\omega$ on $M$ writes

$$\omega=\varphi+df,$$
with

\begin{itemize}

\item[(i)] $\varphi\in C_0^\infty(\Lambda^1 T^* M)$, $d\varphi=0$.

\item[(ii)] $\Delta f=0$ outside a compact set.

\item[(iii)] $f(x)=O(r(x)^{2-n})$ as $x\to \infty$, where $r(x)=d(x,x_0)$ for some fixed $x_0\in M$. In particular,

$$\lim_{x\to\infty}f(x)=0.$$

\item[(iv)] Let $x_0\in M$ be fixed. For $R\gg 1$, the flux of $f$ through $\partial B(x_0,R)$ vanishes, i.e.

$$\int_{\partial B(x_0,R)}\frac{\partial f}{\partial \nu}=0.$$

\end{itemize}

\end{Pro}

\begin{proof}

Let $\omega\in \mathscr{H}^1(M)$. By elliptic regularity, $\omega$ is smooth. The assumption on the cohomology implies that there is $\varphi\in C_0^\infty(\Lambda^1T^* M)$ and $\eta\in \overline{dC_0^\infty}^{L^2}$ such that

$$\omega=\varphi+\eta.$$
In particular, $\eta$ is smooth, and since the class of $\eta$ is zero in $L^2$ reduced cohomology, \cite[Lemma 1.11]{C4} implies that there exists a smooth function $f$ on $M$, such that

$$\eta=df.$$
Since $d\omega=0$ and $d\eta=0$ (weakly), one has $d\varphi=0$. Also, since $d^*\omega=0$, one has $d^*d f=\Delta f=0$ outside the (compact) support of $\varphi$. This proves (i) and (ii), up to adding a constant to $f$. The claim (iii) follows directly from (ii) and Lemma \ref{harmo-out}.

We finally prove (iv). First, by Green's formula,

$$\int_{B(x_0,R)} \Delta f=\int_{\partial B(x_0,R)}\frac{\partial f}{\partial\nu}.$$
Now, since $d^*\omega=0$ and $\omega=\varphi+df$, one has

$$\Delta f=-d^*\varphi.$$
We take $R>0$ large enough so that the support of $\varphi$ is included in $B(x_0,R)$. Let $X$ be the vector field defined by $X=\varphi^\flat$, then $-d^*\varphi=\div X$, and Stokes theorem implies that

$$\begin{array}{rcl}
\int_{\partial B(x_0,R)}\frac{\partial f}{\partial\nu}&=&-\int_{B(x_0,R)}d^*\varphi\\\\
&=&\int_{B(x_0,R)}\div X\\\\
&=&\int_{\partial B(x_0,R)} \langle X,\nu\rangle \\\\
&=&0,
\end{array}$$
since $X$ vanishes identically on $\partial B(x_0,R)$. This proves (iv).

\end{proof}

\section{Manifolds that are conical at infinity}

In this section, we analyze the asymptotic behaviour of certain harmonic functions on manifolds that have a special cone structure at infinity. We first have the following result concerning the $L^2$-cohomology: 

\begin{Thm}\label{coho-comp}

Let $M\simeq_\infty \mathcal{C}(X)$ be an $n$-dimensional  conical manifold at infinity. Assume that $n>4$ and that $\mathrm{Ric}(X)\geq (n-2)\bar{g}$. Then, the natural map $H^1_c(M)\to H^1_{(2)}(M)$ is an isomorphism.

\end{Thm}

\begin{Rem}
{\em 

In the case $M$ is (locally) Euclidean at infinity, i.e. $X$ is (a quotient of) $\R^n$, the result of Theorem \ref{coho-comp} follows from the proof of \cite[Proposition 4.3]{C3} (see also \cite{C2}), without the assumption that $n>4$. Our proof, which is inspired by -and uses results of- \cite{C7}, is different, although the result is most probably already known to experts.

}
\end{Rem}

 We start with the following lemma:

\begin{Lem}\label{coho-out}

Let $p$ be the vertex of the cone $\mathcal{C}(X)$.  Let $u$ be a harmonic function on $\mathcal{C}(X)\setminus B(p,1)$, with $du\in L^2$. If $\frac{\partial u}{\partial r}\Big|_{r=1}=0$, then $u$ is constant.

\end{Lem}

\begin{proof}

We work in spherical coordinates $(r,x)$, $r>0$, $x\in X$ on $\mathcal{C}(X)$. Denote by $\{\lambda_k\}_{k\in \mathbb{N}}$ the spectrum of the Laplacian on $(X,\bar{g})$, and by $\{\varphi_k\}_{k\in\mathbb{N}}$ an associated complete orthonormal family of eigenfunctions. So,

$$\Delta_{X}\varphi_k=\lambda_k\varphi_k.$$
One has $\lambda_0=0$ and $\varphi_0=const.$ One writes a Fourier-type expansion for the function $u$ in $x$:

$$u(r,x)=\sum_{k=0}^\infty a_k(r)\varphi_k(x),$$
and for every $r\in (0,\infty)$,

$$\int_{S^{n-1}}|u(r,x)|^2\,dx=\sum_{k=0}^\infty |a_k(r)|^2<+\infty.$$
By Lemma \ref{harmo-out}, one can assume that 

$$\lim_{r\to\infty}u(r,x)=0,$$
uniformly in $x\in X$, which implies that

\begin{equation}\label{inf-behave}
\lim_{r\to\infty}\int_{X}|u(r,x)|^2\,dx=\lim_{r\to\infty}\sum_{k=0}^\infty |a_k(r)|^2=0.
\end{equation}
The harmonicity of $u$ outside a compact set is equivalent to the fact that the function $a_k$ is solution of the following ODE for $r\geq1$:

$$\left(-\frac{d^2}{d r^2}-\frac{n-1}{r}\frac{d}{d r}+\frac{\lambda_k}{r^2}\right)a=0.$$
This is an Euler equation, and a basis of solutions is given by the functions $r^{\alpha^\pm_k}$, with

$$\alpha_k^\pm=-\frac{n-2}{2}\pm \sqrt{\left(\frac{n-2}{2}\right)^2+\lambda_k}.$$
So, for all $k\geq1$,

$$a_k(r)=\mu_kr^{\alpha_k^+}+\gamma_k r^{\alpha_k^-}$$
and 

$$a_0(r)=\mu_0+\gamma_0r^{-n+2}.$$
Note that for $k\geq1$, $\alpha_k^+>0$ and $\alpha_k^-<0$, so \eqref{inf-behave} implies that $\mu_k=0$ for all $k\geq0$. Therefore, for all $x\in X$ and $r\geq1$,

$$u(r,x)=\sum_{k=0}^\infty \gamma_kr^{\alpha_k^-}\varphi_k(x).$$
Then,

$$0=\frac{\partial u}{\partial r}\Big|_{r=1}=\sum_{k\geq0}\alpha_k^-\gamma_k\varphi_k(x),\quad\forall x\in X,$$
which implies by uniqueness of the coefficients in the Fourier-type decomposition that $\gamma_k=0$ for all $k\geq0$, i.e. $u\equiv 0$.

\end{proof}

\noindent {\em Proof of Theorem \ref{coho-comp}:}

Let $U\subset M$ and $V_1\subset \mathcal{C}(X_1),\ldots,V_k\subset \mathcal{C}(X_k)$ be open sets such that $M\setminus U$ is isometric to $\sqcup_{i=1}^k\mathcal{C}(X_i)\setminus V_i$. Without loss of generality, we will assume that $V= B(p,1)$, where $p$ is the vertex of the cone. We first show that the $L^2$ reduced cohomology in degree one of the manifold with boundary $M\setminus U$ is trivial. It is well-known (see \cite{C6}) that 

$$H_{(2)}^k(M\setminus U)\simeq \mathscr{H}^k_{abs}(M\setminus U),$$
where $\mathscr{H}^k_{abs}(M\setminus U)$ is the space of $k$-forms $\alpha$ that are $L^2$ on $M\setminus U$, such that $d\alpha=0$, $d^*\alpha=0$, and if $\nu$ denotes a unit normal to $\partial U$, such that for all $k=0,1,\ldots,n$,

$$\mathrm{int}_{\nu}\alpha=0.$$
We claim that for a conical manifold $\mathcal{C}(X)$,

\begin{equation}\label{coho_abs}
\mathscr{H}^1_{abs}(\mathcal{C}(X)\setminus B(p,1))=\{0\}.
\end{equation}
Let $\alpha$ belongs to $\mathscr{H}^1_{abs}(\mathcal{C}(X)\setminus B(p,1))$. The manifold $X$ being a deformation retract of $\mathcal{C}(X)\setminus B(p,1)$, the (unreduced) De Rham cohomologies of $X$ and of $\mathcal{C}(X)\setminus B(p,1)$ are equal. By assumption, the Ricci curvature on $X$ is positive, hence the well-known Bochner method (see e.g. \cite[Theorem 6.56]{Besse}) implies that $\mathscr{H}^1(X)=\{0\}$. By the Hodge theorem, this implies that $H^1(X,\R)=\{0\}$. i.e. every closed form on $\mathcal{C}(X)\setminus B(p,1)$ is exact: therefore, there exists $\beta\in C^\infty(\mathcal{C}(X)\setminus B(p,1))$ such that

$$\alpha=d\beta.$$
Since $d\alpha=d^*\alpha=0$, the function $\beta$ must be harmonic on $\mathcal{C}(X)\setminus B(p,1)$. Furthermore, $\mathrm{int}_{\nu}\alpha=0$ is equivalent to 

$$\frac{\partial \beta}{\partial r}\Big|_{r=R}=0.$$
Taking into account that $\alpha=d\beta$ is $L^2$, we get by Lemma \eqref{coho-out} that $\beta$ has to be constant, and thus $\alpha\equiv0$. Clearly, \eqref{coho_abs} implies that 

$$\mathscr{H}^1_{abs}(M\setminus U)=\{0\}.$$
Therefore,

$$H_{(2)}^1(M\setminus U)=\{0\}.$$
Let us now complete the proof of the theorem. First, a direct computation shows that the assumption on the Ricci curvature of $X$ implies that $M$ has non-negative Ricci curvature outside $U$. Furthermore, as indicated previously, the Sobolev inequality of dimension $n>4$ holds. According to \cite[Th\'eor\`eme 3.1]{C7}, there is a long exact sequence

$$\cdots \to H^k(U,\partial U) \to H_{(2)}^k(M)\to H^k_{(2)}(M\setminus U)\to H^{k+1}(U,\partial U)\to \cdots$$
and since $H^1_{(2)}(M\setminus U)=\{0\}$, the map 

$$H^1(U,\partial U) \to H_{(2)}^1(M)$$
is surjective.
But according to the Excision Lemma (see \cite{God}), the relative cohomology of $(U,\partial U)$ is isomorphic to the cohomology with compact support of the interior of $U$. Hence, the natural map

$$H_c^1(M)\to H_{(2)}^1(M)$$
is surjective.  The injectivity follows, for example, from \cite[Th\'eor\`eme 3.3]{C7}.

\cqfd

\begin{Lem}\label{asym_cone}

Let $n\geq2$, and $X$ be a compact manifold of dimension $n-1$. Denote by $p$ the vertex of the cone $\mathcal{C}(X)$. Let $\lambda_1=\lambda_1(X)>0$ the first non-zero eigenvalue of the Laplacian on $X$, and let

$$q^*=\frac{n}{\frac{n}{2}+\sqrt{\left(\frac{n-2}{2}\right)^2+\lambda_1}}.$$
Let $f$ be a smooth function on $\mathcal{C}(X)$ such that $f$ is harmonic outside a compact set $K$, $d f\in L^2$. Assume that

$$\lim_{r\to\infty}f(r,x)=0,$$
uniformly in $x\in X$, and that, for some $r>0$ such that $K\subset B(p,r)$,

$$\int_{\partial B(p,r)}\frac{\partial f}{\partial \nu}=0.$$
Then, $df\in L^q$, for every $q\in [1,2]$ such that $q>q^*$.

\end{Lem}

\begin{proof}

The Laplacian on $\mathcal{C}(X)$ writes

$$\Delta=-\frac{\partial^2}{\partial r^2}-\frac{n-1}{r}\frac{\partial}{\partial r}+\frac{1}{r^2}\Delta_X.$$
Denote by $\{\lambda_k\}_{k\in \mathbb{N}}$ the spectrum of the Laplacian on $X$, and denote by $\{\varphi_k\}_{k\in\mathbb{N}}$ an associated complete orthonormal family of eigenfunctions. So,

$$\Delta_X\varphi_k=\lambda_k\varphi_k.$$
One has $\lambda_0=0$ and $\varphi_0=const.$ One writes a Fourier-type expansion for the function $f$ in $x$:

$$f(r,x)=\sum_{k=0}^\infty a_k(r)\varphi_k(x),$$
and for every $r\in (0,\infty)$,

$$\int_X|f(r,x)|^2\,dx=\sum_{k=0}^\infty |a_k(r)|^2<+\infty.$$
The harmonicity of $f$ outside a compact set is equivalent to the fact that the function $a_k$ is solution of the following ODE for $r$ big enough:

$$\left(-\frac{d^2}{d r^2}-\frac{n-1}{r}\frac{d}{d r}+\frac{\lambda_k}{r^2}\right)a=0.$$
This is an Euler equation, and a basis of solutions is given by the functions $r^{\alpha^\pm_k}$, with

$$\alpha_k^\pm=-\frac{n-2}{2}\pm \sqrt{\left(\frac{n-2}{2}\right)^2+\lambda_k}.$$
Note that for $k\geq1$, $\alpha_k^+>0$ and $\alpha_k^-<0$. Let us now look at

$$df=\frac{\partial f}{\partial r}dr+d_xf.$$
Since, for $k\neq l$,

$$\begin{array}{rcl}
\int_X(\nabla \varphi_k,\nabla \varphi_l)&=& \int_X \varphi_k\Delta_X\varphi_l\\\\
&=&\lambda_k\int_X\varphi_k\varphi_l\\\\
&=&0,
\end{array}$$
and since $(dr,d_x\varphi_k)=0$ for all $k$, one has

$$\begin{array}{rcl}
\int_X|df|^2\,dx&=&\sum_{k=0}^\infty \int_X|\frac{\partial a_k}{\partial r}\varphi_kdr+a_kd_x\varphi_k|^2\\\\
&=& \sum_{k=0}^\infty \int_X|\frac{\partial a_k}{\partial r}\varphi_k|^2+|a_kd_x\varphi_k|^2\\\\
&=& \sum_{k=1}^\infty \left(\frac{\partial a_k}{\partial r}\right)^2+r^{-2}\lambda_k|a_k|^2.
\end{array}$$
Since $df\in L^2$, one must have

$$||a_kd_x\varphi_k||_2^2<\infty.$$
One has

$$\begin{array}{rcl}
||a_kd_x\varphi_k||_2^2&=&\int_0^\infty a_k^2(r)\left(\int_X r^{-2} ||d_x\varphi_k||_X^2\right)\,r^{n-1}dr\\\\
&=& \int_0^\infty a_k^2(r) r^{n-3}\,dr.
\end{array}$$
But 

$$a_k(r)=\mu_kr^{\alpha_k^+}+\gamma_k r^{\alpha_k^-}$$
with $\alpha_k^+>0$ and $\alpha_k^-<0$ for $k\geq 1$. Since $n\geq 2$, in order that the above integral be finite, one must have $\mu_k=0$ for all $k\geq !$, i.e. $a_k(r)$ is proportional to $r^{\alpha_k^-}$. For $k=0$, one has $\alpha_0^+=0$ and $\alpha_0^-=-n+2$. Since $\varphi_0$ is constant, the fact that $||a_kd_x\varphi_k||_2^2<\infty$ does not bring any restriction on $a_0$. However,

$$\int_X |f(r,x)|^2dx=\sum_{k=0}^\infty |a_k(r)|^2,$$
and one sees that if $a_0(r)=\mu_0+\gamma_0r^{-n+2}$, then as $r\to\infty$,

$$\mu_0\lesssim \int_X |f(r,x)|^2dx.$$
By assumption,

$$\lim_{r\to\infty}\int_X |f(r,x)|^2dx=0,$$
therefore $\mu_0=0$, and 

$$a_0(r)=\gamma_0r^{-n+2}.$$
Also, since for all $k\geq 1$,

$$\int_X\varphi_k=const.\int_X\varphi_k\varphi_0=0,$$
one has

$$\begin{array}{rcl}
\int_{\partial B(p,r)}\frac{\partial f}{\partial\nu} &=&\sum_{k=0}^\infty a_k'(r) r^{n-1} \int_X \varphi_k\\\\
&=&  a_0'(r)r^{n-1}\int_X\varphi_0\\\\
&=& a_0'(r)r^{n-1}\mathrm{Vol}(X).
\end{array}$$
One concludes that since $f$ has zero flux on $\partial B(p,r)$,

$$\gamma_0=0,$$
hence

$$a_0\equiv0.$$
Therefore,

$$\int_X |f(r,x)|^2\,dx=\sum_{k=1}^\infty |\gamma_k|^2 r^{2\alpha_k^-},$$
and since 

$$\sum_{k=1}^\infty |\gamma_k|^2<+\infty$$
(by letting $r=1$ in the above formula), one sees that, as $r\to \infty$,

$$\int_X |f(r,x)|^2\,dx\lesssim |\gamma_1|^2 r^{2\alpha_1^-}.$$
Taking into account that 

$$\frac{\partial a_k}{\partial r}=\frac{1}{r}a_k(r),$$
we also obtain that

$$\int_X|df|^2\,dx=r^{-2}\sum_{k=1}^\infty |\gamma_k|^2r^{2\alpha_k^-}(1+\lambda_k). $$
Since the sum must converge for, say, $r=1$, one has

$$\int_X|df|^2\,dx\lesssim |\gamma_1|^2\lambda_1 r^{2\alpha_1^--2}.$$
Now, by H\"{o}lder, for $q<2$,

$$\begin{array}{rcl}
||df||_q^q&\leq& C+\int_1^\infty \left(\int_X|df|^q\,dx\right)\,r^{n-1}dr\\\\
&\leq & C+\mathrm{Vol(X)}^{1-\frac{q}{2}}\int_1^\infty \left(\int_X|df|^2\,dx\right)^{q/2}\,r^{n-1}dr\\\\
&\leq & C+C\int_1^\infty  r^{q(\alpha_1^--1)+n-1}\,dr
\end{array}$$
This is finite if and only if $q(\alpha_1^--1)+n<0$, which is easily seen to be equivalent to $q>q_*$. This concludes the proof.

\end{proof}

\noindent{\em Proof of Theorem \ref{main-cone}:} Since $M$ has non-negative Ricci curvature outside a compact set, the condition \eqref{condik} is satisfied. Assume that $M\simeq_\infty \mathcal{C}(X)$ with $X=S^{n-1}$ (resp. $X\neq S^{n-1}$); according to Theorem \ref{gradient} and Corollary \ref{unif_bdd}, it is enough to prove that \eqref{Ker} for $\mathcal{L}=\vec{\Delta}$, the Hodge Laplacian on $1$-forms, holds for all $q\in (1,2)$ (resp. for $q=1$). By Theorem \ref{coho-comp}, $M$ satisfies the assumption of Proposition \ref{harmo}. Let $\lambda_1=\lambda_1(X)$ be the first eigenvalue of the scalar Laplacian on $X$. According to Proposition \ref{harmo} and Lemma \ref{asym_cone}, \eqref{Ker} holds for every $q\geq 1$ such that 

$$q>q^*=\frac{n}{\frac{n}{2}+\sqrt{\left(\frac{n-2}{2}\right)^2+\lambda_1}}.$$
According to the Lichnerowicz-Obata theorem (see e.g. \cite{Gal}), the fact that $\mathrm{Ric}_X\geq (n-2)\bar{g}$ implies that

$$\lambda_1(X)\geq \lambda_1(S^{n-1})=n-1,$$
with equality if and only if $X$ is isometric to the Euclidean unit sphere $S^{n-1}$. A straightforward computation shows that if $\lambda_1(X)\geq n-1$, then

$$q^*\leq 1,$$
with equality if and only if $\lambda_1(X)=n-1$. Thus, if $X=S^{n-1}$ (resp., $X$ is not isometric to $S^{n-1}$) then \eqref{Ker} for $\mathcal{L}=\vec{\Delta}$, the Hodge Laplacian on $1$-forms, holds for all $q\in (1,2)$ (resp. for $q=1$). This concludes the proof.

\cqfd

\begin{center}{\bf Acknowledgments} \end{center}
It is a pleasure to thank G. Carron for several interesting discussions on the matter of this article, as well as A. Sikora for pointing out a mistake in a previous version of this article and helping fixing it.

\end{document}